\documentclass[11pt]{amsart}
\usepackage{indentfirst,enumerate,cite,amssymb,amsfonts,amsmath,amsthm,mathrsfs,dsfont}
\usepackage{xcolor}
\usepackage[colorlinks=true,linkcolor=blue,citecolor=blue]{hyperref}
\usepackage{setspace}
\usepackage{url}
\usepackage{mathtools}
\usepackage{setspace}
\usepackage[shortlabels]{enumitem}
\usepackage{comment}

\usepackage{geometry}
\numberwithin{equation}{section}
\theoremstyle{plain}

\newtheorem{theorem}{Theorem}
\newtheorem{corollary}{Corollary}
\newtheorem{prop}{Proposition}

\newtheorem{lemma}{Lemma}
\newtheorem{obs}{Remark}
\newtheorem{definition}{Definition}
\newtheorem{exemp}{Example}

\numberwithin{theorem}{section} 
\numberwithin{equation}{section}
\numberwithin{corollary}{section} 
\numberwithin{prop}{section} 
\numberwithin{lemma}{section}
\numberwithin{obs}{section} 
\numberwithin{definition}{section} 
\numberwithin{exemp}{section} 

\DeclareMathOperator{\S3}{\mathnormal{\mathbb{S}^{3s}}}
\DeclareMathOperator{\Lt}{\mathnormal{\prescript{t}{}{L}}}
\DeclareMathOperator{\0}{{{\overline{0}}}}
\DeclareMathOperator{\supp}{\text{supp}}
\DeclareMathOperator{\T}{\mathbb{T}}
\DeclareMathOperator{\G}{\mathnormal{\mathbb{T}^{r+1}\!\times\mathbb{S}^{3s}}}
\DeclareMathOperator{\R}{\mathbb{R}}
\DeclareMathOperator{\Z}{\mathbb{Z}}
\DeclareMathOperator{\N}{\mathbb{N}}
\DeclareMathOperator{\spn}{\text{span}}

\title[Global solvability and hypoellipticity on tori and spheres]{Global solvability and hypoellipticity \\ for evolution operators on tori and spheres}

\author[A. Kirilov]{Alexandre Kirilov}
\address{
	Universidade Federal do Paran\'{a}, 
	Departamento de Matem\'{a}tica,
	C.P.19096, CEP 81531-990, Curitiba, Brazil
}
\email{akirilov@ufpr.br}

\author[A. Kowacs]{Andr\'e Pedroso Kowacs}
\address{
	Universidade Federal do Paran\'{a},
	Programa de P\'os-Gradua\c c\~ao de Matem\'{a}tica,
	C.P.19096, CEP 81531-990, Curitiba, Brazil
}
\email{andrekowacs@gmail.com}

\author[W. de Moraes]{Wagner Augusto Almeida de Moraes}
\address{
	Universidade Federal do Paran\'{a}, 
	Departamento de Matem\'{a}tica,
	C.P.19096, CEP 81531-990, Curitiba, Brazil
}
\email{wagnermoraes@ufpr.br}

\thanks{This study was financed in part by the Coordenação de Aperfeiçoamento de Pessoal de Nível Superior - Brasil (CAPES) - Finance Code 001. The first and third authors were supported in part by CNPq (grants 316850/2021-7 and 423458/2021-3).}
	


\date{\today}

\begin{document}

\begin{abstract}
In this paper, we investigate global properties of a class of evolution differential operators defined on a product of tori and spheres. We present a comprehensive characterization of global solvability and hypoellipticity, providing necessary and sufficient conditions that involve Diophantine conditions and the connectedness of sublevel sets associated with the coefficients of the operator. Furthermore, we recover well-known results from existing literature and introduce novel contributions.
\end{abstract}

\maketitle
\begin{singlespace}
	\tableofcontents
\end{singlespace}

\section{Introduction}

In this article, our investigation revolves around the existence and regularity of solutions for evolution operators defined on a product of tori and $3-$spheres. We recall that an operator $L$ is called globally hypoelliptic if the smoothness of $Lu$ implies the smoothness of $u$. Also, an operator $L$ is called globally solvable if the equation $Lu=f$ admits a smooth solution whenever the function $f$ belongs to a suitable subset of smooth functions. 

Extensive research has been dedicated to exploring these global properties for operators defined on the torus, as evidenced by works such as \cite{BCP2004_mc, Avil2020_jmaa, Avil2023_mn, AM2021_ampa, AGKM2019_jfaa, DGY1996_auf, HP1998_mc, HPC2012_mn}. Additionally, these properties have also been investigated on compact Lie groups, as in \cite{Arau2019_agag, AFR2022_pams, Mora2022_ampa}, and some results have been obtained on compact manifolds, as in \cite{AFR2022_jam, BMZ2012_cpde, AC2022_jfa, AGK2018_jam, AK2019_jst}.

The specific scenario where the operator $L$ is a vector field defined on a torus deserves special attention. As conjectured by S. Greenfield and N. Wallach, see \cite{Forni08_cont-math}, if a smooth closed manifold $M$ admits a globally hypoelliptic vector field $L$, then $M$ is diffeomorphic to a torus, and $L$ can be conjugated to a constant vector field that satisfies a Diophantine condition.  Consequently, the investigation of global hypoellipticity for vector fields on closed manifolds primarily focuses on the torus. Some references in this context include \cite{CC2000_cpde, DGY2002_pems, GW1972_pams, GW1973_tams, Houn1979_tams, Petr2011_tams}.  It is worth noting that the literature contains numerous other references on this topic beyond the ones mentioned here.

We investigate the global hypoellipticity and solvability of a first-order evolution operator, specifically a zero-order perturbation of a complex vector field of tubular type, defined on the product of tori and $3-$spheres $\T^{r+1}\times(\mathbb{S}^3)^s$, where $r$ and $s$ are non-negative integers.

Our results siond the characterization of global solvability on the torus $\T^{r+1}$ obtained in \cite{BDG2017_jfaa, BDGK2015_jpdo}. Additionally, we extend some results from \cite{KMR2021_jfa} and \cite{KMP2021_jde}, where the authors characterized some global properties on $\T^1 \times \mathbb{S}^3$.

To develop our results, we primarily employed Fourier Analysis separately with respect to the variables of the torus and spheres. To achieve this, we extended the techniques presented in the references \cite{KMR2020_bsm, RT2007_fourier_su2, RT2007_fourier_torus, RT2010_book, RT2013_imrn} to a Cartesian product involving multiple copies of the one-dimensional torus and the three-dimensional sphere.

The operators studied in this article are of the form:
\begin{align}
	L = \partial_t+\sum_{j=1}^rc_j(t)\partial_{x_j}+\sum_{k=1}^s id_k(t)\partial_{0,k}+q. \label{main_operator}
\end{align}
Here, $q$ represents a complex constant. The smooth functions $c_j$ and $d_k$ are defined on $\T^1$, while the derivatives $\partial_{t} = \partial/\partial t$ and $\partial_{x_j} = \partial/\partial x_j$ are defined on distinct copies of $\T^1$, for $1\leq j\leq r$. Additionally, the neutral operators $\partial_{0,k}$ are defined on separate copies of the three-dimensional sphere $\mathbb{S}^3$, for $1\leq k\leq s$.

Let us write 
\begin{align*}
   c_j(t)&=a_j(t)+ib_j(t) \in C^\infty(\T^1), \mbox{ for } 1\leqslant j \leqslant r   \mbox{ and, } \\
   d_k(t)&=e_k(t)+if_k(t) \in C^\infty(\T^1),  \mbox{ for } 1\leqslant k \leqslant s, 
\end{align*}
and consider their averages
\begin{align*}
	a_{j0}&=\displaystyle \frac{1}{2\pi}\int_0^{2\pi}a_j(t)dt, \quad  b_{j0}=\frac{1}{2\pi}\int_0^{2\pi}b_j(t)dt, \quad c_{j0}=a_{j0}+ib_{j0},\\
	e_{k0}&=\displaystyle \frac{1}{2\pi}\int_0^{2\pi}e_k(t)dt, \quad  f_{k0}=\frac{1}{2\pi}\int_0^{2\pi}f_k(t)dt, \quad d_{k0}=e_{k0}+if_{k0}.
\end{align*}

We also write $c_0=a_0+ib_0$ and $d_0=e_0+if_0$, where
\begin{align*}
	c(t) = a(t)+ib(t)& = (a_1(t)+ib_1(t),\ldots,a_r(t)+ib_r(t)), \\
	d(t) = e(t)+if(t)& = (e_1(t)+id_1(t),\ldots,e_s(t)+id_s(t)).
\end{align*}

Finally, we consider the following constant coefficient differential operator, associated to $L$  $$L_0=\partial_t+\sum_{j=1}^rc_{j0}\partial_{x_j}+\sum_{k=1}^sd_{k0}i\partial_{0,k}+q,$$
 which has the symbol
\begin{align*}
		\sigma_{L_0}(\tau,\xi,\alpha)& =  i\tau+\sum_{j=1}^r c_{j0}i\xi_j +\sum_{k=1}^s d_{k0}i\alpha_{k}+q \\ 
		& = i\left(\tau+ \langle c_0,\xi\rangle + \langle d_0,\alpha\rangle-iq\right),
\end{align*}
where $\tau \in\mathbb{Z}$,  $\xi=(\xi_1,\xi_2,\ldots,\xi_r)\in\mathbb{Z}^{r}$, $\ell=(\ell_1,\ell_2,\ldots,\ell_s)\in\frac{1}{2}\N_0^s$, and the vector
$\alpha=(\alpha_1,\alpha_2,\ldots,\alpha_s)$ satisfies $-\ell\leq \alpha\leq \ell$, and $ \,\ell-\alpha\in\N_0^s$. 

We denote the set of zeros of $\sigma_{L_0}$ by
\begin{equation*}
		\mathcal{Z}_{L_0} =  \left\{(\tau,\xi,\ell) \in \Z^{r+1}\times\mbox{$\frac{1}{2}$} \N_0^s; \sigma_{L_0}(\tau,\xi,\alpha)= 0,\,\text{for some} -\ell\leq\alpha\leq \ell,\,\ell-\alpha\in\N_0^s\right\}.	
\end{equation*}

We are now equipped with all the necessary elements to provide a comprehensive characterization of the global solvability and global hypoellipticity of the operator $L$ .
	
\begin{theorem}[Global solvability]\label{teogsmain}
    The operator $L$ is globally solvable if and only if one of the following conditions holds:
    \begin{enumerate}[i.]
        \item $(b,f)$ is constant  and  $L_0$ is globally solvable.
        
        \item $\dim \spn \left\{b_1,\dots,b_r,f_1,\dots,f_s\right\}=1$,  none of the functions $b_j$ and $f_k$ change sign, and $L_0$ is globally solvable.

        \item $(b,f)\not\equiv 0$, $(b_0,f_0)=0$, $(a_0,e_0)\in\Z^r \times 2\Z^s$, $q\in i\Z$ and the sublevel sets
        $$\Omega_m^{\xi,\alpha}=\left\{t\in\T^1; \int_0^{t}\!\Big(\sum_{j=1}^r b_j(t)\xi_j+\sum_{k=1}^s d_{k}(t)\alpha_{k} \Big) \!\mathop{dt}<m\right\}$$
        are connected, for every  $m\in\R,\, \xi\in \Z^r$, and $\alpha\in \tfrac{1}{2}\Z^s$ .
    \end{enumerate}
    \end{theorem}

\begin{theorem}[Global hypoellipticity]\label{teoghmain}
     $L$ is globally hypoelliptic if and only one of the following conditions holds:
     \begin{enumerate}[i.]
         \item $(b,f)$ is constant, $L_0$ is globally solvable, and the set $\mathcal{Z}_{L_0}$ is finite;

        \item $\dim \spn \left\{b_1,\dots,b_r,f_1,\dots,f_s\right\}=1$, none of the functions $b_j$ and $f_k$ change sign, and $L_0$ is globally hypoelliptic.
     \end{enumerate}
\end{theorem}

It is worth noting that in both theorems, the global properties of the operator $L$ are intimately connected to those of $L_0$, which are determined by the following Diophantine condition on its coefficients.

\begin{definition}[Diophantine condition]
	We say that the operator $L_0$ satisfies the condition (DC) if there are $M,N>0$ such that:
	\begin{equation}\label{DC_condition_introduction}
			\left|\sigma_{L_0}(\tau,\xi,\alpha) \right|\geq M(|\tau|+|\xi|+|\ell|)^{-N} \tag{DC}
		\end{equation}
for every $(\tau,\xi)\in\mathbb{Z}^{r+1}$, $\ell\in\frac{1}{2}\N_0^s$, $-\ell\leq \alpha\leq \ell$, and $\ell-\alpha\in\N_0^s$, such that neither the left hand side of the inequality is zero, nor $(\tau,\xi,\ell)$ is the zero vector.
\end{definition}

\begin{theorem}[Constant coefficient case] \
\begin{enumerate}[i.]
	\item $L_0$ is globally solvable if, and only if it satisfies the Diophantine condition \eqref{DC_condition_introduction};
	\item  $L_0$ is globally hypoelliptic if, and only if it satisfies the Diophantine condition \eqref{DC_condition_introduction} and  $ \mathcal{Z}_{L_0}$ is finite.
\end{enumerate}
\end{theorem}

This paper is structured as follows: on Section $2$ we recall some classical results about Fourier analysis on the torus and on the sphere, extending them to the product $\G$, as well as we fix notation that will be used throughout the text. We dedicate Section $3$ for the proof of the  constant coefficient cases, in Proposition \ref{propctegs} and \ref{propghcte}. The proof for the main results states above is split into several propositions in Section $4$. Finally, in the Appendix A, we provide a set of crucial and technical results that are fundamental for the proofs.

\section{Fourier Analysis on the product of tori and spheres}

Consider the compact Lie groups $\T^{1}=\R/(2\pi\Z)$ and $\mathbb{S}^3$, and let $\widehat{\T^1}$ and $\widehat{\mathbb{S}^3}$ be their unitary duals, which consist of the sets of equivalence classes of unitary irreducible continuous representations. In this context, two representations are considered equivalent if there exists an invertible intertwining map between them.

It is well known that $\widehat{\T^1}\sim \Z$, where each $\xi\in\Z$ corresponds to the one-dimensional representation $t\mapsto e^{it\xi}\in\mathbb{U}(1)$. Similarly, we have $\widehat{\mathbb{S}^3}\sim\frac{1}{2}\N_0$, where each $\ell\in\frac{1}{2}\N_0$ corresponds to a representation given by a $2\ell+1$ dimensional matrix-valued function $y\mapsto\mathfrak{t}^\ell(y)$. The coefficient functions of this matrix-valued function are denoted as $y\mapsto\mathfrak{t}_{mn}^\ell(y)$, where the indices $m$ and $n$ range from $-\ell$ to $+\ell$ with a step size of $1$. In other words, $m$ and $n$ take values in the set $\{-\ell,-\ell+1,\dots,+\ell-1,+\ell\}$.

We define the Fourier coefficients of $f\in L^2(\T^1)$ and $g\in L^2(\mathbb{S}^3)$ as follows:
\begin{align*}
	\widehat{f}(\xi)& \doteq\frac{1}{2\pi}\int_{0}^{2\pi}f(t){e^{-it\xi}}\mathop{dt}, \ \mbox{ and } \ \widehat{g}(\ell)_{mn}\doteq \int_{\mathbb{S}^3} g(y)\overline{\mathfrak{t}^\ell_{nm}(y)} \mathop{dy},
\end{align*}
for every $\xi\in\Z$ and $\ell\in\frac{1}{2}\N_0$, with $-\ell\leq m,n\leq \ell$.

According to Peter-Weyl's Theorem, we have the following representations:
\begin{align*}
	f(t)&=\sum_{\xi\in\Z}\widehat{f}(\xi)e^{it\xi}, \ \mbox{and } \ g(y) = \!\! \sum_{\ell\in\frac{1}{2}\N_0}(2\ell+1) \!\!\!\! \sum_{-\ell\leq n,m\leq \ell} \!\! \widehat{g}(\ell)_{mn}\mathfrak{t}^\ell_{nm}(y),
\end{align*}
where the series converge almost everywhere and in their respective $L^2$ norms. Moreover, Plancherel's identity provides the following identities for every $f\in L^2(\T^1)$, $g\in L^2(\mathbb{S}^3)$
\begin{equation*}
\|f\|_{L^2(\T^1)}^2 =  2\pi\sum_{\xi\in\Z}|\widehat{f}(\xi)|^2, \ \mbox{ and } \
\|g\|_{L^2(\mathbb{S}^3)}^2 = \!\! \sum_{\ell\in\frac{1}{2}\N_0} \!\! (2\ell+1) \!\! \sum_{\substack{-\ell\leq m,n\leq \ell\\
\ell-m, \ell-n\in\N_0}} \!\!\!\! |\widehat{g}(\ell)_{mn}|^2.
\end{equation*}

As usual, these definitions and results can also be extended to the distributions spaces $\mathscr{D}'(\T^1)$ and $\mathscr{D}'(\mathbb{S}^3)$. 

For a function $f\in L^2(\G)$, we can apply the previous definitions iteratively to each coordinate of $(t,x,y)=(t,x_1,\ldots,x_r,y_1,\ldots,y_s)\in\G$, and define the following Fourier coefficient:
\begin{equation}\label{eqtotalfourier}
	\widehat{f}(\tau,\xi,\ell)_{\alpha\beta} \doteq\frac{1}{(2\pi)^{r+1}} \int_{\T^1}\int_{\T^{r}}\int_{\S3}f(t,x,y)e^{-i\langle (\tau,\xi), (t,x)\rangle} \overline{\mathfrak{t}^{\ell}_{\beta\alpha}(y)}\mathop{dy}\mathop{dx}\mathop{dt},
\end{equation}
for each $(\tau,\xi)\in\Z^{r+1}$, $\ell\in\frac{1}{2}\N_0^s$, and $-\ell\leq\alpha,\beta\leq \ell$, such that $\ell-\alpha,\ell-\beta\in\N_0^s$. Where
\begin{equation*}
	\mathfrak{t}^{\ell}_{\alpha\beta}(y) = \mathfrak{t}^{\ell_1}_{\alpha_1\beta_1} (y_1)\dots\mathfrak{t}^{\ell_s}_{\alpha_s\beta_s}(y_s) \mbox{ for } \alpha,\beta\in\mbox{$\frac{1}{2}$}\Z^s.
\end{equation*}

Setting $d_\ell \doteq \prod_{j=1}^s(2\ell_j+1)$, for each  $\ell\in\frac{1}{2}\N_0^s$,we can apply Peter-Weyl's theorem to obtain the following representation:
$$f(t,x,y) = \sum_{(\tau,\xi)\in\Z^{r+1}}\sum_{\ell\in\frac{1}{2}\N_0^s}d_l\sum_{\substack{-\ell\leq \alpha,\beta\leq \ell \\ \ell-\alpha,\ell-\beta\in\N_0^s}}e^{i\langle (\tau,\xi),(t,x)\rangle}\widehat{f}(\tau,\xi,\ell)_{\alpha\beta}\mathfrak{t}^{\ell}_{\beta\alpha}(y)$$
where the series converge almost everywhere and in $L^2(\G)$, and 
$$\langle (\tau,\xi),(t,x)\rangle \doteq \tau t+\xi_1x_1+\dots+\xi_rx_r.$$

We can also define the partial Fourier coefficients of a function $f\in L^2(\G)$ with respect to the $(x,y)$ variables as follows:
\begin{equation}\label{eqfourierpartial}
    \widehat{f}(t,\xi,\ell)_{\alpha\beta} \doteq \frac{1}{(2\pi)^r} \int_{\T^{r}}\int_{\S3}f(t,x,y)e^{-i\langle \xi,x\rangle}\overline{\mathfrak{t}^{\ell}_{\beta\alpha}(y)}\mathop{dy}\mathop{dx},
\end{equation}
so that we have:
$$f(t,x,y) = \sum_{\xi\in\Z^{r}}\sum_{\ell\in\frac{1}{2}\N_0^s}d_\ell\sum_{\substack{-\ell\leq \alpha,\beta\leq \ell \\ \ell - \alpha, \ell-\beta \in\N_0^s}} e^{i\langle\xi,x\rangle} \widehat{f}(t,\xi,\ell)_{\alpha\beta} \mathfrak{t}^{\ell}_{\beta\alpha}(y)$$
where the series converge almost everywhere and in $L^2(\G)$, as usual.

Furthermore, if $f\in C^\infty(\G)$, then the partial Fourier coefficient $\widehat{f} (t,\xi,\ell)_{\alpha\beta}\in C^\infty(\T^1)$. These definitions and results can also be naturally extended to the set of distributions $\mathscr{D}'(\G)$.

\begin{lemma}\label{lemmaformula0} \ 
\begin{enumerate}[i.]
	\item If $f,g\in L^2(\mathbb{S}^3)$, then
    \begin{align}\label{eqlemmacoef}
        \int_{\mathbb{S}^3} f(x)g(x)\mathop{dx} = \sum_{\ell\in\frac{1}{2}\N_0}(2\ell+1)\sum_{m,n=-\ell}^l\widehat{f}(\ell)_{mn}\widehat{g}(\ell)_{-m-n}(-1)^{n-m},
    \end{align}
    \item if $f,g\in L^2(\mathbb{S}^{3s})$, then:
       \begin{align}\label{eqlemmacoef2}
        \int_{\mathbb{S}^{3s}} f(x)g(x)\mathop{dx} = \sum_{\ell\in\frac{1}{2}\N_0^s}d_\ell\sum_{-\ell\leq \alpha,\beta\leq \ell} \widehat{f}(\ell)_{\alpha\beta} \widehat{g}(\ell)_{-\alpha-\beta}(-1)^{\Sigma (\alpha_j-\beta_j)}.
    \end{align} 
\end{enumerate}
\end{lemma}
\begin{proof}
	To prove {\it i.}, we can observe that the integral on the left-hand side of expression \eqref{lemmaformula0} corresponds to the $L^2$ inner product between the functions $f$ and $\overline{g}$. By applying Plancherel's identity for $L^2(\mathbb{S}^3)$, the polarization identities for inner products, and considering the property $\overline{\mathfrak{t}^\ell_{mn}} = \mathfrak{t}^\ell_{-m-n} (-1)^{n-m}$, we can derive the first formula.
	
    This last fact can be deduced from the definition itself (see \cite{RT2010_book}), as evaluating $\mathfrak{t}^\ell_{mn}$ with Euler angles yields, as follows: 
    \begin{align*}
        \overline{\mathfrak{t}^\ell_{mn}(\omega(\phi,\theta,\psi))}  &= e^{i(m\phi+n\psi)}\overline{P_{mn}^\ell(\cos(\theta))}\\
        &= e^{-i((-m)\phi+(-n)\psi)}(-1)^{n-m}P_{mn}^\ell(\cos(\theta))\\
        &=(-1)^{n-m}e^{-i((-m)\phi+(-n)\psi)}P_{-m-n}^\ell(\cos(\theta))\\
        &=(-1)^{n-m}\mathfrak{t}^\ell_{-m-n}(\omega(\phi,\theta,\psi))
    \end{align*}
    where 
    $$P_{mn}^\ell(x)=c_{mn}^\ell\frac{(1-x)^{(n-m)/2}}{(1+x)^{(m+n)/2}}\left(\frac{\mathop{d}}{\mathop{dx}}\right)^{\ell-m}[(1-x)^{\ell-n}(1+x)^{\ell+n}]$$
    with
    $$c^\ell_{mn} = 2^{-\ell}\frac{(-1)^{\ell-n}i^{n-m}}{\sqrt{(\ell-n)!(\ell+n)!}}\sqrt{\frac{(\ell+m)!}{(\ell-m)!}}.$$
    
    Finally, identity \eqref{eqlemmacoef2} is obtained by applying \eqref{lemmaformula0} iteratively.
\end{proof}

\begin{lemma}\label{lemmaformula}
Let $f,g\in L^2(\G)$. Then:
    \begin{align}
        \int_{\G} \!\!\! fg = (2\pi)^r\sum_{\xi\in\Z^r}\sum_{\ell\in\frac{1}{2}\N_0^s}d_\ell\sum_{-\ell\leq \alpha,\beta \leq \ell} \int_0^{2\pi} \widehat{f}(t,\xi,\ell)_{\alpha\beta} \widehat{g}(t,-\xi,\ell)_{-\alpha-\beta}(-1)^{\Sigma(\alpha_j-\beta_j)}\mathop{dt}.
    \end{align} 
\end{lemma}

\begin{proof}
By applying Plancherel's identity partially to the last $r+s$ coordinates, we can use Fubini's theorem to interchange the order of integration and summation, resulting in:
    
\begin{align*}
    \int_{\G} |f|^2&=\int_0^{2\pi}(2\pi)^r \sum_{\xi\in\Z^r} \sum_{\ell\in\frac{1}{2}\N_0^s}d_\ell\sum_{-\ell\leq \alpha,\beta\leq \ell}| \widehat{f}(t,\xi,\ell)_{\alpha\beta}|^2\mathop{dt}=\\
     &=(2\pi)^r\sum_{\xi\in\Z^r}\sum_{\ell\in\frac{1}{2}\N_0^s}d_\ell\sum_{-\ell\leq \alpha,\beta\leq \ell}\int_0^{2\pi}|\widehat{f}(t,\xi,\ell)_{\alpha\beta}|^2\mathop{dt}.
\end{align*}
The result follows in a similar manner as in Lemma \ref{lemmaformula0}.
\end{proof}

It is well known that smooth functions and distributions on $\T^1$ and $\mathbb{S}^3$ can be characterized by the rate of growth or decay of their Fourier coefficients with respect to the quantities $\langle\xi_j\rangle = \sqrt{1+\xi_j^2}$ for the torus and $\langle\ell_m\rangle = \sqrt{1+(1+\ell_m)\ell_m}$ for the sphere(see \cite{KMR2020_bsm} and \cite{RT2010_book} for further details).

In particular, using the equivalence established in Lemma \ref{lemmapseudo} (page \pageref{lemmapseudo}), we can define 
$$
|\xi|+|\ell|=|\xi_1| + \dots+ |\xi_r| + \ell_1 + \dots + \ell_s,
$$
for $\xi = (\xi_1,\dots,\xi_r)\in\Z^r$, $\ell=(\ell_1,\dots,\ell_s)\in\frac{1}{2}\N_0^s$.  Moreover, we denote $\0=(0,\dots,0), $

Based on these definitions, we can establish the following characterizations for functions and distributions defined on $\G$. The proofs of these characterizations follow the same principles and techniques as those presented in \cite{KMR2020_bsm}.

\begin{prop}[Total Fourier series]\label{lemmadecaysmooth} \
	\begin{enumerate}[i.]
		\item $f\in C^\infty(\G)$ if and only if, for every $N>0$, $\exists M_{N}>0$ such that:
		\begin{align*}
				|\widehat{f}(\tau,\xi,\ell)_{\alpha\beta}|\leq M_{N}\left(|\tau|+|\xi|+|\ell|\right)^{-N} 
		\end{align*}
		for every $(\tau,\xi)\in\Z^{r+1},$ $\ell \in\frac{1}{2}\N_0^s,$ $-\ell\leq\alpha,\beta\leq \ell,$  $\ell-\alpha,\,\ell-\beta\in\N_0^s,$ and  $(\tau,\xi,\ell)\neq \0.$ \\

		\item $u\in \mathscr{D}'(\G)$ if and only if, there exists $M,N>0$, such that:
		\begin{align*}
			|\widehat{u}(\tau,\xi,\ell)_{\alpha\beta}|\leq M(|\tau|+|\xi|+|\ell|)^{N} 
		\end{align*}
		for every $(\tau,\xi)\in\Z^{r+1},$ $\ell \in\frac{1}{2}\N_0^s,$ $-\ell\leq\alpha,\beta\leq \ell,$  $\ell-\alpha,\,\ell-\beta\in\N_0^s,$ and  $(\tau,\xi,\ell)\neq \0.$
	\end{enumerate}
\end{prop}

\begin{prop}[Partial Fourier series]\label{lemmadecaysmoothpartial}  \
	\begin{enumerate}[i.]
	\item $f\in C^\infty(\G)$ if and only if, for every $m\in\mathbb{N}_0$ and every $N>0$, $\exists M_{mN}>0$ such that:
	\begin{align*}
		&|\partial_t^m\widehat{f}(t,\xi,\ell)_{\alpha\beta}|\leq M_{mN}\left(|\xi|+|\ell|\right)^{-N}  
	\end{align*}
	for every $t\in\mathbb{T}^1,$ $\xi\in\Z^{r}$ $\ell \in\frac{1}{2}\N_0^s,$ $-\ell\leq\alpha,\beta\leq \ell,$  $\ell-\alpha, \ell-\beta\in\N_0^s,$ and $(\xi,\ell)\neq \0.$\\
	
	\item $u\in \mathscr{D}'(\G)$ if and only if, there exists $N\in\mathbb{N}$ and $M>0$,  such that:
	\begin{align*}
		&|\langle\widehat{u}(t,\xi,\ell)_{\alpha\beta},\psi(t)\rangle|\leq Mp_N(\psi)(|\xi|+|\ell|)^{N}
	\end{align*}
	for every $\xi\in\Z^{r},$ $\ell \in\frac{1}{2}\N_0^s,$ $-\ell\leq\alpha,\beta\leq \ell,$  $\ell-\alpha,\ell-\beta\in\N_0^s,$ $(\xi,\ell)\neq \0,$ and $\psi\in C^\infty(\T^1).$
	Where $p_N(\psi) = \sum_{m=0}^N\sup_{t\in\T^1}|\partial^m\psi(t)|$.
	\end{enumerate}
\end{prop}

Next, let us consider the operators $\partial_t$ and $\partial_{x_j}$, which are defined on $\T^1$ and $\T^{r}$, respectively. For $u\in \mathscr{D}'(\T^1)$ and $v\in \mathscr{D}'(\T^{r})$, we have:
$$
\widehat{\partial_t u}(\tau)=i\tau\widehat{u}(\tau) \mbox{ \ and \ } \widehat{\partial_{x_j} v}(\xi)=i\xi_j\widehat{u}(\xi).
$$
for every $\tau \in \Z,$ and $\xi\in\Z^r.$

Furthermore, if we consider the coordinate system in $\mathbb{S}^3$ given by the standard Euler angles $(\phi,\theta,\psi)$, as described in \cite{RT2010_book}, the vector field $D_3=i\partial_0$, which is represented by $-{\partial}/{\partial\psi}$ in these coordinates, has the following symbol:
$$
  \sigma_{D_{3}}(\ell)_{mn} = im\delta_{mn}.
$$
In other words, for any $g\in \mathscr{D}'(\mathbb{S}^3)$, we have:
$$
  \widehat{D_3g}(\ell)_{mn}=im\widehat{g}(\ell)_{mn},
$$
for every $\ell\in\frac{1}{2}\N_0$ and $-\ell\leq n,m\leq \ell$.

\begin{obs}
	It is worth mentioning that any left invariant vector field on $\mathbb{S}^3$ can be conjugated with $D_3$ by rigid motions of the sphere, therefore they share the same global properties.  Hence, the results presented in this article can be extended to first-order differential operators involving any other left-invariant vector field on $\mathbb{S}^3$ in lieu of $D_3$.
\end{obs}

\subsection{Global properties} \

In $\mathbb{S}^3$, considering the system of coordinates given by Euler angles in each copy of the $3$-sphere, denoted by $(\phi_1,\theta_1,\psi_1,\dots,\phi_s,\theta_s,\psi_s)$, we can define the operators $D_{3,m}$ as $-{\partial}/{\partial\psi_m}$ for $m=1,\dots,s$.

With this notation, we consider the operator
$$L = \partial_t+\sum_{j=1}^rc_j\partial_{x_j}+\sum_{k=1}^sd_jD_{3,k}+q$$
defined on $\G$, where $c_j,d_k,q\in\mathbb{C}$, for $j=1,\dots,r$ and $k=1,\dots,s$. 

Based on the above discussions, for $u\in \mathscr{D}'(\G)$, we can express the Fourier coefficients of $Lu$ as follows:
\begin{align*}
    \widehat{Lu}(\tau,\xi,\ell)_{\alpha\beta}& = i(\tau+c_1\xi_1+\dots+c_r\xi_r+d_1\alpha_1+\dots+d_s\alpha_s-iq)\widehat{u}(\tau,\xi,\ell)_{\alpha\beta}\\
    & = i(\tau+\langle c,\xi\rangle+\langle d,\alpha\rangle-iq) \widehat{u} (\tau,\xi,\ell)_{\alpha\beta} \\ & = \sigma_L(\tau,\xi,\alpha) \widehat{u} (\tau,\xi,\alpha)_{\alpha\beta} 
\end{align*}
for each $(\tau,\xi)\in\Z^{r+1}, \ell\in\frac{1}{2}\N_0^s,-\ell\leq\alpha,\beta\leq \ell$, $\ell-\alpha,\ell-\beta\in\N_0^s$. Here we defined:
$$\sigma_L(\tau,\xi,\alpha)\doteq i(\tau+\langle c,\xi\rangle+\langle d,\alpha\rangle-iq) $$

Likewise, for $u\in \mathscr{D}'(\G)$, we can express the partial Fourier coefficients of $Lu$ as follows:
\begin{align*}
    \widehat{Lu}(t,\xi,\ell)_{\alpha\beta} = i(\partial_t+\langle c,\xi\rangle+\langle d,\alpha\rangle-iq)\widehat{u}(t,\xi,\ell)_{\alpha\beta}
\end{align*}
for each $t\in\T^1$, $\xi\in\Z^r, \ell\in\frac{1}{2}\N_0^s,\,-\ell\leq\alpha,\beta\leq \ell$, $\ell-\alpha,\ell-\beta\in\N_0^s$.

Next, we will provide precise definitions for the global properties that are the focus of this article.

\begin{definition}
	Let $L$ be an operator in $\G$ and $\Lt$ its transpose. We say that:
	\begin{enumerate}[i.]
		\item $L$ is globally solvable, if for every $f\in C^\infty(\G)\cap(\ker \Lt)^0$, there is a function $u\in C^\infty(\G)$ such that $Lu=f$.
		\item $L$ is globally hypoelliptic if the conditions $u\in \mathscr{D}'(\G)$ and $Lu\in C^\infty(\G)$ imply $u\in C^{\infty}(\G)$.
	\end{enumerate}
\end{definition}

\begin{obs}\label{obsanulker}
	By Lemma \ref{lemmaanulker}, if $L$ is a constant coefficient first-order differential operator, and
	$$
	\widehat{Lu}(\tau,\xi,\ell)_{\alpha\beta} = \sigma_{L}(\tau,\xi,\alpha)\widehat{u}(\tau,\xi,\ell)_{\alpha\beta},
	$$
	then 
	$$
	(\ker \Lt)^0= \{f\in \mathscr{D}'(\G); \mbox{ such that } \widehat{f}(\tau,\xi,\ell)_{\alpha\beta} = 0 \text{ if } \sigma_{L}(\tau,\xi,\alpha)=0\}.
	$$
	Therefore, $L$ is globally solvable  for every $f\in C^\infty(\G)$ such that $\widehat{f}(\tau,\xi,\ell)_{\alpha\beta}=0$ whenever $\sigma_{L}(\tau,\xi,\alpha)=0$, there exists $u\in C^\infty(\G),$ such that $Lu=f$.
\end{obs}

The lemma presented below is an adaptation of the classical method of Hormander (Lemma 6.1.2 in \cite{Horm93_book}) to prove the nonsolvability of operators, which provides a necessary condition for global solvability by involving the transposed operator. Notably, the proof of this result relies only on functional analysis arguments in Frechet spaces, specifically tailored for compact Lie groups.

\begin{lemma}\label{Hormander_Method}
Let $L$ be a globally solvable differential operator on a compact Lie group $G$. Then, there exist $m\in\N$ and $C>0$ such that for every $f\in (\ker \Lt)^0$ and $v\in C^\infty(G)$, the following inequality holds:
	$$\left|\int_Gf(x)v(x)\mathop{dx}\right|\leq C\left(\sum_{|\alpha|\leq m}\sup_{G}|\partial^\alpha f|\right)\left(\sum_{|\alpha|\leq m}\sup_{G}|\partial^\alpha (\Lt v)|\right).$$
The sum $\sum_{|\alpha|\leq m}$ is taken over all left-invariant differential operators on $G$ of order at most $m$, or equivalently, it is taken over all linear combinations of at most $m$ compositions of elements of a basis $B$ for $Lie(G)$.
\end{lemma}
\begin{proof}
Note that the map $f\mapsto \sum_{|\alpha|\leq m}\sup_{x\in G}|\partial^\alpha f(x)|$, for $m\in\N$, defines a sequence of separating seminorms (and actually norms) on $C^\infty(G)$, making it a locally convex metrizable Fréchet space. 

The same holds for the map $v\mapsto \sum_{|\alpha|\leq m}\sup_{x\in G}|\partial^\alpha (\Lt v)(x)|$ on $C^\infty(G)/\ker \Lt$ (which is well-defined). As $(\ker \Lt)^0$ is a closed subspace of $C^\infty(G)$, it follows that the induced topology makes it a Fréchet space as well.	
	
Consider the map defined on the product space $(\ker \Lt)^0\times C^\infty(G)/\ker \Lt$ given by $$(f,v)\mapsto \int_G f(x)v(x)\mathop{dx}.$$ 
This map is well-defined since $f\in(\ker \Lt)^0$. We claim that it is a bilinear continuous map. To see this, it suffices to prove that the linear maps $f\mapsto \int_G fv$ and $v\mapsto \int_G fv$ are continuous for each $f$ and $v$. This follows from the fact that for a fixed $v\in C^\infty(G)/\ker \Lt$, the map $f\mapsto \int_G fv$ is continuous in $(\ker \Lt)^0$ since
	$$\left|\int_G f(x)v(x)dx\right|\leq \sup_{x\in G}|f(x)|\sup_{x\in G}|v(x)|\leq \|v\|_\infty \sum_{|\alpha|\leq 1}\sup_{x\in G}|\partial^\alpha f(x)|.$$

Furthermore, for any $f\in (\ker \Lt)^0$, by global solvability of $L$, there is $u\in C^\infty(G)$ such that $Lu=f$. Consequently, the linear map defined as
	$$C^\infty(G)/\ker \Lt\ni v\mapsto\int_G fv$$
	satisfies
\begin{align*}
	\left|\int_G fv\right| &= |\langle v,Lu\rangle|=|\langle \Lt v,u\rangle|
	=\left|\int_G(\Lt v(x))u(x)\mathop{dx}\right|\\
	&\leq \left|\int_G u(x)\mathop{dx}\right|\sup_{x\in G}|\Lt v(x)| \leq \left|\int_G u(x)\mathop{dx}\right|\sum_{|\alpha|\leq 1}\sup_{x\in G}|\partial^\alpha (\Lt v)(x)|.
\end{align*}

As a result, this map is continuous under this topology, confirming the claim. Considering that this is a bilinear map defined on the product of locally convex topological vector spaces equipped with the topology given by families of seminorms, we can conclude that there exists a constant $C>0$ such that
	$$\left|\int_Gf(x)v(x)\mathop{dx}\right|\leq C\left(\sum_{|\alpha|\leq m}\sup_{G}|\partial^\alpha f|\right)\left(\sum_{|\alpha|\leq m}\sup_{G}|\partial^\alpha (\Lt v)|\right)$$
holds for all $f\in (\ker \Lt)^0$ and $v\in C^\infty(G)/\ker \Lt$. 

Furthermore, since the inequality also trivially holds for elements in $\ker \Lt$, the lemma follows. In particular, if $v\in C^\infty(G)$ can be expressed as $u=v+w$, where $w\in\ker \Lt$, then 
	\begin{align*}
		\left|\int_Gfv\right| &=\left|\int_G fu\right|
		\leq C\left(\sum_{|\alpha|\leq m}\sup_{G}|\partial^\alpha f|\right)\left(\sum_{|\alpha|\leq m}\sup_{G}|\partial^\alpha (\Lt u)|\right)\\
		&\leq C\left(\sum_{|\alpha|\leq m}\sup_{G}|\partial^\alpha f|\right)\left(\sum_{|\alpha|\leq m}\sup_{G}|\partial^\alpha (\Lt v)|\right).
	\end{align*}
\end{proof}

\section{Constant-coefficient Operators}

We will now establish the necessary and sufficient conditions for the global hypoellipticity and global solvability of constant coefficient first-order differential operators on $\G$. Consider the operator:
\begin{equation}\label{eqoplcte}
	L = \partial_t+\sum_{j=1}^rc_j\partial_{x_j}+\sum_{k=1}^s d_kD_{3,k}+q.
\end{equation}
where $c_1,\dots,c_r$, $d_1,\dots,d_s$, and $q$ are complex constants. 

By Plancherel's Theorem, if $u,g\in \mathscr{D}'(\G)$ such that $Lu=g$, then taking the partial Fourier transform on each variable separately yields the following equations:
\begin{equation}\label{eqcte}
	\sigma_{L}(\tau,\xi,\alpha)\widehat{u}(\tau,\xi,\ell)_{\alpha\beta} = \widehat{g}(\tau,\xi,\ell)_{\alpha\beta},  
\end{equation}
where $(\tau,\xi)=(\tau,\xi_1,\ldots,\xi_r)\in\mathbb{Z}^{r+1}$, $\ell=(\ell_1,\ldots,\ell_s) \in \frac{1}{2}\N_0^s$, and $-\ell\leq \alpha,\beta\leq \ell$ satisfy $\ell-\alpha,\ell-\beta\in\N_0^s$ and 
$$\sigma_L(\tau,\xi,\alpha)\doteq i(\tau+\langle c,\xi\rangle+\langle d,\alpha\rangle-iq). $$

To state our first result of this section, let us recall the Diophantine condition stated in the introduction of this article.
\begin{definition}
	We say that the constant coefficient operator $L$ satisfies the condition \emph{(DC)} if there are $M,N>0$ such that:
	\begin{equation}\label{DC_condition_const_coef_section}
		\left|\sigma_{L}(\tau,\xi,\alpha) \right|\geq M(|\tau|+|\xi|+|\ell|)^{-N} \tag{DC}
	\end{equation}
	for every $(\tau,\xi)\in\mathbb{Z}^{r+1}$, $\ell\in\frac{1}{2}\N_0^s$, $-\ell\leq \alpha\leq \ell$, and $\ell-\alpha\in\N_0^s$, such that neither the left-hand side of the inequality is zero, nor $(\tau,\xi,\ell)$ is the zero vector.
\end{definition}

\begin{prop}\label{propctegs}
	The operator $L$ is globally solvable if and only if it satisfies condition \emph{(DC)}.
\end{prop}

\begin{proof}
	Suppose $L$ does not satisfy condition (DC). Then there exists $(\tau(n),\xi(n))\in\mathbb{Z}^{r+1}$, $\ell(n)\in\frac{1}{2}\N_0^s$, $\alpha(n)\in\frac{1}{2}\mathbb{Z}^s$ such that $-\ell(n)\leq \alpha(n)\leq \ell(n)$ and 
	\begin{equation}\label{eqineqdiof1}
		0<|\tau(n)+\langle c,\xi(n) \rangle +\langle d,\alpha(n)\rangle -iq |<(|\tau(n)|+|\xi(n)|+|\ell(n)|)^{-n},
	\end{equation}
	for every $n\in \N$. 
	
Let us define 
\begin{equation*}
	\widehat{g}(\tau,\xi,\ell)_{\alpha\beta} =
	\begin{cases*}
		|\tau(n)+\langle c,\xi(n) \rangle +\langle d,\alpha(n)\rangle -iq|^{1/2}, &   \mbox{ if  $\alpha=\beta=\alpha(n)$ for some  $n\in\N$}; \\
		 0,  & \mbox{ otherwise.}
	\end{cases*}
\end{equation*}	
Note that, according to Proposition \ref{lemmadecaysmooth}, these coefficients define a function $g\in C^{\infty}(\T^{r+1}\times \mathbb{S}^{3s})$. By Remark \ref{obsanulker} and Lemma \ref{lemmaanulker}, we have $g\in(\ker \Lt)^0$ since $i(\tau(n)+\langle c,\xi(n) \rangle +\langle d,\alpha(n)\rangle -iq)\neq 0$ for every $n\in\N$.

We can also verify directly that $L$ satisfies the following equations:
	\begin{align}
		\widehat{Lu}(\tau,\xi,\ell)_{\alpha\beta}& =i(\tau+\langle c,\xi \rangle +\langle d,\alpha\rangle -iq)\widehat{u}(\tau,\xi,\ell)_{\alpha\beta}, \label{eqproof1}  \\
		\widehat{\Lt u}(\tau,\xi,\ell)_{\alpha\beta}& =i(-\tau+\langle c,-\xi \rangle +\langle d,-\alpha\rangle -iq)\widehat{u}(\tau,\xi,\ell)_{\alpha\beta}  \nonumber.
	\end{align}

	Now, if $Lu=g$, then their Fourier coefficients must coincide. By  \eqref{eqproof1}, we can conclude that for every $n\in\N$,
	\begin{equation*}	
		\widehat{u}(\tau(n),\xi(n),\ell(n))_{\alpha(n)\alpha(n)} = \frac{|\tau(n)+\langle c,\xi(n) \rangle +\langle d,\alpha(n)\rangle -iq|^{1/2}}{i(\tau(n)+\langle c,\xi(n) \rangle +\langle d,\alpha(n)\rangle -iq)}.
	\end{equation*}	

	From \ref{eqineqdiof1} we have
		\begin{equation*}	
			|\widehat{u} (\tau(n),\xi(n),\ell(n))_{\alpha(n)\alpha(n)}| \geq (|\tau(n)| + |\xi(n)| + |\ell(n)|)^{n/2}.
		\end{equation*}	
	
	Therefore, $u\not\in C^{\infty}(\G)$ by Proposition \ref{lemmadecaysmooth}, in fact, $u\not\in \mathscr{D}'(\G)$. Hence, $L$ is not globally solvable.

	On the other hand, suppose that $L$ satisfies condition (DC), and let $g\in (\ker \Lt)^0$. We can define $u\in \mathscr{D}'(\G)$ as follows:
	$$
	\widehat{u}(k,\xi,\ell)_{\alpha\beta} = \frac{\widehat{g}(\tau,\xi,\ell)_{\alpha\beta}}{i(\tau+\langle c,\xi \rangle +\langle d,\alpha\rangle -iq)}
	$$
	whenever the denominator is not zero, and $\widehat{u}(k,\xi,\ell)_{\alpha\beta} = 0$ otherwise. 
	
	Then we have 
	$$
	|\widehat{u}(k,\xi,\ell)_{\alpha\beta}|\leq |\widehat{g}(\tau,\xi,\ell)_{\alpha\beta}| \frac{1}{M} (|\tau|+|\xi|+|\ell|)^{N}
	$$
	 for every $(\tau,\xi,\ell)$ (except possibly at $\0$). By Proposition \ref{lemmadecaysmooth}, if $g\in C^\infty(\G)$, then $u\in C^\infty(\G)$ (or if $g\in \mathscr{D}'(\G)$, then $u\in \mathscr{D}'(\G)$).
	
	Moreover, by equation \eqref{eqcte} and Plancherel's identity, we have $Lu=g$. Therefore, $L$ is globally solvable.
\end{proof}

In order to establish the result concerning the global hypoellipticity of the operator $L$, we need to impose an additional condition, namely, that the number of zeros of the symbol of the operator is finite.

\begin{prop}\label{propghcte}
The operator $L$ is globally hypoelliptic if and only if it satisfies the Diophantine Condition (DC) and the set $\mathcal{Z}_L$ of zeros of its symbol, defined as
	\begin{equation}\label{eqsetN}
	\mathcal{Z}_L =  \left\{(\tau,\xi,\ell) \in \Z^{r+1}\times\mbox{$\frac{1}{2}$} \N_0^s; \sigma_{L}(\tau,\xi,\alpha)= 0,\,\text{for some} -\ell\leq\alpha\leq \ell,\,\ell-\alpha\in\N_0^s\right\}.	
	\end{equation}
	is finite.
\end{prop}

\begin{proof}
We will establish the sufficiency using a proof by contradiction. Let us assume that the set $\mathcal{Z}_L$ contains an infinite number of elements. Consider the distribution $u \in \mathscr{D}'(\G)$, whose Fourier coefficients are defined as follows:
	\begin{equation*}
		\widehat{u}(\tau,\xi,\ell)_{\alpha\beta}=
		\begin{cases}
			1, & \mbox{ if } (\tau,\xi,\ell) \in \mathcal{Z}_L, \\
			0, & \mbox{ otherwise}.
		\end{cases}
	\end{equation*}

By Proposition \ref{lemmadecaysmooth}, it is clear that $u$ belongs to $\mathscr{D}'(\G)$ but not to $C^\infty(\G)$. Furthermore, it is evident that $\widehat{Lu}(\tau,\xi,\ell)_{\alpha\beta}=0$ for all indices. Applying Plancherel's theorem, we obtain $Lu\equiv 0$, which implies that $Lu$ is a smooth function. Consequently, $L$ cannot be globally hypoelliptic.

Continuing with the proof of sufficiency, we now consider the case where $L$ does not satisfy condition (DC). Similarly to the proof of Proposition \ref{propctegs}, we can find sequences $(\tau(n),\xi(n))\in\mathbb{Z}^{r+1}$, $\ell(n)\in \frac{1}{2}\N_0^s$, $\alpha(n)\in\frac{1}{2}\mathbb{Z}^s$ satisfying $-\ell(n)\leq \alpha(n)\leq \ell(n)$ and
	$$0<|\tau(n)+\langle c,\xi(n) \rangle +\langle d,\alpha(n)\rangle -iq |<(|\tau(n)|+|\xi(n)|+|\ell(n)|)^{-n}$$
	for every $n\in \N$. 

Let us define 
\begin{equation*}
	\widehat{g}(\tau,\xi,\ell)_{\alpha\beta} =
	\begin{cases*}
		i(\tau(n)+\langle c,\xi(n) \rangle +\langle d,\alpha(n)\rangle -iq), &   \mbox{ if  $\alpha=\beta=\alpha(n)$ for some  $n\in\N$}; \\
		0  & \mbox{ otherwise.}
	\end{cases*}
\end{equation*}	
	
  Based on the inequality above and Proposition \ref{lemmadecaysmooth}, the coefficients defined in the previous paragraph give rise to a function $g\in C^\infty(\G)$. However, since $\widehat{u} (\tau(n),\xi(n),\ell(n)) = 1$ for every $n\in\N$ and $0$ otherwise, we can conclude that $u\in \mathscr{D}'(\G)\backslash C^\infty(\G)$. Furthermore, it is clear that $Lu=g$, indicating that $L$ is not globally hypoelliptic.
	
Now, let us proceed to the proof of necessity. Given that the set $\mathcal{Z}_L$ is finite and the operator $L$ satisfies the Diophantine condition, we can establish the following inequality for all $u\in \mathscr{D}'(\G)$:
	$$
		|\widehat{Lu}(\tau,\xi,\ell)_{\alpha\beta}| \geq {M}(|\tau|+|\xi|+|\ell|)^{-N}|\widehat{u}(\tau,\xi,\ell)_{\alpha\beta}|
	$$
for all $(\tau,\xi,\ell)\neq \0$, except for at most finitely many $(\tau,\xi,\ell)$ in the set $\mathcal{Z}_L$. By applying Proposition \ref{lemmadecaysmooth}, we have $u\not\in C^\infty(\G) \implies Lu\not\in C^\infty(\G)$. Therefore, we can conclude that $L$ is globally hypoelliptic.
\end{proof}

\begin{corollary}\label{coroghcte}
	If $s\neq0$, then the operator $L$ is globally hypoelliptic if and only if it is globally solvable and the set $\mathcal{Z}_L$ is empty.
\end{corollary}
Indeed, it suffices to observe that if $(\tau,\xi,\ell)\in \mathcal{Z}_L$, then $(\tau,\xi,\ell+n)\in \mathcal{Z}_L$ for every $n\in\N_0$. Therefore, the proof of this corollary follows directly from Proposition \ref{propghcte}.

\begin{corollary}
	If $s\neq0$, then $L$ is globally hypoelliptic if and only if $L$ is an automorphism of $C^\infty(\G)$.
\end{corollary}
\begin{proof}
	First, as $L$ is a differential operator, it is clear that $L(C^\infty(\G))\subset C^\infty(\G)$.
	If $L$ is an automorphism, it is evident that it is globally hypoelliptic by injectivity. Conversely, suppose that $L$ is globally hypoelliptic. According to Corollary \ref{coroghcte}, we know that $\mathcal{Z}_L$ is empty. Therefore, if $Lu=Lv$, it follows that $\widehat{Lu}(\tau,\xi,\ell)_{\alpha\beta}=\widehat{Lv}_(\tau,\xi,\ell)_{\alpha\beta}$, which can be expressed as
	$$
	\sigma_L(\tau,\xi,\alpha) \widehat{u}(\tau,\xi,\ell)_{\alpha\beta} = \sigma_L(\tau,\xi,\alpha)\widehat{v}(\tau,\xi,\alpha)_{\alpha\beta}.
	$$
	
	This implies $\widehat{u}(\tau,\xi,\ell)_{\alpha\beta}=\widehat{v}(\tau,\xi,\ell)_{\alpha\beta}$ for every coefficient. By applying Plancherel's Theorem, we obtain $u=v$, demonstrating the injectivity of $L$.
	
	Furthermore, as in the proof of Proposition \ref{propctegs}, Lemma \ref{lemmaanulker} implies that 
	$$(\ker \Lt)^0\cap C^\infty(\G)= C^{\infty}(\G).$$ 
	
	By Proposition \ref{propctegs}, we conclude that $L$ is globally solvable. Hence, it is also surjective, leading to the conclusion that $L$ is an automorphism.
\end{proof}

In the case where $s=0$, it should be noted that the operator $L$ is defined only on the torus. Consequently, the set $\mathcal{Z}_L$ may not be empty, and Corollary 3.2 does not hold in this scenario. For further details, refer to \cite{BDGK2015_jpdo,BDG2017_jfaa}.

Now, we will showcase some examples demonstrating the applications of our results, including the recovery of known results from the literature. The torus case, extensively studied in \cite{BDG2017_jfaa}, provides the first example.


\begin{exemp}\label{exemptoruscte}
{\em
If $s = 0$, we have $L = \partial_t + \sum_{j=1}c_j\partial_{x_j} + q$. According to Proposition \ref{propctegs}, if $c = a + ib$, where $a,b\in \mathbb{Q}^{r}$, then $L$ is globally solvable. 

To prove this result, we need to show that the the Diophantine condition \eqref{DC_condition_const_coef_section} is satisfied. Observe that if $\sigma_L(\tau,\xi) = i(\tau + \langle c,\xi\rangle - iq)\neq 0$, then either $\Im(\sigma_L(\tau,\xi))\neq 0$ or $ \Re(\sigma_L(\tau,\xi))\neq 0.$

Let us begin by analyzing the case where $\Re(\sigma_L(\tau,\xi))\neq 0$. For each $j\in {1,2,\ldots,r}$, let us write $b_j={b_{j1}}/{b_{j2}}$ with $b_{j1}, b_{j2}\in\mathbb{Z}$ and $b_{j2}>0$.
\begin{enumerate}[{\it i.}]
	\item If $\Re(q)={m}/{n}$, with $m,n \in \mathbb{Z}$ and $n>0$ we have
	\begin{align*}
		|\sigma_L(\tau,\xi)|&\geq   |\Re(\sigma_L(\tau,\xi))|  = |b_1\xi_1 + \dots + b_r\xi_r -\Re(q)|\\
		&= \frac{1}{nb_{12}\dots b_{r2}}\big|(nb_{22}\dots b_{r2})b_{11}\xi_1+\dots + (nb_{12}\dots b_{(r-1)2})b_{r1}\xi_r - (b_{12}\dots b_{r2})m\big|\\
		&\geq \frac{1}{nb_{12}\dots b_{r2}}>0.
	\end{align*}
	This inequality holds because all numbers within the last modulus are integers, and we already established that $|\Re(\sigma_L(\tau,\xi))|>0$.\\

	\item If $\Re(q)\not\in\mathbb{Q}$, then $\lambda\Re(q)\not\in\Z$ for all $\lambda\in\Z\backslash\{0\}$, therefore
	\begin{align*}
		|\sigma_L(\tau,\xi)|&\geq   |\Re(\sigma_L(\tau,\xi))| = |b_1\xi_1+\dots+b_r\xi_r-\Re(q)|\\
		&= \frac{1}{b_{12}\dots b_{r2}}|(b_{11}b_{22}\dots b_{r2})\xi_1+\dots+ (b_{r1}b_{12}\dots b_{(r-1)2})\xi_r -(b_{12}\dots b_{r2})\Re(q)|\\
		&\geq \frac{\mbox{dist}\!\left((b_{12}\dots b_{r2})\Re(q),\Z\right)}{b_{12}\dots b_{r2}}>0.
	\end{align*}
\end{enumerate}

Similarly, we can apply the same type of estimate when $\Im(\sigma_L(\tau,\xi))\neq 0$ by utilizing the coefficients $a_j$ and $\Im(q)$. Consequently, we can establish the existence of $\varepsilon>0$ such that $|\sigma_L(\tau,\xi)|\geq \varepsilon$ whenever $\sigma_L(\tau,\xi)\neq 0$, thereby satisfying the condition (DC) trivially.

Now, if we assume that $(a,b)\not\in\mathbb{Q}^r\times \mathbb{Q}^r$, the satisfaction of condition (DC) by $L$ will depend on the type of approximation of the coefficients $a_j$ and $b_k$ by rational numbers.

We say that a number $\mu\in\mathbb{R}\backslash\mathbb{Q}$ is an irrational Liouville number if there exists a sequence $(p_n,j_n)\in\mathbb{Z}\times\N$ such that $j_n\to\infty$ and
$|\mu-p_n/j_n|<(j_n)^{-n}.$
In other words, any irrational Liouville number can be approximated rapidly by a sequence of rational numbers.

For instance, if any of the $a_{j_1}$ values is an irrational non-Liouville number, where $1\leq j_1\leq r$, and ${a_j}\in\mathbb{Q}$ for all $j\neq j_1$ with $1\leq j\leq r$, as well as ${\Im(q)}\in\mathbb{Q}$, then the operator $L$ also satisfies condition (DC) and, therefore, is globally solvable. On the contrary, if any of the $a_{j_1}$ values is a Liouville number, $b_0=\0$, and $\Im(q)\in\mathbb{Z}$, then $L$ is not globally solvable.
}
\end{exemp}

\begin{exemp}\label{exempliouvvect}
	{\em
Consider the operator $L = \partial_t + c\partial_{x} + dD_{3} + q$ defined on $\mathbb{T}^2\times \mathbb{S}^3$, where $c,d,q \in \mathbb{C}$. Similarly to the previous example, if $c,d \in \mathbb{Q}+i\mathbb{Q}$, then $L$ is globally solvable.

However, when the coefficients $c$ and $d$ are not in $\mathbb{Q}+i\mathbb{Q}$, the question of whether or not $L$ satisfies condition (DC) is determined by the extent to which these coefficients can be approximated by rational numbers.

We say that $(\alpha,\beta)\in \R^2\backslash\mathbb{Q}^2$ is a Liouville vector if there is a constant $C>0$ and a sequence $(p_n,q_n,j_n)\in\Z^2\times\N$ such that $j_n\geq 2$ and     
$$
   |\alpha-p_n/j_n|+|\beta-q_n/j_n|< C(j_n)^{-n}.
$$ 

Let us carefully analyze the following situation: $\Im(q), e \in \mathbb{Z}$, $(a, \frac{b}{f})$ is a Liouville vector, and $\frac{\Re(q)}{f} \in \mathbb{Z}$. We will now prove that, under these assumptions, the operator $L$ does not satisfy condition (DC) and therefore is not globally solvable.

Since $(a, \frac{b}{f})$ is a Liouville vector, there exist sequences of integers $(p_n)_n$ and $(q_n)_n$, an increasing sequence of natural numbers $(j_n)_n$, and a constant $C > 0$ such that the following inequality holds:
    $$
       \left|a-\frac{p_n}{j_n}\right|+\left|\frac{b}{f}-\frac{q_n}{j_n}\right|< \frac{C}{(j_n)^n}
    $$

Let us define $\ell_n = \left(\Re(q)/f - q_n\right) \in \mathbb{Z}$ and $\tau_n = -\left(p_n + e\ell_n - \Im(q)\right) \in \mathbb{Z}$, for $n \in \mathbb{N}$. We can now rewrite the left-hand side of the previous inequality as follows:
\begin{align*}
	\left|a-\frac{p_n}{j_n}\right|+\left|\frac{b}{f}-\frac{q_n}{j_n}\right|
	& = \frac{1}{j_n}\left|aj_n+\tau_n+ e\ell_n+ \Im(q)\right| +\frac{1}{j_n|f|}\left|bj_n+f\ell_n-\Re(q)\right|.
\end{align*}

Thus, for every $n\in \mathbb{N}$ we have
    \begin{align*}
         |\sigma_L (\tau_n,j_n,\ell_n)|
        &\leq |\tau_n+aj_n+e\ell_n+\Im(q)|+ |f| \frac{1}{|f|} |bj_n+f\ell_n-\Re(q)|\\
        &\leq  (1+|f|) \left(|aj_n-\tau_n+e\ell_n+\Im(q)|+ \frac{1}{|f|} |bj_n+f\ell_n-\Re(q)|\right)\\
        & = (1+|f|)  j_n \left( 	\left|a - \frac{p_n}{j_n} \right| + \left| \frac{b}{f} - \frac{q_n}{j_n} \right| \right) \\
        & <(1+|f|)  j_n \frac{C}{(j_n)^n} =  C'(j_n)^{-n+1}.
    \end{align*}

This shows that inequality \eqref{DC_condition_const_coef_section} cannot be satisfied and $L$ is not globally solvable.
}
\end{exemp}

\begin{exemp}\label{exemphypocte}
\emph{
If $q = 0$, then $L$ is not globally hypoelliptic. Indeed, in this case $\sigma_L(\tau,\xi,\alpha) = i(\tau + \langle c,\xi\rangle + \langle d,\alpha\rangle) = 0$ for $(\tau,\xi,\alpha) = (0,\0,\0)$, therefore $\mathcal{Z}_L$ is nonempty and $L$ is not globally hypoelliptic, see Corollary \ref{coroghcte}.
}
\end{exemp}

\begin{exemp}
\emph{
	In the case where $r = s = 1$, consider the scenario where $b/f$ is an irrational  non-Liouville number, $\Re(q) \in \mathbb{Q}$, with $f\in \mathbb{Q}$ chosen such that $mf \neq \Re(q)$ for all $m \in \frac{1}{2}\mathbb{Z}$. In other words, we have $\Re(q) = \frac{q_n}{q_d}$ and $f = \frac{f_n}{f_d}$, and $\frac{q_nf_d}{q_df_n} \not\in \frac{1}{2}\mathbb{Z}$. Under these conditions, the operator $L$ is globally hypoelliptic. A specific example that satisfies these conditions is given by
	$L=\partial_t+\sqrt{2}i\partial_{x}+iD_3+\frac{1}{4}.$
}
\end{exemp}

\section{Variable Coefficient Operators} 

In this section, we will present the proof of the two main results of this work concerning the global solvability and global hypoellipticity of operators in the form:
\begin{equation}\label{eqL}
	L = \partial_t+\sum_{j=1}^rc_j(t)\partial_{x_j}+\sum_{k=1}^sd_k(t)D_{3,k}+q,
\end{equation}
where $c_j(t)=a_j(t)+ib_j(t)$, $d_k(t) = e_k(t)+if_k(t)$ are smooth functions on $\T^1$, and $q\in\mathbb{C}$.

These results were previously announced in the introduction of this article as Theorem \ref{teogsmain} and Theorem \ref{teoghmain}. The proof of these results is divided into several propositions and lemmas, which collectively establish the main findings.  The notations that follow, and which will be extensively used in this section, were introduced mainly on pages \pageref{main_operator} and \pageref{teogsmain} for reference.

We will begin by showing that we can assume the real part of the coefficients of the operator $L$ to be constant. This means that  the global properties of $L$ depend only on the average value of the real part of its coefficients.

\begin{prop}\label{proprealpartcte}
   The operator $L$ is globally solvable (or globally hypoelliptic) if and only if the operator $\tilde{L}$ is globally solvable (or globally hypoelliptic), where
    \begin{equation}\label{eqL'}
        \tilde{L} = \partial_t+\sum_{j=1}^r(a_{j0}+ib_j(t))\partial_{x_j}+\sum_{k=1}^s(e_{k0}+if_{k}(t))D_{3,k}+q.
    \end{equation}
\end{prop}
\begin{proof}
	We will construct an automorphism $\Psi:C^\infty(\G)\to C^\infty(\G)$ such that $L\circ \Psi = \Psi\circ \tilde{L}$. To define $\Psi$, we introduce the following functions:
	
    For each $j=1,\dots,r$ and $k=1,\dots,s$, let
    $$A_j(t) = \int_0^ta_j(w)dw-a_{j0}t, \mbox{  and \ }  E_k(t) = \int_0^te_k(w)dw-e_{k0}t.$$
    
    Firstly, note that $A_j$ and $E_k$ are well-defined real-valued functions on $\T^1$. Now, given $u\in \mathscr{D}'(\G)$, we define $\Psi u\in \mathscr{D}'(\G)$ in terms of its partial Fourier coefficients by
    \begin{align*}
        \widehat{\Psi u}(t,\xi,\ell)_{\alpha\beta}
        & = e^{-i\left[\langle A(t),\xi\rangle+\langle E(t),\alpha\rangle \right]}\widehat{u}(t,\xi,\ell)_{\alpha\beta}
    \end{align*}
    for every $t\in \T^1$, $\xi\in\Z^r,$ and $\ell\in\frac{1}{2}\N_0^s,$ with $-\ell\leq \alpha,\beta\leq \ell$, where 
    $$
    \langle A(t),\xi\rangle= \sum_{j=1}^r A_j(t)\xi_j, \mbox{ \ and \ } 
    \langle E(t),\alpha\rangle = \sum_{k=1}^s E_k(t)\alpha_k.
    $$
    Observe that the mapping $\Psi:\mathscr{D}'(\G)\to \mathscr{D}'(\G)$ is well-defined. Indeed,
    \begin{align*}
            |\langle\widehat{\Psi u}(\cdot,\xi,\ell)_{\alpha\beta},\psi\rangle_{\T^1}  |&=|\langle\widehat{u}(\cdot,\xi,\ell)_{\alpha\beta},e^{-i\left[\langle A(\cdot),\xi\rangle+\langle E(\cdot),\alpha\rangle \right]}\psi\rangle_{\T^1}|\\
            &\leq K(|\xi|+|\ell|)^{N}p_N(e^{-i\left[\langle A(\cdot),\xi\rangle+\langle E(\cdot),\alpha\rangle \right]}\psi)
    \end{align*}
    and for each $m\in\N_0$, there is $K_m>0$ such that: \begin{equation}\label{ineqauto}
        \partial_t^m(e^{-i\left[\langle A(t),\xi\rangle+\langle E(t),\alpha\rangle \right]})|\leq K_m(|\xi|+|\ell|)^m
    \end{equation}
    for every $t\in\T^1$, and $(\xi,\ell)\neq\0$. Combining these results, we obtain:
 \begin{align*}
            |\langle\widehat{\Psi u}(\cdot,\xi,\ell)_{\alpha\beta},\psi\rangle_{\T^1}|\leq K(|\xi|+|\ell|)^{2N}p_N(\psi)
    \end{align*}
     for some $K,N>0$, since $u\in \mathscr{D}'(\G)$.

  Inequality \eqref{ineqauto} and Proposition \ref{lemmadecaysmoothpartial} also provide the following result: $u\in C^\infty(\G)$ if and only if $\Psi u\in C^\infty(\G)$. Moreover, $\Psi$ is a bijective mapping, as we can define its inverse $\Psi^{-1}: \mathscr{D}'(\G)\to \mathscr{D}'(\G)$ as follows:
    \begin{align*}
        \widehat{\Psi^{-1} u}(t,\xi,\ell)_{\alpha\beta} = e^{+i\left[\langle A(t),\xi\rangle+\langle E(t),\alpha\rangle \right]}\widehat{u}(t,\xi,\ell)_{\alpha\beta}.
    \end{align*} 
Note that
    \begin{align*}
        \partial_t \widehat{\Psi u}(t,\xi,\ell)_{\alpha\beta} &= e^{-i\left[\langle A(t),\xi\rangle+\langle E(t),\alpha\rangle \right]}\partial_t \widehat{u}(t,\xi,\ell)_{\alpha\beta}\\
        &=-i\left[\langle a(t)-a_0,\xi\rangle+\langle e(t)-e_0,\alpha\rangle\right]e^{-i\left[\langle A(t),\xi\rangle+\langle E(t),\alpha\rangle \right]} \widehat{u}(t,\xi,\ell)_{\alpha\beta}\\
        &=\widehat{\Psi\partial_t u}(t,\xi,\ell)_{\alpha\beta}-i\left[\langle a(t)-a_0,\xi\rangle+\langle e(t)-e_0,\alpha\rangle\right]\widehat{\Psi u}(t,\xi,\ell)_{\alpha\beta}.
    \end{align*}

From this, we can conclude that $L\circ \Psi=\Psi\circ \tilde{L}$ as claimed. Similarly, $\Lt\circ \Psi = \Psi\circ \prescript{t}{}{\tilde{L}}$.
Therefore, we have:
     \begin{align*}
         u\in\ker\tilde{\Lt} \iff \tilde{\Lt}u\equiv0\iff \Psi(\tilde{\Lt}u)\equiv0\iff\Lt\Psi u\equiv0\iff\Psi u\in\ker\Lt 
     \end{align*}

 Hence, $\Psi^{-1}(\ker \Lt) = \ker \prescript{t}{}{\tilde{L}}$.  
     
     Next, we claim that $\Psi^{-1}((\ker \Lt)^0) = (\ker \prescript{t}{}{\tilde{L}})^0$. To prove this, let $g\in (\ker \Lt)^0$ and take any $v\in\ker \prescript{t}{}{\tilde{L}}$. From what we have just established, we know that there exists $u\in\ker \Lt$ such that $v=\Psi^{-1}u$. By applying Lemma \ref{lemmaformula} to $u$, we obtain:     
      \begin{align*}
        \langle \Psi^{-1}g,\Psi^{-1}u\rangle = & (2\pi)^r\sum_{\xi\in\Z^r}\sum_{\ell\in\frac{1}{2}\N_0^s}d_\ell \!\!\! \sum_{-\ell\leq  \alpha,\beta \leq \ell} \int_0^{2\pi}e^{+i\left[\langle A(t),\xi\rangle+\langle E(t),\alpha\rangle \right]}\times\\
        & \times\widehat{g}(t,\xi,\ell)_{\alpha\beta}e^{+i\left[\langle A(t),-\xi\rangle+\langle E(t),-\alpha\rangle \right]}\widehat{u}(t,-\xi,\ell)_{(-\alpha)(-\beta)}(-1)^{\Sigma(\beta_j-\alpha_j)} \mathop{dt}\\
        =&(2\pi)^r\sum_{\xi\in\Z^r}\sum_{\ell\in\frac{1}{2}\N_0^s} d_\ell \!\!\! \sum_{-\ell\leq \alpha,\beta\leq \ell} \int_0^{2\pi} \widehat{g} (t,\xi,\ell)_{\alpha\beta} \widehat{u}(t,-\xi,\ell)_{(-\alpha)(-\beta)}(-1)^{\Sigma(\beta_j-\alpha_j)}\mathop{dt}\\
        =&\langle g,u\rangle \, = \, 0.
    \end{align*}
    Hence, we have shown that $\Psi^{-1}g\in(\ker \prescript{t}{}{\tilde{L}})^0$. This implies that $\Psi^{-1}((\ker \Lt)^0) \subseteq (\ker \prescript{t}{}{\tilde{L}})^0$. Furthermore, due to the bijectivity of $\Psi$, we can conclude that $\Psi^{-1}((\ker \Lt)^0) = (\ker \prescript{t}{}{\tilde{L}})^0$. Consequently, the global solvability of $L$ is equivalent to the global solvability of $\tilde{L}$.
    
	Similarly, $\tilde{L}u=g\in C^\infty(\G)$ if and only if $ \Psi g=\Psi \tilde{L}u = L\Psi u\in C^\infty(\G)$. Therefore, $L$ is globally hypoelliptic if and only if $\tilde{L}$ is globally hypoelliptic.
\end{proof}


\begin{corollary}
	Consider the following operators defined on $\G$:
	$$
	L = \partial_t+\sum a_j(t)\partial_{x_j} + \sum e_k(t)D_{3,k} + q \mbox{ \ and \ } L_0 = \partial_t + \sum a_{j0} \partial_{x_j} + \sum e_{k0}D_{3,k} + q,
	$$
	where $a_j, e_k \in C^\infty(\T^1)$ are real-valued functions, $a_{j0},e_{k0}\in\mathbb{R}$ are their average means, and $q\in\mathbb{C}$. Then
\begin{enumerate}[i.]
	\item $L$ is globally solvable $\iff$ $L_0$ is globally solvable;
	\item $L$ is globally hypoelliptic $\iff$ $L_0$ is globally hypoelliptic.
\end{enumerate}
\end{corollary}

	Note that the this establishes the first part of Theorems \ref{teogsmain} and \ref{teoghmain}. We now turn our attention to the case where some $b_j$ or $f_k$ coefficients are non-zero. In this situation, the global properties of the operator $L$ are closely related to the behavior of linear combinations involving the imaginary parts of this coefficients. Specifically, we consider whether these linear combinations change sign or not.

\begin{definition}\label{CScondition_definition}
    We will say the operator $L$, defined in \eqref{eqL},  satisfies condition (CS) if
    there exist non-zero vectors $\tilde{\xi}\in\Z^r$ and $\tilde{\alpha} \in \frac{1}{2}\Z^s$ such that the expression
    $$\langle c_0,\tilde{\xi}\rangle+\langle d_0,\tilde{\alpha}\rangle-iq\not\in\Z$$
    and the following real-valued smooth function changes sign
    $$
    \theta(t) \doteq \langle b(t),\tilde{\xi}\rangle+\langle f(t),\tilde{\alpha}\rangle, t\in\T^1.
    $$
\end{definition}

\begin{theorem}\label{teocs}
If there exists at least one $b_j \not\equiv 0$ or $f_k \not\equiv 0$, then the operator $L$ is globally solvable if and only if it does not satisfy condition (CS) and the following conditions hold:
 \begin{enumerate}
     \item For every $(\xi,\alpha,m) \in \Z^r \times \frac{1}{2}\Z^s \times \R$, the sublevel set  
     $$
     \Omega_m^{\xi,\alpha} = \left\{t\in\T^1;\int_0^t\langle b(\tau),\xi\rangle+\langle f(\tau),\alpha\rangle d\tau<m\right\}
     $$
     is connected.
     \item If $(b_0,f_0) \neq \0$, then the operator $L_0$ is globally solvable.
 \end{enumerate}
Furthermore, the operator $L$ is globally hypoelliptic if and only if it does not satisfy condition (CS) and $L_0$ is globally hypoelliptic.
\end{theorem}

Later on, we will see that this theorem, combined with the previous result, establishes Theorems \ref{teogsmain} and \ref{teoghmain}. The proof of the present theorem will be developed in the following two subsections. However, before diving into the proof, we would like to present the following corollary.

\begin{corollary}
If $s>0$, $L$ is globally hypoelliptic if and only if $L$ is an automorphism of $C^\infty(\G)$.
\end{corollary}
\begin{proof}
	If $L$ is an automorphism, then it is clear that $L$ is globally hypoelliptic due to its injectivity. On the other hand, suppose that $L$ is globally hypoelliptic. According to Theorems \ref{teogsmain} and \ref{teoghmain}, the operator $L$ is also globally solvable. Furthermore, since $L_0$ is globally hypoelliptic, it follows from Corollary \ref{coroghcte} and Lemma \ref{lemma0Nimpliesanul} that $(\ker \Lt)^0\cap C^\infty(\G)=C^\infty(\G)$. This implies that $L$ is surjective over the space of smooth functions.
	
	Now suppose that $Lu=g\in C^{\infty}(\G)$. This implies that $\widehat{Lu}(t,\xi,\ell)_{\alpha\beta} = \widehat{g}(t,\xi,\ell)_{\alpha\beta},$ 
	for every $t\in\T^1$, $\xi\in\Z^r$, $\ell\in\frac{1}{2}\N_0^s$, and $-\ell\leq \alpha,\beta\leq \ell$. Therefore, we obtain a differential equation on $\T^1$, as discussed in Remark \ref{obseqedo}. Remarkably, since $\langle c_0, \xi\rangle+\langle d_0,\alpha\rangle-iq\not\in\mathbb{Z}$ for all $(\xi,\alpha)\in\Z^r\times\frac{1}{2}\Z^s$, Lemma \ref{lemmaodesol} guarantees the existence of a unique solution $\widehat{u}(t,\xi,\ell)_{\alpha\beta}$ to this equation. Moreover, since this holds for every coefficient, the solution $\widehat{u}(t,\xi,\ell)_{\alpha\beta}$ uniquely determines the function $u$ that solves $Lu=g$. Since $L$ is globally hypoelliptic, it follows that $u\in C^\infty(\G)$. Therefore, $L$ is injective over the space of smooth functions, which implies that $L$ must be an automorphism.
\end{proof}

To facilitate the comprehension and organization of the proof of Theorem \ref{teocs}, we have divided it into several propositions and lemmas. Before delving into the technical details, we will introduce a set of useful lemmas. First, we establish a sufficient condition on the partial Fourier coefficients of a function $g$ that ensures its inclusion in $(\ker\Lt)^0$. This condition plays a crucial role in the subsequent construction of singular solutions.

To establish this result, we begin by recalling the operator:
$$
  L_0=\partial_t+\sum_{j=1}^rc_{j0}\partial_{x_j}+\sum_{k=1}^sd_{k0}D_{0,k}+q,
$$
which has the symbol symbol $\sigma_{L_0}(\tau,\xi,\alpha) = \tau+ \langle c_0,\xi\rangle + \langle d_0,\alpha\rangle-iq,$  where $(\tau,\xi) \in\mathbb{Z}^{r+1}$, $\ell \in \frac{1}{2}\N_0^s$, $-\ell\leq \alpha\leq \ell$, and $ \,\ell-\alpha\in\N_0^s$.  Additionally, we introduce the set:
$$ 
\mathcal{N}=\left\{(\xi,\alpha)\in\Z^r\times\frac{1}{2}\Z^s:  \sigma_{L_0}(0,\xi,\alpha)\in\mathbb{Z}\right\}.
$$ 

\begin{lemma}\label{lemma0Nimpliesanul}
	Let $g\in \mathscr{D}'(\G)$. If $\widehat{g}(t,\xi,\ell)_{\alpha\beta}\equiv0$ for all $(\xi,\alpha)\in \mathcal{N}$, then $g\in(\ker \Lt)^0$.
\end{lemma}

\begin{proof}
Let $g\in \mathscr{D}'(\G)$ be such that $\widehat{g}(t,\xi,\ell)_{\alpha\beta}\equiv 0$ for $(\xi,\alpha)\in \mathcal{N}$. Suppose $v\in\ker \Lt$. Since $-\Lt v =0$, we can take the partial Fourier transform with respect to the last $r+s$ variables, which gives:
\begin{align}\label{eqvker}
\widehat{-\Lt v}(t,-\xi,\ell)_{(-\alpha)(-\beta)} = & \  \partial_t\widehat{v}(t,-\xi,\ell)_{(-\alpha)(-\beta)}+\nonumber\\
& +i(\langle c(t),-\xi\rangle+\langle d(t),-\alpha\rangle+iq)\widehat{v}(t,-\xi,\ell)_{(-\alpha)(-\beta)}=0.
\end{align}

Now if $(\xi,\alpha)\not\in\mathcal{N}$, then 
\begin{align*}
\sigma_{L_0}(0,\xi,\alpha) =-(\langle c_0,-\xi\rangle+\langle d_0,-\alpha\rangle+iq)\not\in\mathbb{Z}.
\end{align*}

By applying Lemma \ref{lemmaodesol}, we can deduce from equation \eqref{eqvker} that $\widehat{v}(t,-\xi,\ell)_{(-\alpha)(-\beta)} \equiv 0$. Then, invoking Lemma \ref{lemmaformula}, we have

\begin{align*}
\langle g,v\rangle &= (2\pi)^r\sum_{\xi\in\Z^r}\sum_{\ell\in\frac{1}{2}\N_0^s}d_\ell \!\!\! \sum_{-\ell\leq\alpha,\beta\leq \ell} \int_0^{2\pi} \widehat{g}(t,\xi,\ell)_{\alpha\beta} \widehat{v}(t,-{\xi},\ell)_{(-\alpha)(-\beta)}(-1)^{\Sigma(\alpha_j-\beta_j)}\mathop{dt} =0
\end{align*}
since every term in the sum is zero. This holds for arbitrary $v\in\ker \Lt$. Therefore, we can conclude that $g\in(\ker \Lt)^0$.
\end{proof}

\begin{obs}\label{obseqedo}
{\em 
Notice that if $u,g\in \mathscr{D}'(\G)$ satisfy $Lu=g$, then $\widehat{Lu}(t,\xi,\ell)_{\alpha\beta} = \widehat{g}(t,\xi,\ell)_{\alpha\beta}$ for every $t\in\T^1$, $\xi\in\Z^r$, $\ell\in\frac{1}{2}\N_0^s$, $-\ell\leq\alpha,\beta\leq \ell$. This can be expressed as a system of independent differential equations on $\T^1$ as follows:
\begin{equation}\label{eqedo}
	\left[\partial_t + i(\langle c(t),\xi\rangle+\langle d(t),\alpha\rangle-iq)\right]\widehat{u}(t,\xi,\ell)_{\alpha\beta} = \widehat{g}(t,\xi,\ell)_{\alpha\beta}
\end{equation}
for all $(\xi,\alpha,\beta,\ell)$ satisfying the given conditions.
}
\end{obs}

The following lemma shows that being in $(\ker \Lt)^0$ is sufficient to guarantee the existence of a solution to the system of O.D.E.s given by \eqref{eqedo}.

\begin{lemma}\label{lemmaanulimpliessolution}
 If $g\in (\ker \Lt)^0\cap C^\infty(\G)$, then the system of differential equations \eqref{eqedo}
admits solution for every $\xi\in\Z^r,\,\ell\in\frac{1}{2}\N_0^s,\,-\ell\leq\alpha,\beta\leq \ell$.
\end{lemma}
\begin{proof}
Let $g\in (\ker \Lt)^0\cap C^\infty(\G)$. A solution $\widehat{u}$ for the system \eqref{eqedo} must satisfy, for each $t\in\T^1$, $\xi\in\Z^r$, $\ell \in \frac{1}{2}\N_0^s$, $-\ell\leq\alpha,\beta\leq \ell$,  the following equation:
\begin{align*}
\partial_t\widehat{u}(t,\xi,\ell)_{\alpha\beta}+i(\langle c(t),\xi\rangle+\langle d(t),\alpha\rangle-iq)\widehat{u}(t,\xi,\ell)_{\alpha\beta} =\widehat{g}(t,\xi,\ell)_{\alpha\beta}.  
\end{align*}
We can consider two cases:
\begin{enumerate}[1.]
    \item $\langle c_0,\xi\rangle+\langle d_0,\alpha\rangle-iq\not\in \Z$
    \item $\langle c_0,\xi\rangle+\langle d_0,\alpha\rangle-iq\in \Z$
\end{enumerate}

In Case 1, according to Lemma \ref{lemmaodesol}, the equation always admits a unique smooth solution.
In Case 2, applying the same lemma, the equation admits a solution if 
$$
	\int_0^{2\pi}\widehat{g}(t,\xi,\ell)_{\alpha\beta}\exp\left(\int_0^ti(\langle c(\tau),\xi\rangle+\langle d(\tau),\alpha\rangle-iq) d\tau\right)\mathop{dt}=0.
$$

Let us define:
\begin{align*}
\widehat{v}(t,-{\xi},\ell)_{(-\alpha)(-\beta)} = \exp\left(\int_0^ti(\langle c(\tau),\xi\rangle+\langle d(\tau),\alpha\rangle-iq) d\tau\right)
\end{align*}
and $\widehat{v}(t,\xi',\ell')_{\alpha'\beta'}\equiv0$ for all other $\xi'\in\Z^r$, $\ell'\in\frac{1}{2}\N_0^s$, $-\ell'\leq\alpha',\beta'\leq \ell'$. 

These coefficients are well-defined since $\langle c_0,\xi\rangle+\langle d_0,\alpha\rangle-iq\in \Z$. By taking the inverse partial Fourier transform of these coefficients, using Proposition \ref{lemmadecaysmoothpartial}, we obtain $v\in C^\infty(\G)$. 

Moreover, for every $t\in\T^1$ and  $v\in\ker \Lt$, we have 
\begin{align*}
\widehat{-\Lt v}(t,-{\xi},\ell)_{(-\alpha)(-\beta)} = & \ \partial_t\widehat{v}(t,-{\xi},\ell)_{(-\alpha)(-\beta)}  \\ 
 & \ + i(\langle c(t),-\xi\rangle+\langle d(t),-\alpha\rangle+iq)\widehat{v}(t,-{\xi},\ell)_{(-\alpha)(-\beta)}\\
= & \ i(\langle c(t),\xi\rangle+\langle d(t),\alpha\rangle-iq)\widehat{v}(t,-{\xi},\ell)_{(-\alpha)(-\beta)} \\
& \ +i(-\langle c(t),\xi\rangle-\langle d(t), \alpha\rangle +iq) \widehat{v}(t,-{\xi},\ell)_{(-\alpha)(-\beta)}=  \ 0.
\end{align*}

Since $g\in(\ker \Lt)^0$, according to Lemma \ref{lemmaformula}, we have:
\begin{align*}
0 &= \langle g,v\rangle = (2\pi)^rd_l\int_0^{2\pi}\widehat{g}(t,\xi,\ell)_{\alpha\beta}\widehat{v}(t,-{\xi},\ell)_{(-\alpha)(-\beta)}\mathop{dt}(-1)^{\Sigma(\beta_j-\alpha_j)}\\
&\implies\int_0^{2\pi}\widehat{g}(t,\xi,\ell)_{\alpha\beta}\exp\left(\int_0^ti(\langle c(\tau),\xi\rangle+\langle d(\tau),\alpha\rangle-iq) d\tau\right)\mathop{dt}=0.
\end{align*}
Therefore, by Lemma \ref{lemmaodesol}, the equation admits infinitely many solutions.
\end{proof}  

Lastly, we present two technical results that will be utilized in the proofs of the forthcoming results regarding necessary conditions.

\begin{lemma}\label{sequence_kn_notinZ}
	If $(\langle c_0,\tilde{\xi}\rangle+\langle d_0,\tilde{\alpha}\rangle)-iq\not\in\Z$, for some $(\tilde{\xi},\tilde{\alpha})\in\Z^r\times \frac{1}{2}\Z^s \setminus  \{0\}$, then  there exists a sequence $(k_n)_{n\in\N}$ of natural numbers such that $k_n\to+\infty$ and
	$$
	k_n(\langle c_0,\tilde{\xi}\rangle+\langle d_0,\tilde{\alpha}\rangle)-iq\not\in\Z, \mbox{ for any } n\in\N.
	$$ 
\end{lemma}

\begin{proof}
	We begin by noting that the result holds trivially when $q\in i\Z$. Hence, we will focus on the case where $q\not\in i\Z$. In this situation, we consider two subcases: either $\Re(q)\neq 0$ or $\Im(q)\not\in \Z$.
	
	In the subcase where $\Re(q)\neq0$, we have either $b_0\neq\0$ or $f_0\neq \0$. In either case, the claim is straightforward because we can find a unique value of $k_n$ such that 
	$$
	\Im(k_n(\langle c_0,\tilde{\xi}\rangle+\langle d_0,\tilde{\alpha}\rangle)-iq) = k_n(\langle b_0,\tilde{\xi}\rangle+\langle f_0,\tilde{\alpha}\rangle)-\Re(q) = 0.
	$$ 
	Now, excluding this value from the sequence $(k_n)$ guarantees that the claim holds for all $n\in\N$.
	
	It remains to analyze only the subcase where $\Im(q)\not\in \Z$. In this situation, we have:
	\begin{equation*}
		(\langle a_0,\tilde{\xi}\rangle+\langle e_0,\tilde{\alpha}\rangle)+\Im(q)\not\in\Z.
	\end{equation*}
	
	Assuming that such a sequence does not exist, we would have:
	$$n(\langle a_0,\tilde{\xi}\rangle+\langle e_0,\tilde{\alpha}\rangle)+\Im(q)\in\Z$$
	for every sufficiently large $n\in\N$. Therefore, for some $n\in\N$, we would have:
	$$
	n(\langle a_0,\tilde{\xi}\rangle+\langle e_0,\tilde{\alpha}\rangle)+\Im(q)\in\Z
	\mbox{ \ and \  \ }
	(n+1)(\langle a_0,\tilde{\xi}\rangle+\langle e_0,\tilde{\alpha} \rangle) + \Im(q)\in\Z.
	$$
	Taking the difference between these two expressions, we have:
	$$(n+1)(\langle a_0,\tilde{\xi}\rangle+\langle e_0,\tilde{\alpha}\rangle)+\Im(q)-n(\langle a_0,\tilde{\xi}\rangle+\langle e_0,\tilde{\alpha}\rangle)-\Im(q))=(\langle a_0,\tilde{\xi}\rangle+\langle e_0,\tilde{\alpha}\rangle)\in \Z.$$
	Since $\Im(q)\not\in \Z$, we would have  
	$$n(\langle a_0,\tilde{\xi}\rangle+\langle e_0,\tilde{\alpha}\rangle)+\Im(q)\not \in\Z$$
	for every $n\in\N$. This leads to a contradiction.
\end{proof}

\begin{lemma}\label{lemmalichangessign}
	Let $g,h\in C^\infty(\T^1)$ be $\R$-linearly independent real valued functions. Then there exists $p_1\neq p_2\in\Z\backslash\{0\}$ such that $t\mapsto p_1g(t)+p_2h(t)$ changes sign. Moreover, if $\int_0^{2\pi}g(t)\mathop{dt}\neq0$ or $\int_0^{2\pi}h(t)\mathop{dt}\neq0$, then we may choose $p_1$ and $p_2$ so that also $\frac{1}{2\pi}\int_0^{2\pi}p_1g(t)+p_2h(t)\mathop{dt}\neq 0$.
\end{lemma}
\begin{proof}
	See  Lemma 3.1. in the reference \cite{BDGK2015_jpdo}.
\end{proof}

\subsection{Necessary conditions} \

We are now ready to embark on the proof of Theorem \ref{teocs}. This proof will be divided into two subsections: in the first subsection, we will establish the necessity of the conditions, and in the next, we will prove their sufficiency. We will show in this section that these results, will ultimately establish the main theorems stated in the introduction.

\begin{prop}\label{propchangessign}
	If the operator $L$ satisfies the condition (CS) then it is neither globally solvable nor globally hypoelliptic. 
\end{prop}
\begin{proof}
	Since $L$ satisfies condition (CS), by Definition \ref{CScondition_definition},  the smooth function
	$$\theta(t) = \langle b(t),\tilde{\xi}\rangle+\langle f(t),\tilde{\alpha}\rangle, \ t \in \T^1,$$
	changes sign, for some $(\tilde{\xi},\tilde{\alpha})\in\Z^r\times \frac{1}{2}\Z^s \setminus  \{0\}$ such that $(\langle c_0,\tilde{\xi}\rangle+\langle d_0,\tilde{\alpha}\rangle)-iq\not\in\Z$. 
	
	By Lemma \ref{sequence_kn_notinZ}, there exists a sequence $(k_n)_{n\in\N}$ of natural numbers such that $k_n\to+\infty$ and
	$$
	k_n(\langle c_0,\tilde{\xi}\rangle+\langle d_0,\tilde{\alpha}\rangle)-iq\not\in\Z, \mbox{ for any } n\in\N.
	$$ 
	
	Considering a subsequence if necessary, we can assume that $k_n\geq n$ for each $n\in\N$. Then we define the following sequences for $n\in\N$:
	\begin{align*}
		\xi(n) & = k_n\tilde{\xi}\in\Z^r,  \  \alpha(n)  = k_n\tilde{\alpha}\in\mbox{$\frac{1}{2}$}\Z^s, \mbox{ and} \\
		\ell(n) & = (k_n|\tilde{\alpha}_1|,\dots,k_n|\tilde{\alpha}_s|)\in \mbox{$\frac{1}{2}$} \N_0^s. 
	\end{align*}
	Let us first consider the case where $$\theta_0\doteq\langle b_0,\tilde{\xi}\rangle+\langle f_0,\tilde{\alpha}\rangle \leq 0.$$
	
	Define
	$$H(s,t) \doteq \int_{t-s}^t\theta(\tau)\mathop{d\tau},\,0\leq t,s\leq 2\pi$$
	and 
	$$A = H(s_0,t_0) \doteq \max_{0\leq s,t\leq 2\pi} H(s,t).$$
	
	Notice that $A>0$ since $\theta$ changes sign. By performing a change of variables, we can assume $0<t_0,s_0,t_0-s_0<2\pi$, and define $\sigma_0 = t_0-s_0$.
	
	Now, let us choose a smooth function $\phi\in C^\infty_c((\sigma_0-\delta,\sigma_0+\delta))$ such that $0\leq\phi\leq 1$, and $\phi\equiv1$ in a neighborhood of $[\sigma_0-\delta/2,\sigma_0+\delta/2]$, where $\delta>0$ is small enough such that $(\sigma_0-\delta,\sigma_0+\delta)\subset (0,t_0)$. 
	
	For each $n\in\N$, define $\widehat{g}(\cdot,\xi(n),\ell(n))_{\alpha(n)\alpha(n)}\in C^\infty(\T^1)$ to be the $2\pi$-periodic smooth extension of
	\begin{equation}\label{hat_g_prop_necessity}
	 t\in\T^1 \mapsto \left(1-e^{-2\pi i(k_n(\langle a_0,\tilde{\xi}\rangle+\langle e_0,\tilde{\alpha}\rangle)+ik_n\theta_0-iq)}\right)e^{-k_nA}\phi(t)e^{ik_n(\langle a_0,\tilde{\xi}\rangle+\langle e_0,\tilde{\alpha}\rangle)(t_0-t)}e^{q(t_0-t)}.
	 	\end{equation}
	In all other cases, define $\widehat{g}(\cdot,\xi,\ell)_{\alpha\beta}\equiv0$. 
	
	Notice that as $-k_nA\leq -nA\to-\infty$, the term $e^{-k_nA}$ decays rapidly. Moreover, we have
	\begin{align*}
		\left|1-e^{+2\pi(k_n\theta_0-\Re(q))-2\pi i(k_n(\langle a_0,\tilde{\xi}\rangle+\langle e_0,\tilde{\alpha}\rangle)+\Im(q))}\right| \leq 1+e^{-2\pi\Re(q)},
	\end{align*}
	since $\theta_0\leq 0$. This shows that the sequence of functions and their derivatives decay rapidly. Therefore, by Proposition \ref{lemmadecaysmoothpartial}, these coefficients define a function $g\in C^\infty(\G)$.
	
	By the definition of $\xi(n)$ and $\alpha(n)$, 
	$$\langle c_0,{\xi(n)}\rangle+\langle d_0,{\alpha(n)}\rangle-iq\not\in\Z.$$
	Hence, by Lemma \ref{lemma0Nimpliesanul}, $g\in (\ker \Lt)^0$. 
	
	Let us assume that there exists $u\in C^\infty(\G)$ such that $Lu=g$. Then, by Plancherel's identity, we have 
	$$\widehat{Lu}(t,\xi(n),\ell(n))_{\alpha(n)\alpha(n)} = \widehat{g}(t,\xi(n),\ell(n))_{\alpha(n)\alpha(n)},$$ 
	for each $n\in \N$. According to Remark \ref{obseqedo} this is equivalent to
	\begin{align*}
		\partial_t\widehat{u}(t,\xi(n),l(n))_{\alpha(n)\alpha(n)}+i(k_n(\langle c(t),\tilde{\xi}\rangle+\langle d(t),\tilde{\alpha}\rangle)-iq)\widehat{u}(t,\xi(n),l(n))_{\alpha(n)\alpha(n)}\\
		=\widehat{g}(t,\xi(n),l(n))_{\alpha(n)\alpha(n)}.
	\end{align*}

Since $k_n(\langle c_0,\tilde{\xi}\rangle+\langle d_0,\tilde{\alpha}\rangle)-iq\not\in\Z$ for each $n\in\N$, according to Lemma \ref{lemmaodesol} and the choice made in \eqref{hat_g_prop_necessity}, we can conclude that these differential equations admit a unique solution, which is given by:

	\begin{align*}
		\widehat{u}(t,\xi(n),\ell(n))_{\alpha(n)\alpha(n)} 
		= & \left(1-e^{-2\pi i(k_n(\langle a_0,\tilde{\xi}\rangle+\langle e_0,\tilde{\alpha}\rangle)+ik_n\theta_0-iq)}\right)^{-1}\int_0^{2\pi}\widehat{g}(t-s,\xi(n),\ell(n))_{\alpha(n)\alpha(n)}\\
		&\times e^{k_nH(s,t)}e^{-is(k_n(\langle a_0,\tilde{\xi}\rangle+\langle e_0,\tilde{\alpha}\rangle)}e^{-qs}\mathop{ds}\\
		= &\int_{0}^{2\pi}\phi(t-s)e^{-k_n(A-H(s,t))}\mathop{ds}e^{i(k_n(\langle a_0,\tilde{\xi}\rangle+\langle e_0,\tilde{\alpha}\rangle)(t_0-t)}e^{q(t_0-t)}.
	\end{align*}
   Since $\supp \phi \subset (\sigma_0-\delta,\sigma_0+\delta)$,  and $0\leq\phi\leq 1$, then
    \begin{align*}
        |\widehat{u}(t_0,\xi(n),\ell(n))_{\alpha(n)\alpha(n)}|&= \int_{s_0-\delta}^{s_0+\delta}\phi(t_0-s)e^{-k_n(A-H(s,t_0))}\mathop{ds}\\
        &\geq \int_{s_0-\delta/2}^{s_0+\delta/2}e^{-k_n(A-H(s,t_0))}\mathop{ds}.
    \end{align*}


Consider the function $\Theta(s) = A - H(s, t_0)$. Since $\Theta(s_0) = 0$ is a local minimum for $\Theta$, it is a zero of order greater than 1. By Lemma \ref{lemmaintegralineq}, for sufficiently large $n$, there exists $M > 0$ such that the following inequality holds:
\begin{align*}
	|\widehat{u}(t,\xi(n),\ell(n))_{\alpha(n)\alpha(n)}|&\geq\left(\int_{-\delta/2}^{\delta/2}e^{-s^2}ds\right)(k_n)^{-1/2}M^{-1/2}. 
\end{align*}

Therefore, we can conclude that there exists a constant $K > 0$ such that\begin{align*}
	|\widehat{u}(t,\xi(n),\ell(n))_{\alpha(n)\alpha(n)}|&\geq K(|\xi(n)|+|\ell(n)|)^{-\frac{1}{2}}.
\end{align*}
According to Proposition \ref{lemmadecaysmoothpartial},it follows that $u\not\in C^{\infty}(\G)$, which is a contradiction. Hence, we can conclude that no such $u$ exists, and therefore $L$ is not globally solvable.

On the other hand, let us consider the expression $\widehat{u}(t,\xi(n),\ell(n))_{\alpha(n)\alpha(n)}$, as defined in \eqref{hat_g_prop_necessity}. For any $n\in\N$, we have the following estimate:
\begin{align*}
	|\widehat{u}(t,\xi(n),\ell(n))_{\alpha(n)\alpha(n)}|&\leq\int_0^{2\pi}\phi(t-s)e^{-k_n(A-H(s,t))}\mathop{ds}e^{\Re(q)(t_0-t)}\\
	&\leq 2\pi\max_{s\in[-2\pi,2\pi]}e^{\Re(q)s}= 2\pi e^{2\pi|\Re(q)|}.
\end{align*}
Since $\widehat{u}(\cdot,\xi,\ell){\alpha\beta}\equiv0$ for all other coefficients, by Proposition \ref{lemmadecaysmoothpartial}, this defines a distribution $u\in \mathscr{D}'(\G)\backslash C^\infty(\G)$. Furthermore, by definition, we have $\widehat{Lu}(t,\xi,\ell){\alpha\beta} = \widehat{g} (t,\xi,\ell)_{\alpha\beta}$ for all $\xi,\ell,\alpha,\beta$. Hence, using Plancherel's identity, we can conclude that $Lu=g\in C^\infty(\G)$. Thus, $L$ is not globally hypoelliptic. This completes the proof of the proposition in the case where $\theta_0\leq 0$.

The case 
$$\theta_0\doteq\langle b_0,\tilde{\xi}\rangle+\langle f_0,\tilde{\alpha}\rangle > 0.$$
follows a similar approach with a few adaptations. Firstly, we define the function
    $$G(s,t) = \int_{t}^{t+s}\theta(\tau)\mathop{d\tau},\,0\leq t,s\leq 2\pi,$$
    and  set
    $$B = G(s_1,t_1)=\min_{0\leq s,t\leq 2\pi} G(s,t).$$
Since $\theta$ changes sign, we have $B < 0$.
Next, we perform a change of variables, assuming $0 < t_1,s_1,t_1+s_1 < 2\pi$, and define $\sigma_1 = t_1+s_1$. We choose $\psi\in C^\infty_c((\sigma_1-\delta,\sigma_1+\delta))$ such that $0 \leq \psi(t) \leq 1$ for all $t \in [0,2\pi]$ and $\psi\equiv1$ in a neighborhood of $[\sigma_1-\delta/2,\sigma_1+\delta/2]$, where $\delta > 0$ is small enough such that $(\sigma_1-\delta,\sigma_1+\delta) \subset (0,t_1)$.
For each $n\in\N$, we define $\widehat{g}(\cdot,\xi(n),\ell(n))_{\alpha(n)\alpha(n)}\in C^\infty(\T^1)$ to be the $2\pi$-periodic smooth extension of
\begin{align*}
	t\in\T^1 \mapsto \left(e^{2\pi i(k_n(\langle a_0,\tilde{\xi}\rangle+\langle e_0,\tilde{\alpha}\rangle)+ik_n\theta_0-iq)}-1\right)e^{k_nB}\psi(t)e^{ik_n(\langle a_0,\tilde{\xi}\rangle+\langle e_0,\tilde{\alpha}\rangle)(t_1-t)}e^{q(t_1-t)}
\end{align*}
and set $\widehat{g}(\cdot,\xi,\ell)_{\alpha\beta} \equiv 0$ for all other cases.
Note that since $k_nB\leq nB\to-\infty$, we have $e^{k_nB} \to 0$ as $n\to\infty$. Also, since $\theta_0 > 0$, we have
    $$\left|e^{-2\pi(k_n\theta_0-\Re(q))+2\pi i(k_n(\langle a_0,\tilde{\xi}\rangle+\langle e_0,\tilde{\alpha}\rangle)+\Im(q))}-1\right|\leq 1+e^{+2\pi\Re(q)}.$$
Thus, the sequence of functions and their derivatives decay rapidly. By Proposition \ref{lemmadecaysmoothpartial}, these coefficients define $g\in C^\infty(\G)$.

Moreover, using the definition of $\xi(n)$ and $\alpha(n)$, we have
$$\langle c_0,{\xi(n)}\rangle+\langle d_0,{\alpha(n)}\rangle-iq\not\in\Z.$$

Therefore, by Lemma \ref{lemma0Nimpliesanul}, we conclude that $g\in (\ker \Lt)^0$. Assuming the existence of $u\in C^\infty(\G)$ such that $Lu=g$, we can employ similar arguments as before to obtain:
    \begin{align*}
        \widehat{u}(t,\xi(n),\ell(n))_{\alpha(n)\alpha(n)}  
        &=\int_{0}^{2\pi}\psi(t+s)e^{-k_n(G(s,t)-B)}\mathop{ds}e^{i(k_n(\langle a_0,\tilde{\xi}\rangle+\langle e_0,\tilde{\alpha}\rangle)(t_1-t)}e^{q(t_1-t)}.
    \end{align*}
Thus
    \begin{align*}
        |\widehat{u}(t_1,\xi(n),\ell(n))_{\alpha(n)\alpha(n)}|&=\int_{s_1-\delta}^{s_1+\delta}\psi(t_1+s)e^{-k_n(G(s,t_1)-B)}\mathop{ds}\\
        &\geq \int_{s_1-\delta/2}^{s_1+\delta/2}e^{-k_n(G(s,t_1)-B)}\mathop{ds}.
    \end{align*}
   
If we let $\Gamma(s) = G(s,t_1)-B$, then we observe that $\Gamma(s_1) = 0$ is a local minimum for $\Gamma$ and has a zero of order greater than 1. By Lemma \ref{lemmaintegralineq}, for sufficiently large $n$, there exists $M > 0$ such that:
\begin{align*}
	|\widehat{u}(t,\xi(n),\ell(n))_{\alpha(n)\alpha(n)}| &\geq \left(\int{-\delta/2}^{\delta/2}e^{-s^2}ds\right)(k_n)^{-1/2}M^{-1/2}.
\end{align*}
Thus, there exists $K > 0$ such that:
\begin{align*}
	|\widehat{u}(t,\xi(n),\ell(n))_{\alpha(n)\alpha(n)}| &\geq K(|\xi(n)|+|\ell(n)|)^{-1/2}.
\end{align*}
Therefore, by Proposition \ref{lemmadecaysmoothpartial}, we conclude that $u\not\in C^{\infty}(\G)$, which leads to a contradiction. This means that no such $u$ exists, and thus $L$ is not globally solvable.

On the other hand, consider $\widehat{u}(t,\xi(n),\ell(n))_{\alpha(n)\alpha(n)}$ as defined previously, we have
   \begin{align*}
	|\widehat{u}(t,\xi(n),\ell(n))_{\alpha(n)\alpha(n)}|&\leq\int_0^{2\pi}\psi(t+s)e^{-k_n(G(s,t)-B)}\mathop{ds}e^{\Re(q)(t_1-t)}\\
	&\leq 2\pi\max_{s\in[-2\pi,2\pi]}e^{\Re(q)s}= 2\pi e^{2\pi|\Re(q)|}, \ n\in\N.
\end{align*}
By letting $\widehat{u}(\cdot,\xi,\ell){\alpha\beta}\equiv0$ for all other coefficients, we can apply Proposition \ref{lemmadecaysmoothpartial} to conclude that these coefficients define $u\in \mathscr{D}'(\G)\backslash C^\infty(\G)$. Therefore, $L$ is not globally hypoelliptic.     
    \end{proof}

\begin{corollary}\label{corocoefficients}
	Suppose that at least one of the functions $b_j$ or $f_k$ is non-null. Then the operator $L$ is neither globally solvable nor globally hypoelliptic if:
\begin{enumerate}[{\em (A)}]
    \item $(b_0,f_0)=\0$ and $(a_0,e_0,q)\not\in\Z^r\times 2\Z^s\times i\Z$ or
    \item $(b_0,f_0)\neq\0$ and either $\dim \spn\{b_1,\dots,b_r,f_1,\dots,f_s\}\geq 2$ or some  of the functions $b_j$ or $f_k$ changes sign.
\end{enumerate}
\end{corollary}
\begin{proof}
	By Proposition \ref{propchangessign}, it suffices to prove that in each case, there exists $\tilde{\xi}\in\Z^r$ and $\tilde{\alpha}\in\frac{1}{2}\Z^s$, both not null, such that:
    $$\langle c_0,\tilde{\xi}\rangle+\langle d_0,\tilde{\alpha}\rangle-iq\not\in\Z$$
    and the real-valued smooth function
    $$\theta(t) = \langle b(t),\tilde{\xi}\rangle+\langle f(t),\tilde{\alpha}\rangle, \ t \in \T^1$$
changes sign.

Let us consider case (A) first.

Notice that since $(b_0,f_0)=\0$ and some of the functions $b_j$ or $f_k$ is not identically zero, then it changes sign. After a change of variables, we may assume that one of the following subcases is true:
\begin{enumerate}[1.]
     \item  $f_1\not\equiv0$,  $q\not\in i\Z$ 
    \item $f_1\not\equiv0$, $e_{10}\not\in 2\Z$ and $q\in i\Z$
    \item $f_1\not\equiv0$, $e_{10}\in 2\Z$, $f_2\equiv0$, $e_{20}\not\in 2\Z$ and $q\in i\Z$
     \item $f_1\not\equiv0$, $e_{10}\in 2\Z$,  $b_1\equiv 0$, $a_{10}\not\in\Z$ and $q\in i\Z$
      \item  $b_1\not\equiv0$,  $q\not\in i\Z$
      \item $b_1\not\equiv 0$, $a_{10}\not\in\Z$ and $q\in i\Z$
    \item $b_1\not\equiv0$, $a_{10}\in \Z$, $b_2\equiv0$, $a_{20}\not\in \Z$ and $q\in i\Z$
     \item $b_1\not\equiv0$, $a_{10}\in \Z$, $f_1\equiv0$, $e_{10}\not\in 2\Z$ and $q\in i\Z$
\end{enumerate}
For subcase 1., if $e_{10}\in\Z$, then it is clear that $e_{10}-iq\not\in \Z$. If $e_{10}\not\in\Z$, we can consider two possibilities: either $e_{10}-iq\not\in \Z$ or $2e_{10}-iq\not\in\Z$. Therefore, by taking $\tilde{\xi} = \0$ and $\tilde{\alpha} = (1,0,\dots,0)$ or $\tilde{\alpha} = (2,0,\dots,0)$, respectively, we have 
$$(\langle c_0,\tilde{\xi}\rangle+\langle d_0,\tilde{\alpha}\rangle)-iq=\tilde{\alpha}_1{e_{10}}-iq\not\in\Z,$$ 
and $\theta(t) = \langle b(t),\tilde{\xi}\rangle+\langle f(t),\tilde{\alpha}\rangle=\tilde{\alpha}_1f_1$ changes sign. Thus, condition (CS) is satisfied. 

The same reasoning applies to subcases 2. and 3., where we can choose $\tilde{\xi} = \0$ and $\tilde{\alpha} = \left(\frac{1}{2},0,\dots,0\right)$ or $\tilde{\alpha} = \left(1,\frac{1}{2},0,\dots,0\right)$, respectively. In subcase 4., take $\tilde{\xi} = \tilde{\alpha} = (1,0,\dots,0)$. The remaining cases can be handled in a similar manner. This completes the proof for case (A).

\noindent Now consider the case (B).

First, assume that $\dim\spn\{b_1,\dots,b_r,f_1,\dots,f_s\}\geq 2$. In this case, let us consider the situation where at least one of the coefficients $f_{k0}$ is non-zero. By performing a change of variables, without loss of generality, we can assume $k=1$ and that one of the following subcases holds:	
     \begin{enumerate}
         \item $f_1$ and $f_2$ are linearly independent
         \item $f_1$ and $b_1$ are linearly independent
     \end{enumerate}

In subcase (1), by applying Lemma \ref{lemmalichangessign}, we can find $p_1,p_2\in\mathbb{Z}\backslash{0}$ such that $p_1f_1(t)+p_2f_2(t)$ changes sign, and $f_{10}p_1+f_{20}p_2\neq 0$. If $f_{10}p_1+f_{20}p_2=\Re(q)$, we can redefine $(p_1,p_2)$ as $(2p_1,2p_2)$ so that the equality does not hold. Let $\tilde{\xi} = \0$ and $\tilde{\alpha}=(p_1,p_2,0,\dots,0)$. Since
     $$\Im\left((\langle c_0,\tilde{\xi}\rangle+\langle d_0,\tilde{\alpha}\rangle)-iq\right)\ = (f_{10}p_1+f_{20}p_2)-\Re(q)\neq 0,$$
we conclude that $(\langle c_0,\tilde{\xi}\rangle+\langle d_0, \tilde{\alpha} \rangle) - iq \not\in\Z$.

Subcase (2) follows a similar approach. We again apply Lemma \ref{lemmalichangessign} to find $p_1,p_2\in\Z\backslash\{0\}$ such that $p_1f_1(t)+p_2b_1(t)$ changes sign, and $f_{10}p_1+b_{10}p_2\neq 0$. If $f_{10}p_1+b_{10}p_2=\Re(q)$, we can redefine $(p_1,p_2)$ as $(2p_1,2p_2)$ to avoid this equality. We then take $\tilde{\xi} = (p_1,0,\dots,0)$ and $\tilde{\alpha}=(p_2,0,\dots,0)$. Consequently,
     $$\Im\left((\langle c_0,\tilde{\xi}\rangle+\langle d_0,\tilde{\alpha}\rangle)-iq\right)\ = (b_{10}p_2+f_{10}p_1)-\Re(q)\neq 0$$
and therefore $(\langle c_0,\tilde{\xi}\rangle+\langle d_0,\tilde{\alpha}\rangle)-iq\not\in\Z$.

In the case where $f_0=\0$, we can proceed similarly by considering the situation where some $b_{j0}\neq 0$. The proof follows along the same lines.

Next, assume that at least one of the coefficients $b_j$ or $f_k$ changes sign. Based on the previous arguments, we can assume that each $b_j$ and $f_k$ is either null or changes sign. 

Let us begin by considering the case where $b_0\neq \0$. After a change of variables, we can assume that $b_{10}\neq 0$, implying that $b_1$ changes sign. We can then choose $\tilde{\xi} = (\pm1,0,\dots,0)$, taking into account that $b_{10}\neq 0$. This choice ensures that   
     $$
     \Im\left((\langle c_0,\tilde{\xi}\rangle+\langle d_0, \tilde{\alpha} \rangle) - iq \right)\ = \pm b_{10}-\Re(q)\neq 0
     $$
and therefore $\langle c_0,\tilde{\xi}\rangle+\langle d_0,\tilde{\alpha}\rangle)-iq\not\in\Z$. Additionally, $\langle b(t),\tilde{\xi}\rangle+\langle f(t),\tilde{\alpha}\rangle=\pm b_1(t)$ changes sign.   
   
In the case where $b_0=\0$, we can argue similarly. Since $f_0\neq 0$, we can conclude that $f_1$ changes sign. The proof follows along the same lines by choosing appropriate $\tilde{\xi}$ and $\tilde{\alpha}$ values.
\end{proof}

This corollary provides the first evidences for the necessity of the conditions stated in Theorems \ref{teogsmain} and \ref{teoghmain}. The remaining aspects will be addressed in the next propositions.

\begin{prop}\label{notconnectednotgs}
	Assume that at least one of the functions $b_j$ or $f_k$ is non-zero, and that $(b_{0},f_{0})=\0$. Furthermore, suppose there exists a sublevel set $\Omega_m^{\xi,\alpha}$ that is not connected. In this situation, $L$ is not globally solvable.
\end{prop}
\begin{proof}
In view of case (A) in Corollary \ref{corocoefficients}, we can assume that $(a_0,e_0,q)\not\in\Z^r\times 2\Z^s\times i\Z$. And note that $L$ does not satisfy condition (CS).

Now, let us proceed with a proof by contradiction using the traditional Hormander method of violating inequalities, which relies on Lemma \ref{Hormander_Method}.

Suppose that $L$ is globally solvable and that, for some  $\tilde{\xi}\in\Z^r$, $\tilde{\alpha}\in \frac{1}{2}\Z^s$, and $m_0'\in\R$, the sublevel set  
$$
  \Omega_{m_0'}^{\tilde{\xi}, \tilde{\alpha}} = \left\{t\in\T^1;\int_0^t \big(\langle b(\tau),\tilde{\xi}\rangle+\langle f(\tau),\tilde{\alpha}\rangle\big) d\tau<m_0'\right\} \mbox{ is  disconnected.}
$$
It is worth noting that this assumption implies that either $\tilde{\xi}\neq\0$ or $\tilde{\alpha}\neq\0$. By applying Lemma \ref{lemmaf0v0}, we can find $m_0\in\R$ as well as functions $g_0$ and $v_0$ in $C^\infty(\T^1)$ that satisfy the following conditions:
	\begin{align*}
		& \int_0^{2\pi} g_0(t)dt=0, \quad  \supp g_0\cap \Omega_{m_0}^{\tilde{\xi},\tilde{\alpha}}=\varnothing,  \mbox{ and} \\
		& \int_0^{2\pi}g_0(t)v_0(t)\mathop{dt}=\ell_0>0 \ \ 		\supp v_0'\subset \Omega_{m_0}^{\tilde{\xi},\tilde{\alpha}},  
	\end{align*}
	
For each $n\in\N$, we define $n\tilde{\ell} = (n|\tilde{\alpha}_1|,\dots,n|\tilde{\alpha}_s|)$ and $g_n, v_n \in C^\infty(\T^{r+1}\times\S3)$ as follows:
    \begin{align*}
        g_n(t,x,y) &= {\sqrt{d_{n\tilde{\ell}}}}
        \exp\left(ni\int_0^t\langle c(\tau),\tilde{\xi}\rangle+\langle d(\tau),\tilde{\alpha}\rangle d\tau{-qt}\right)g_0(t)e^{-ni\langle x,\tilde{\xi}\rangle}\mathfrak{t}^{n\tilde{\ell}}_{(-n\tilde{\ell})(-n\tilde{\ell})}(y)\\
        v_n(t,x,y) &={\sqrt{d_{n\tilde{\ell}}}}
        \exp\left(-ni\int_0^t\langle c(\tau),\tilde{\xi}\rangle+\langle d(\tau),\tilde{\alpha}\rangle d\tau{+qt}\right)v_0(t)e^{ni\langle x,\tilde{\xi}\rangle}\mathfrak{t}^{n\tilde{\ell}}_{(n\tilde{\ell})(n\tilde{\ell})}(y)
    \end{align*}
for every $(t,x,y)\in\G$. It is important to note that these functions are well-defined based on the given hypotheses for $c_0, d_0,$ and $q$.

We claim that $g_n\in(\ker \Lt)^0$. Indeed, if $u\in\ker \Lt$ then 
\begin{equation}\label{eqkerlt}
	\widehat{-\Lt u}(t,\xi,\ell)_{\alpha\beta}=0\implies [\partial_t+i(\langle c(t),\xi\rangle+\langle d(t),\alpha\rangle+iq)]\widehat{u}(t,\xi,\ell)_{\alpha\beta}=0.
\end{equation}
Now, using the formula from Lemma \ref{lemmaformula}, we have
    \begin{align*}
        \langle u,g_n\rangle &= (2\pi)^r\sum_{\xi\in\Z^r}\sum_{\ell\in\frac{1}{2}\N_0}d_\ell\sum_{-\ell\leq\alpha,\beta\leq \ell} \int_0^{2\pi} \widehat{u}(t,\xi,\ell)_{\alpha\beta} \widehat{g_n}(t,-\xi,\ell)_{(-\alpha)(-\beta)} \mathop{dt}(-1)^{\Sigma(\beta_j-\alpha_j)}\\
        &=\sqrt{d_{n\tilde{\ell}}}(2\pi)^r\int_0^{2\pi}\widehat{u}(t,n\tilde{\xi},n\tilde{\ell})_{(n\tilde{\ell})(n\tilde{\ell})}\exp\left(ni\int_0^t\langle c(\tau),\tilde{\xi}\rangle+\langle d(\tau),\tilde{\alpha}\rangle d\tau{-qt}\right) g_0(t)\mathop{dt}\\
        &=\sqrt{d_{n\tilde{\ell}}}(2\pi)^r\int_0^{2\pi}\widehat{u}(t,n\tilde{\xi},n\tilde{\ell})_{(n\tilde{\ell})(n\tilde{\ell})}\exp(niw(t){-qt})g_0(t)\mathop{dt},
    \end{align*}
     where $w(t) = \displaystyle \int_0^t\langle c(\tau),\tilde{\xi}\rangle+\langle d(\tau),\tilde{\alpha}\rangle d\tau$. 
     
     Since
    \begin{align*}
       \partial_t&\left(\exp(niw(t){-qt})\widehat{u}(t,n\tilde{\xi},n\tilde{\ell})_{(n\tilde{\alpha})(n\tilde{\alpha})}\right)= \exp(niw(t){-qt}) \times \\
        & \ \times \left[\partial_t\widehat{u}(t,n\tilde{\xi},n\tilde{\ell})_{(n\tilde{\alpha})(n\tilde{\alpha})}+ \left(ni[\langle c(t),\tilde{\xi}\rangle+\langle d(t),\tilde{\alpha}\rangle]-q\right)\widehat{u}(t,n\tilde{\xi},n\tilde{\ell})_{(n\tilde{\alpha})(n\tilde{\alpha})}\right] = 0,
    \end{align*}
    
    Then, for every $n\in\N$, we have
    $$\exp(niw(t){-qt})\widehat{u}(t,n\tilde{\xi},n\tilde{\ell})_{(n\tilde{\alpha})(n\tilde{\alpha})} \equiv k_n\in\mathbb{C}.$$
    Consequently
    \begin{align*}
        \langle u,g_n\rangle = \sqrt{d_{n\tilde{\ell}}}\ (2\pi)^r k_{n} \int_0^{2\pi} g_0(t) \mathop{dt} = 0,
    \end{align*}
which implies that $g_n\in(\ker \Lt)^0$ for all $n\in\N$, as claimed.

\
    
    Therefore, by Lemma \ref{Hormander_Method} there exists $C>0$, $\lambda\in\N_0$ such that
    \begin{align}\label{eqgnvngs}
    \left|\int_{\T^{r+1}\times\S3}g_nv_n\right|\leq & \ C\left(\sum_{|\mu|\leq \lambda}\sup_{(t,x,y)\in\T^{r+1}\times\S3}|\partial^\mu g_n(t,x,y)|\right) \nonumber\\
        &\ \ \times\left(\sum_{|\mu|\leq \lambda}\sup_{(t,x,y)\in\T^{r+1}\times\S3}|\partial^\mu (\Lt v_n)(t,x,y)|\right).
    \end{align}

However, note that by the definitions of $g_n$ and $v_n$, we have
   \begin{equation}\label{eqgnvn}
	\int_{\T^{r+1}\times\S3}g_nv_n=(2\pi)^r{\frac{d_{n\tilde{\ell}}}{d_{n\tilde{\ell}}}}\int_0^{2\pi}f_0(t)v_0(t)dt=(2\pi)^rl_0\geq \ell_0>0,\ n\in\N.
\end{equation}
On the other hand, since $\supp f_0\cap\Omega_{m_0}^{\tilde{\xi},\tilde{\alpha}}=\varnothing$, by the Leibniz formula, there exist positive constants $M_1$, $M_\mu$, and $M_1'$ such that
    \begin{align*}
    &\sum_{|\mu|\leq      
    \lambda}\sup_{(t,x,y)\in\T^{r+1}\times\S3}|\partial^\mu g_n(t,x,y)|\leq\\
    & \leq  \sqrt{d_{n\tilde{\ell}}}\sum_{|\mu|\leq      
    \lambda}\sup_{(t,x,y)\in\T^{r+1}\times\S3}\left|\sum_{\gamma_1+\gamma_2+\gamma_3+\gamma_4=\mu}{\binom{\mu}{ \gamma_1,\gamma_2,\gamma_3,\gamma_4}}\partial^{\gamma_1}[\exp(niw(t)-qt)]\partial^{\gamma_2}[e^{-in\langle x,\tilde{\xi}\rangle}]\right|\\[1mm]
    &\quad  \ \times \left|\partial^{\gamma_3} [g_0(t)] \partial^{\gamma_4} \mathfrak{t}^{n\tilde{\ell}}_{(-n\tilde{\alpha})(-n\tilde{\alpha})}(y)\right|\\[2mm]
    &\leq  \sqrt{d_{n\tilde{\ell}}} \sum_{|\mu|\leq      
    \lambda}M_1M_\mu n^{|\mu|}\sup_{t\in \T^1\backslash\Omega_{m_0}^{\tilde{\xi},\tilde{\alpha}}}|\exp(niw(t){-qt})|\sum_{\gamma_4\leq\mu}\sup_{y\in\S3}|\partial^{\gamma_4}\mathfrak{t}^{n\tilde{\ell}}_{(-n\tilde{\alpha})(-n\tilde{\alpha})}(y)|\\
    &\leq  \sqrt{d_{n\tilde{\ell}}}  M_1'n^\lambda\sup_{t\in \T^1\backslash\Omega_{m_0}^{\tilde{\xi},\tilde{\alpha}}}\left[\exp\left(-n\int_0^t\langle b(\tau),\tilde{\xi}\rangle+\langle f(\tau),\tilde{\alpha}\rangle d\tau\right)\right]\sup_{\gamma_4\leq \mu}\sup_{y\in\S3}|\partial^{\gamma_4}\mathfrak{t}^{n\tilde{\ell}}_{(-n\tilde{\alpha})(-n\tilde{\alpha})}(y)|
    \end{align*}
    
    For $t\not\in \Omega_{m_0}^{\tilde{\xi},\tilde{\alpha}}$ we have $\displaystyle \int_0^t \big(\langle b(\tau),\tilde{\xi}\rangle+\langle f(\tau),\tilde{\alpha}\rangle\big) d\tau>m_0$, hence
    $$\sup_{t\in \T^1\backslash\Omega_{m_0}^{\tilde{\xi},\tilde{\alpha}}} \left[\exp\left(-n\int_0^t\langle b(\tau),\tilde{\xi}\rangle+\langle f(\tau),\tilde{\alpha}\rangle d\tau\right)\right]\leq e^{-nm_0}, n \in \N.$$

Furthermore, for any $-\ell_j \leq \alpha_j \leq \ell_j$ and $y_j \in \mathbb{S}^3$  the unitarity identity holds:
    \begin{align*}
        1=&\mathfrak{t}^{\ell_j}_{\alpha_j\alpha_j}({\tt e})=\sum_{\beta_j=-\ell_j}^{\ell_j}\mathfrak{t}^{\ell_j}_{\alpha_j\beta_j}(y_j)\mathfrak{t}^{\ell_j}_{\beta_j\alpha_j}(y_j^{-1})\\
        &=\sum_{\beta_j=-\ell_j}^{\ell_j}\mathfrak{t}^{\ell_j}_{\alpha_j\beta_j}(y_j)\overline{\mathfrak{t}^{\ell_j}_{\alpha_j\beta_j}}(y_j)= \sum_{\beta_j=-\ell_j}^{\ell_j}\left|\mathfrak{t}^{\ell_j}_{\alpha_j\beta_j}(y_j)\right|^2
    \end{align*}
where ${\tt e} \in \mathbb{S}^3$ is the neutral element of the Lie group $\mathbb{S}^3$.

This implies that $\left|\mathfrak{t}^{\ell_j}_{\alpha_j\beta_j}(y_j)\right| \leq 1$ for all $\alpha_j, \beta_j$, and $y_j$. Now, if $\partial^{\gamma_4}$ is a left-invariant differential operator of order less than or equal to $\lambda$, then it can be expressed as a linear combination of at most $\lambda$ (including 0) compositions of the vector fields in $\{D_{1,j}, D_{2,j}, D_{3,j}\}_{j=1}^s$, where each $D_{1,j}$ and $D_{2,j}$ represent the vector fields on $\mathbb{S}^3$ given by ${\partial}/{\partial\phi_j}$ and ${\partial}/{\partial\theta_j}$, respectively, in local coordinates. Here, $(\phi_j, \theta_j, \psi_j)$ are Euler angle coordinates on $\mathbb{S}^3$, and so their associated coordinate vector fields form a basis for the Lie algebra of the group.

It is clear that $D_{k,j}\mathfrak{t}^{\ell_{j'}}_{\alpha{j'}\beta_{j'}} = 0$ if $j \neq j'$.
Furthermore, according to \cite{RT2010_book} Chapter 11, we have the following result:
    \begin{align*}
        &D_{3,j}\mathfrak{t}^{\ell_j}_{\alpha_j\beta_j} = i\beta_j\mathfrak{t}^{\ell_j}_{\alpha_j\beta_j};\\[2mm]
        &D_{2,j}\mathfrak{t}^{\ell_j}_{\alpha_j\beta_j} =\frac{\sqrt{(\ell_j-\beta_j)(\ell_j+\beta_j+1)}}{2} \mathfrak{t}^{\ell_j}_{\alpha_j\beta_j+1}-\frac{\sqrt{(\ell_j+\beta_j)(\ell_j-\beta_j+1)}}{2} \mathfrak{t}^{\ell_j}_{\alpha_j\beta_j-1};\\[2mm]
        &D_{1,j}\mathfrak{t}^{\ell_j}_{\alpha_j\beta_j} =\frac{\sqrt{(\ell_j-\beta_j)(\ell_j+\beta_j+1)}}{-2i} \mathfrak{t}^{\ell_j}_{\alpha_j\beta_j+1}+\frac{\sqrt{(\ell_j+\beta_j)(\ell_j-\beta_j+1)}}{-2i} \mathfrak{t}^{\ell_j}_{\alpha_j\beta_j-1}.
    \end{align*}

Therefore, for all $\ell_j, \alpha_j, \beta_j$, and $k$, we have
    $$|D_{k,j}\mathfrak{t}^{\ell_j}_{\alpha_j\beta_j}(y_j)|\leq 2\ell_j .$$

By induction, we obtain
    $$|\partial^{\gamma_4}\mathfrak{t}^{n|\tilde{\alpha_j}|}_{(-n\tilde{\alpha}_j)(-n\tilde{\alpha}_j)}|\leq 2^{|\gamma_4|}n^{|\gamma_4|}|\tilde{\alpha_j}|^{|\gamma_{4}|}. $$
    
Since 
$$d_{n\tilde{\ell}}=\prod_{j=1}^s(2n|\tilde{\alpha}_j|+1)\leq\prod_{j=1}^s 4(\|\tilde{\alpha}\|_{\infty}+1)n=4^s(\|\tilde{\alpha}\|_{\infty}+1)^sn^s,$$
substituting this inequality into the previous inequality, we obtain:
\begin{align*}
	\sum_{|\mu|\leq \lambda}\sup_{(t,x,y)\in\T^{r+1}\times\S3}|\partial^\mu g_n(t,x,y)|&\leq M_1'n^{\lambda+s/2}e^{-nr_0}2^{\lambda+s} n^\lambda \|\tilde{\alpha}\|_\infty^\lambda(\|\tilde{\alpha}\|_{\infty}+1)^{s/2}\\
	&=M_1''n^{2\lambda+s/2}e^{-nm_0}.
\end{align*}     
    
    Also, observe that $\Lt\left(\exp(-inw(t){+qt})e^{in\langle x,\tilde{\xi}\rangle}\mathfrak{t}^{n\tilde{\ell}}_{(n\tilde{\alpha})(n\tilde{\alpha})}(y)\right)=0$. 
    
    Therefore, we have
    $$\Lt v_n(t,x,y) = \sqrt{d_{n\tilde{\ell}}}(\Lt+\partial_t)\left[v_0'(t)\exp(-inw(t){+qt})e^{in\langle x,\tilde{\xi}\rangle}\mathfrak{t}^{n\tilde{\ell}}_{(n\tilde{\alpha})(n\tilde{\alpha})}(y)\right].$$
    
    Hence,
        \begin{align*}
        &\sum_{|\mu|\leq \lambda}\sup_{(t,x,y)\in\T^{r+1}\times\S3}|\partial^\mu (\Lt v_n)(t,x,y)| \\[2mm]
        &\leq  \sqrt{d_{n\tilde{\ell}}}\sum_{|\mu|\leq 
        \lambda+1}\sup_{(t,x,y)\in\T^{r+1}\times\S3}\left|\partial^{\mu}\left[v_0'(t)\exp(-niw(t){+qt})e^{in\langle x, \tilde{\xi}\rangle} \mathfrak{t}^{n\tilde{\ell}}_{(n\tilde{\alpha}) (n\tilde{\alpha})}(y)\right]\right| \\[2mm]
        &\leq \sqrt{d_{n\tilde{\ell}}} M_2n^{(\lambda+1)}\sup_{t\in\supp v_0'}\left[\exp\left(n\int_0^t\langle b(\tau),\tilde{\xi}\rangle+\langle f(\tau),\tilde{\alpha}\rangle d\tau\right)\sup_{|\mu|\leq \lambda+1}\left|\partial^{\mu}\mathfrak{t}^{n\tilde{\ell}}_{(n\tilde{\alpha})(n\tilde{\alpha})}(y)\right|\right] \\[2mm]
        &\leq 2^{s}M_2n^{2(\lambda+1)+s/2}\|\tilde{\alpha}\|_{\infty}^{\lambda+1}e^{n\int_0^{t_1}\langle b(\tau),\tilde{\xi}\rangle+\langle f(\tau),\tilde{\alpha}\rangle d\tau}(\|\tilde{\alpha}\|_{\infty}+1)^{s/2}\\[2mm]
        &= M_2'n^{2(\lambda+1)+s/2}e^{n\int_0^{t_1}\langle b(\tau),\tilde{\xi}\rangle+\langle f(\tau),\tilde{\alpha}\rangle d\tau},
    \end{align*}
where $t_1\in\supp(v_0')\subset\Omega_{m_0}^{\tilde{\xi},\tilde{\alpha}}$ is a point where the  restriction of the exponential function achieves its maximum value over $\supp(v_0')$. 

We define 
$$\omega = \int_0^{t_1}\langle b(\tau),\tilde{\xi}\rangle+\langle f(\tau),\tilde{\alpha}\rangle d\tau - m_0.$$ 

Notice that $\omega < 0$, since $t_1\in\Omega_{m_0}^{\tilde{\xi},\tilde{\alpha}}$.

By using the previous inequalities, we can establish the following:
    \begin{align*}
	\left(\sum_{|\mu|\leq \lambda}\sup_{(t,x,y)\in\T^{r+1}\times\S3}|\partial^\mu g_n(t,x,y)|\right) & \left(\sum_{|\mu|\leq\lambda} \sup_{(t,x,y) \in\G}| \partial^\mu (\Lt v_n)(t,x,y)|\right) \\[3mm]
	&\leq M_1''M_2'n^{4\lambda+3+s}e^{n\omega}.
\end{align*}

Consequently, by \eqref{eqgnvngs} and \eqref{eqgnvn}, we have:
\begin{align*}
	0 = \lim_{n\to\infty}\left|\int_{\T^{r+1}\times\S3}g_nv_n\right| \geq \ell_0 > 0.
\end{align*}
This leads to a contradiction, proving that $L$ cannot be globally solvable.
\end{proof}

Next, we prove the necessity of the conditions on the global properties of the operator $L_0$. First, let us recall the Diophantine Condition (DC) defined in the previous section, as stated in Definition \ref{DC_condition_const_coef_section}:

The operator $L_0$ satisfies the Diophantine Condition (DC) if there exist $M$ and $N$ such that:
\begin{equation}\label{DC_lemma_4.5}
	|\tau + \langle c_0,\xi\rangle + \langle d_0,\alpha\rangle - iq| \geq M(|\tau| + |\xi| + |\ell|)^{-N}
\end{equation}
for every $(\tau,\xi)\in\mathbb{Z}^{r+1}$, $\ell\in\frac{1}{2}\N_0^s$, $-\ell\leq \alpha\leq \ell$, and $\ell-\alpha\in\N_0^s$, where neither the left-hand side of the inequality is zero nor $(\tau,\xi,\ell)$ is the zero vector.

We will make use of the following lemma to establish the equivalence:

\

\begin{lemma}\label{lemmadiofequi}
The following statements are equivalent:
	\begin{enumerate}[1.]
		\item There exists $C>0$ and $N>0$ such that
		$$\left|1-e^{\pm2\pi i(\langle c_0,\xi\rangle+\langle d_0,\alpha\rangle-iq)}\right|\geq C (|\xi|+|\ell|)^{-N}$$
		for all $\xi\in\Z^r$, $\ell\in\frac{1}{2}\N_0^{s}$, $-\ell\leq \alpha\leq \ell$,\,$\ell-\alpha\in\N_0^s$ such that $\langle c_0,\xi\rangle+\langle d_0,\alpha\rangle-iq\not\in\Z$ and $(\xi,\ell)\neq \0.$
		\item The operator $L_0$ satisfies the Diophantine Condition (DC).
	\end{enumerate}
\end{lemma}
\begin{proof}
	Suppose that the statement 1 does not hold. This means that for every $n\in\N$, there exists $\xi(n)\in\Z^r$, $\ell(n)\in\frac{1}{2}\N_0^s$, and $\alpha(n)\in\frac{1}{2}\Z$ such that:
	$$0<\left|1-e^{\pm2\pi i(\langle c_0,\xi(n)\rangle+\langle d_0,\alpha(n)\rangle-iq)}\right|<\frac{1}{n}\left(|\xi(n)|+|l(n)|\right)^{-n}.$$

	This implies that $|\Re(q) - \langle b_0,\xi(n)\rangle - \langle f_0,\alpha(n)\rangle| \to 0$ and there exists a sequence of integers $(\tau_n)_n$ such that:
	$$|\tau_n+\langle a_0,\xi(n)\rangle+\langle e_0,\alpha(n)\rangle+\Im(q)|\to0.$$

	Moreover, it is easy to verify that for $a,b\in\R$,
	$$|1-e^{a+bi}|^2\geq|1-e^a|^2.$$

	Therefore, for sufficiently large $n$, by the Mean Value Theorem, we have:
	\begin{align*}
		\left|1-e^{\pm2\pi i(\langle c_0,\xi(n)\rangle+\langle d_0,\alpha(n)\rangle-iq)}\right|&\geq \left|1-e^{\pm2\pi (\Re(q)-\langle b_0,\xi(n)\rangle-\langle f_0,\alpha(n)\rangle)}\right|\\
		&\geq e^{-1}2\pi|\Re(q)-\langle b_0,\xi(n)\rangle-\langle f_0,\alpha(n)\rangle|.
	\end{align*}

	Moreover, since $\frac{\sin(2\pi x)}{2\pi x} \to 1$ as $x \to 0$, for all sufficiently large $n$, we have:
	\begin{align*}
		|\sin(2\pi(\tau_n+\langle a_0,\xi(n)\rangle+\langle e_0,\alpha(n)\rangle+\Im(q))|\geq \pi |\tau_n+\langle a_0,\xi(n)\rangle+\langle e_0,\alpha(n)\rangle+\Im(q)|.
	\end{align*}

	Therefore, for all sufficiently large $n$, we have:
	\begin{align*}
		&\pi |\tau_n+\langle a_0,\xi(n)\rangle+\langle e_0,\alpha(n)\rangle+\Im(q)|\leq |\sin(2\pi(\tau_n+\langle a_0,\xi(n)\rangle+\langle e_0,\alpha(n)\rangle+\Im(q))|\\
		&\leq 2e^{2\pi(\Re(q)-\langle b_0,\xi(n)\rangle-\langle f_0,\alpha(n)\rangle)}|\sin(2\pi(\tau_n+\langle a_0,\xi(n)\rangle+\langle e_0,\alpha(n)\rangle+\Im(q))|\\
		&=2\left|\Im(1-e^{\pm2\pi i(\langle c_0,\xi(n)\rangle+\langle d_0,\alpha(n)\rangle-iq)})\right|\\
		&\leq \left|1-e^{\pm2\pi i(\langle c_0,\xi(n)\rangle+\langle d_0,\alpha(n)\rangle-iq)}\right|.
	\end{align*}
	
	Combining this with the previous inequalities, we can conclude that there exists $C > 0$ such that: 
	\begin{align*}
		0<|\tau_n+\langle c_0,\xi(n)\rangle+\langle d_0,\alpha(n)\rangle|&\leq  |\tau_n+\langle a_0,\xi(n)\rangle+\langle e_0,\alpha(n)\rangle+\Im(q)|+\\
		&+|\Re(q)-\langle b_0,\xi(n)\rangle-\langle f_0,\alpha(n)\rangle|\\
		&\leq C\left|1-e^{\pm2\pi i(\langle c_0,\xi(n)\rangle+\langle d_0,\alpha(n)\rangle-iq)}\right|\\
		&\leq \frac{C}{n}\left(|\xi(n)|+|\ell(n)|\right)^{-n}.
	\end{align*}

	Hence, condition 2 cannot hold. Indeed, if it did hold, then for some $K, N > 0$, for every $n \in \N$, we would have:
	\begin{align*}
		K\left(|\tau_n|+|\xi(n)|+|\ell(n)|\right)^{-N}\leq |\tau_n+\langle c_0,\xi(n)\rangle+\langle d_0,\alpha(n)\rangle-iq|<\frac{C}{n}\left(|\xi(n)|+|\ell(n)|\right)^{-n}.
	\end{align*}
	Thus 
	\begin{align*}
		0<\frac{K}{C}&<\frac{1}{n}\left(|\xi(n)|+|\ell(n)|\right)^{-n}\left(|\tau_n|+|\xi(n)|+|\ell(n)|\right)^{N}\\
		&=\frac{1}{n}\left(|\xi(n)|+|\ell(n)|\right)^{-n+N}\left(\frac{|\tau_n|}{|\xi(n)|+|\ell(n)|}+1\right)^N.
	\end{align*}

	But since $|\tau_n+\langle a_0,\xi(n)\rangle+\langle e_0,\ell(n)\rangle+\Im(q)|<1$ for all sufficiently large $n$, 
	$$|\tau_n|\leq 1+(|\xi(n)|+|\ell(n)|)M+|\Im(q)|,$$
	where $0\leq M= \max\{|a_{10}|,\dots,|a_{r0}|,|e_{10}|,\dots,|e_{s0}|\}$.
	
	Therefore
	$$0<\frac{K}{C}<\frac{1}{n}\left(|\xi(n)|+|\ell(n)|\right)^{-n+N}\left(\frac{1+|\Im(q)|}{|\xi(n)|+|\ell(n)|}+M+1\right)^{N}.$$

	This leads to a contradiction since the right-hand side of the inequality tends to $0$ as $n$ approaches infinity.

	Conversely, suppose that condition 2 does not hold. Then, for every $n\in\N$, there exist $\tau_n\in\Z$, $\xi(n)\in\Z^r$, $\ell(n)\in\frac{1}{2}\N_0^s$, $-\ell(n)\leq \alpha(n)\leq \ell(n)$, and $\ell(n)-\alpha(n)\in\N_0^s$ such that:
	\begin{align*}
		0<|\tau_n+\langle c_0,\xi(n)\rangle+\langle d_0,\alpha(n)\rangle-iq|<\frac{1}{n}(|\tau_n|+|\xi(n)|+|\ell(n)|)^{-n}.
	\end{align*}

	In particular, the middle term tends to $0$ as $n$ approaches infinity. Now, we can define, for every $n$:
	$$\gamma_n+ i\mu_n = \tau_n+\langle c_0,\xi(n)\rangle+\langle d_0, \alpha(n) \rangle - iq,$$
	where each $\gamma_n,\mu_n\in \R$. 
	
	Using this notation, we have
	\begin{align*}
		\left|1-e^{\mp2\pi i(\langle c_0,\xi(n)\rangle+\langle d_0,\alpha(n)\rangle-iq)}\right|&\leq | 1-e^{\pm2\pi \mu_n}\cos(2\pi\gamma_n)|+e^{\pm2\pi \mu_n}|\sin(2\pi\gamma_n)|.
	\end{align*}

	But since, for $x,y\in\R$, 
	$$|1-e^x\cos(y)|\leq |1-\cos(y)|+|1-e^x|.$$
	From the previous inequality:
	\begin{align*}
		\left|1-e^{\pm2\pi i(\langle c_0,\xi(n)\rangle+\langle d_0,\alpha(n)\rangle-iq)}\right|&\leq|1-\cos(2\pi\gamma_n)|+|1-e^{\pm2\pi \mu_n}|+e^{\pm2\pi\mu_n}|0-\sin(2\pi\gamma_n)|\\
		&\leq C_12\pi|\gamma_n|+C_2 2\pi|\mu_n|+C_32\pi|\gamma_n|\\
		&\leq C(|\gamma_n|+|\mu_n|),
	\end{align*}
	for some $C_1,C_2,C_3,C>0$, where we used the mean value theorem for each term, and the fact that $\mu_n$ and $\gamma_n$ tend to $0$. 
	
	Therefore,
	\begin{align*}
		\left|1-e^{\pm2\pi i(\langle c_0,\xi(n)\rangle+\langle d_0,\alpha(n)\rangle-iq)}\right|&\leq 2C|\gamma_n+i\mu_n|\\
		&<2\frac{C}{n}(|\tau_n|+|\xi(n)|+|\ell(n)|)^{-n}\\
		&\leq2\frac{C}{n}(|\xi(n)|+|\ell(n)|)^{-n}.
	\end{align*}
	Hence, we conclude that condition 1 does not hold.
\end{proof}

\begin{prop}\label{propl0notgsgh}
	Suppose $(b_0,f_0)\neq\0$. If the operator $L_0$ is not globally solvable, then neither the operator $L$ is globally solvable nor globally hypoelliptic.
\end{prop}
\begin{proof}
First we will now prove that $L$ is not globally solvable.  By Corollary \ref{corocoefficients} and Proposition \ref{propchangessign}, we can assume that
there are $\tilde{b}\in C^\infty(\T^1)$, $\lambda\in\R^r$, and $\gamma\in\R^s$ such that
	$$b(t)=\tilde{b}(t)\lambda, \mbox{ and \ } \, f(t)=\tilde{b}(t)\gamma, \ \ t\in\T^1$$

Since $L_0$ is not globally solvable, it does not satisfy \eqref{DC_condition_const_coef_section} according to Proposition \ref{propctegs}. By Lemma \ref{lemmadiofequi}, and by considering subsequences if necessary, there exist sequences $(\xi(n))_{n\in\N}$ in $\Z^{r}$, $(\ell(n))_{n\in\N}$ in  $\frac{1}{2}\N_0^s$, and $(\alpha(n))_{n\in\N}$ in $\frac{1}{2}\Z^s$ such that for every $n\in\N$, 
$$|\xi(n)|+|\ell(n)|\geq n, \quad \langle c_0,\xi(n)\rangle+\langle d_0,\alpha(n)\rangle-iq\not\in \Z,$$
and, for every $n\in\N$, the following inequalities hold with $C>0$:
	\begin{align*}
		|\omega_n| &= \left|1-e^{-2\pi i(\langle c_0,\xi(n)\rangle+\langle d_0,\alpha(n)\rangle-iq)}\right|<C(|\xi(n)|+|\ell(n)|)^{-n}\\
		|\kappa_n| &=  \left| \ e^{2\pi i(\langle c_0,\xi(n)\rangle+\langle d_0,\alpha(n)\rangle-iq)} - 1 \, \right|<C(|\xi(n)|+|\ell(n)|)^{-n}.
	\end{align*}

Choose $\delta > 0$ such that $(\frac{\pi}{2}-\delta,\frac{\pi}{2}+\delta)\subset(0,\pi)$, and let $\phi \in C^\infty_c((\frac{\pi}{2}-\delta,\frac{\pi}{2}+\delta))$ be a function satisfying $0 \leq \phi(t) \leq 1$ and $\phi \equiv 1$ in a neighborhood of $[\frac{\pi-\delta}{2},\frac{\pi+\delta}{2}]$. 

For each $n\in\N$, define $\widehat{g}(\cdot,\xi(n),\ell(n))_{\alpha(n)\alpha(n)}$ to be the $2\pi$-periodic extension of:
$$ 
t\in\T^1 \mapsto \omega_ne^{-(\langle \lambda,\xi(n)\rangle+\langle \gamma,\alpha(n)\rangle)2\pi\tilde{b}_0}\phi(t)e^{i(\pi-t)(\langle a_0,\xi(n)\rangle+\langle e_0,\alpha(n)\rangle)}e^{q(\pi-t)}
$$
when  $(\langle \lambda,\xi(n)\rangle+\langle \gamma,\alpha(n)\rangle)2\pi\tilde{b}_0\geq0$, and 
$$ 
t\in\T^1 \mapsto \kappa_ne^{+(\langle \lambda,\xi(n)\rangle+\langle \gamma,\alpha(n)\rangle)2\pi\tilde{b}_0}\phi(t)e^{i(\pi-t)(\langle a_0,\xi(n)\rangle+\langle e_0,\alpha(n)\rangle)}e^{q(\pi-t)},
$$
otherwise. Moreover,  define $\widehat{g}(t,\xi,\ell)_{\alpha\beta}\equiv 0$ in every other case. 

These choices ensure that $g\in C^{\infty}(\G)\cap(\ker \Lt)^0$ as stated in Propositions \ref{lemmadecaysmoothpartial} and \ref{lemma0Nimpliesanul}. Specifically, the decay properties of $(\omega_n)$ and $(\kappa_n)$ ensure that $g$ decays faster than any polynomial, the fact of the functions $e^{\pm(\langle \lambda,\xi(n)\rangle+\langle \gamma,\alpha(n)\rangle)2\pi\tilde{b}_0}$ to be bounded by $1$ in each case, and the condition $\langle c_0,\xi(n)\rangle+\langle d_0,\alpha(n)\rangle-iq\not\in\Z$, for every $n\in\N$.

Assuming the existence of $u\in C^\infty(\G)$ such that $Lu=g$, we apply Plancherel's Theorem to obtain $\widehat{Lu}(t,\xi(n),\ell(n))_{\alpha(n)\alpha(n)}=\widehat{g}(t,\xi(n),\ell(n))_{\alpha(n)\alpha(n)}$, for every $t\in\T^1$ and $n\in\N$. 
Using Remark \ref{obseqedo}, this leads to the system of differential equations \ref{eqedo}. 

As $\langle c_0,\xi(n)\rangle+\langle d_0,\alpha(n)\rangle-iq\not\in \Z$ for each $n\in\N$, according to Lemma \ref{lemmaodesol}, the coefficients $\widehat{u}(\cdot,\xi(n),\ell(n))_{\alpha(n)\alpha(n)}$ must coincide with the unique solution of these equations, which can be expressed as:	
	\begin{align*}
		\widehat{u}(t,\xi(n),\ell(n))_{\alpha(n)\alpha(n)} = & \int_0^{2\pi}\phi(t-s)e^{(\langle \lambda,\xi(n)\rangle+\langle \gamma,\alpha(n)\rangle)\left(\int_{t-s}^t\tilde{b}(\tau) d\tau-2\pi \tilde{b}_0\right)} \\ 
		& \qquad\qquad \times e^{(\pi-t)i(\langle a_{0},\xi(n)\rangle+\langle e_0,\alpha(n)\rangle)}e^{q(\pi-t)}ds
	\end{align*}
	if $(\langle \lambda,\xi(n)\rangle+\langle \gamma,\alpha(n)\rangle)\tilde{b}_0\geq 0$,  and
	\begin{align*}
		\widehat{u}(t,\xi(n),\ell(n))_{\alpha(n)\alpha(n)} = & \int_0^{2\pi}\phi(t+s)e^{(\langle \lambda,\xi(n)\rangle+\langle \gamma,\alpha(n)\rangle)\left(2\pi \tilde{b}_0-\int_{t}^{t+s}\tilde{b}(\tau) d\tau\right)} \\ 
		& \qquad\qquad \times e^{(\pi-t)i(\langle a_{0},\xi(n)\rangle+\langle e_0,\alpha(n)\rangle)}e^{q(\pi-t)}ds
	\end{align*}
	otherwise.

	Note that the function $$t\in \T^1 \mapsto (\langle \lambda,\xi(n)\rangle+\langle \gamma,\alpha(n)\rangle)\tilde{b}(t)$$ has the same sign as the real number $(\langle \lambda,\xi(n)\rangle+\langle \gamma,\alpha(n)\rangle)\tilde{b}_0$. This is because $\tilde{b}$ does not change sign. Moreover, we have 
	$$\lim_{n\to \infty}\langle \lambda,\xi(n)\rangle+\langle \gamma,\alpha(n)\rangle= 0$$
	which implies that for sufficiently large $n$, we can choose 
	$$\pm\left(\langle \lambda,\xi(n)\rangle+\langle \gamma,\alpha(n)\rangle\right)2\pi\tilde{b}_0>\ln\left(\tfrac{1}{2}\right).$$ 
	
	By selecting a subsequence, we can assume that this holds for every $n\in\N$. Therefore, for each $n\in\N$, if $(\langle \lambda,\xi(n)\rangle+\langle \gamma,\alpha(n)\rangle)\tilde{b}_0\geq 0$, using the previous observations:	
	\begin{align*}
		\widehat{u}(\pi,\xi(n),\ell(n))_{\alpha(n)\alpha(n)} &= \int_0^{2\pi}\phi(\pi-s)e^{(\langle \lambda,\xi(n)\rangle+\langle \gamma,\alpha(n)\rangle)\left(\int_{\pi-s}^\pi\tilde{b}(\tau) d\tau-2\pi \tilde{b}_0\right)}ds\\
		&=e^{-(\langle \lambda,\xi(n)\rangle+\langle \gamma,\alpha(n)\rangle)\tilde2\pi{b}_0}\int_{\pi-\delta}^{\pi+\delta}\phi(\pi-s)e^{(\langle \lambda,\xi(n)\rangle+\langle \gamma,\alpha(n)\rangle)\int_{\pi-s}^\pi\tilde{b}(\tau) d\tau}ds.
	\end{align*}

	Hence
	\begin{align*}
		|\widehat{u}(\pi,\xi(n),\ell(n))_{\alpha(n)\alpha(n)}|\geq \delta e^{- (\langle \lambda,\xi(n)\rangle+\langle \gamma,\alpha(n)\rangle)2\pi\tilde{b}_0}\geq \frac{\delta}{2}> \frac{\delta}{2}(|\xi(n)|+|\ell(n)|)^{-1}.
	\end{align*}

	Similarly, if $(\langle \lambda,\xi(n)\rangle+\langle \gamma,\alpha(n)\rangle)\tilde{b}_0< 0$, then
	\begin{align*}
		|\widehat{u}(\pi,\xi(n),\ell(n))_{\alpha(n)\alpha(n)}| &= \int_0^{2\pi}\phi(\pi+s)e^{(\langle \lambda,\xi(n)\rangle+\langle \gamma,\alpha(n)\rangle)\left(2\pi \tilde{b}_0-\int_{\pi}^{\pi+s}\tilde{b}(\tau) d\tau\right)}ds\\
		&=e^{+ (\langle \lambda,\xi(n)\rangle+\langle \gamma,\alpha(n)\rangle)2\pi\tilde{b}_0}\int_{\pi-\delta}^{\pi+\delta}\phi(\pi+s)e^{-(\langle \lambda,\xi(n)\rangle+\langle \gamma,\alpha(n)\rangle)\int_{\pi}^{\pi+s}\tilde{b}(\tau) d\tau}ds\\
		&\geq\frac{\delta}{2}>\frac{\delta}{2}(|\xi(n)|+|\ell(n)|)^{-1}.
	\end{align*}

Thus, from the previous observations and  Proposition \ref{lemmadecaysmoothpartial}, we can conclude that $u\not\in C^{\infty}(\G)$. Therefore, $L$ is not globally solvable.

Now, we proceed with the proof that $L$ is not globally hypoelliptic. We will continue using the notations and definitions introduced earlier for the function $g$ and the solutions $\widehat{u}(\cdot,\xi(n),\ell(n))\in C^\infty(\T^1)$. Additionally, we define $\widehat{u}(\cdot,\xi,\ell)_{\alpha\beta}\equiv0$ for all other Fourier coefficients.

Based on the previous estimates, these coefficients do not define a function $u\in C^\infty(\G)$. Let us show that these coefficients define a distribution $u\in \mathscr{D}'(\G)$.

As mentioned before, if $(\langle \lambda,\xi(n)\rangle+\langle \gamma,\alpha(n)\rangle)\tilde{b}_0\geq 0$, since $\tilde{b}$ does not change sign, it follows that $(\langle \lambda,\xi(n)\rangle+\langle \gamma,\alpha(n)\rangle)\tilde{b}(t)\geq 0$ for all $t\in\T^1$. Therefore, for $s\in[0,2\pi]$, we have:
	\begin{align*}
		(\langle \lambda,\xi(n)\rangle+\langle \gamma,\alpha(n)\rangle)\int_{t-s}^t\tilde{b}(\tau) d\tau&\leq(\langle \lambda,\xi(n)\rangle+\langle \gamma,\alpha(n)\rangle)\int_{0}^{2\pi}\tilde{b}(\tau) d\tau\\
		&=(\langle \lambda,\xi(n)\rangle+\langle \gamma,\alpha(n)\rangle)2\pi\tilde{b}_0.
	\end{align*}
	Thus, we have \ 
		$e^{(\langle \lambda,\xi(n)\rangle+\langle \gamma,\alpha(n)\rangle)\left(\int_{t-s}^{t}\tilde{b}(\tau) d\tau-2\pi \tilde{b}_0\right)}\leq 1,$ for all $n\in\N.$

Similarly, if $(\langle \lambda,\xi(n)\rangle+\langle \gamma,\alpha(n)\rangle)\tilde{b}_0< 0$, then $(\langle \lambda,\xi(n)\rangle+\langle \gamma,\alpha(n)\rangle)\tilde{b}(t)\leq 0$ for all $t\in\T^1$. Thus, for $s\in[0,2\pi]$, we have:
	\begin{align*}
		(\langle \lambda,\xi(n)\rangle+\langle \gamma,\alpha(n)\rangle)\int_{t}^{t+s}\tilde{b}(\tau) d\tau&\geq(\langle \lambda,\xi(n)\rangle+\langle \gamma,\alpha(n)\rangle)\int_{0}^{2\pi}\tilde{b}(\tau) d\tau\\
		&=(\langle \lambda,\xi(n)\rangle+\langle \gamma,\alpha(n)\rangle)2\pi\tilde{b}_0.
	\end{align*}
Therefore, we obtain $e^{(\langle \lambda,\xi(n)\rangle+\langle \gamma,\alpha(n)\rangle)\left(2\pi \tilde{b}0-\int{t}^{t+s}\tilde{b}(\tau) d\tau\right)}\leq 1$. 

Hence, in either case, we have:
	$$ |\widehat{u}(t,\xi(n),\ell(n))_{\alpha(n)\alpha(n)}|\leq 2\pi e^{\pi|\Re(q)|}.$$

It follows from Proposition \ref{lemmadecaysmoothpartial} that these coefficients define $u\in \mathscr{D}'(\G)\backslash C^\infty(\G)$. Moreover, it is clear that $Lu=g\in C^\infty(\G)$. Therefore, $L$ is not globally hypoelliptic.
\end{proof}

\begin{corollary}
	Suppose $L_0$ is not globally hypoelliptic. Then $L$ is not globally hypoelliptic.
\end{corollary}
\begin{proof}
	By Proposition \ref{propghcte}, we have two possible cases:
	\begin{enumerate}[(i)]
		\item $L_0$ is not globally solvable;
		\item The set 
		$$\mathcal{Z}_{L_0} =  \left\{(\tau,\xi,\ell) \in \Z^{r+1}\times\mbox{$\frac{1}{2}$} \N_0^s; \sigma_{L}(\tau,\xi,\alpha)= 0,\,\text{for some} -\ell\leq\alpha\leq \ell,\,\ell-\alpha\in\N_0^s\right\}.$$
		is infinite.
	\end{enumerate}
	
	If case (i) holds true, then $L$ is not globally hypoelliptic according to Proposition \ref{propl0notgsgh}.
	
	Now, let us consider case (ii). In this situation, there exists a sequence $(\xi(n),\alpha(n))_{n\in\N}$ in $ \Z^{r} \times \frac{1}{2}\Z^s$ consisting of distinct terms such that $\langle c_0,\xi(n)\rangle + \langle d_0,\alpha(n)\rangle - iq \in \Z$ for every $n\in\N$. Specifically, we have $(\langle \lambda,\xi(n)\rangle + \langle \gamma,\alpha(n)\rangle) = \frac{\Re(q)}{\tilde{b}_0}$ for each $n$.
	
	For each $n\in\N$, $t\in [0,2\pi]$, set $\ell(n) = (|\alpha(n)_1|, \dots,|\alpha(n)_s|) \in\frac{1}{2}\N_0^s$, and define
	$$\widehat{u}(t,\xi(n),\ell(n))_{\alpha(n)\alpha(n)} = e^{-\int_0^ti(\langle c(\tau),\xi(n)\rangle+\langle d(\tau),\alpha(n)\rangle-iq)d\tau}$$
	and $\widehat{u}(t,\xi',\ell')_{\alpha\beta} \equiv0$ otherwise.

These functions are well-defined in $\T^1$. Moreover, we have:
	\begin{align*}
		|\widehat{u}(t,\xi(n),\ell(n))_{\alpha(n)\alpha(n)}| =  e^{-\int_0^t\Re(q)\left(1-\frac{\tilde{b}(\tau)}{\tilde{b}_0}\right)d\tau}\leq \max_{s\in[0,2\pi]} e^{-\int_0^s\Re(q)\left(1-\frac{\tilde{b}(\tau)}{\tilde{b}_0}\right)d\tau}\doteq K >0,
	\end{align*}
for $t\in\T^1$ and $n\in\N$, where we observe that $K > 0$.

Furthermore, since $|\widehat{u}(0,\xi(n),\ell(n))_{\alpha(n)\alpha(n)}|=1$, for all $n\in \N$, by Proposition \ref{lemmadecaysmoothpartial}  we have that these coefficients define $u\in \mathscr{D}'(\G)\backslash C^\infty(\G)$. Moreover, it is evident that $Lu=0\in C^\infty(\G)$, as we can observe from the equation:
	\begin{align*}
		\widehat{Lu}(t,\xi(n),\ell(n))_{\alpha(n)\alpha(n)}&=[\partial_t+i(\langle c(t),\xi(n)\rangle+\langle d(t),\alpha(n)\rangle-iq)]\widehat{u}(t,\xi(n),\ell(n))_{\alpha(n)\alpha(n)}\\
		&=-i(\langle c(t),\xi(n)\rangle+\langle d(t),\alpha(n)\rangle-iq)\widehat{u}(t,\xi(n),\ell(n))_{\alpha(n)\alpha(n)}\\
		&+i(\langle c(t),\xi(n)\rangle+\langle d(t),\alpha(n)\rangle-iq)\widehat{u}(t,\xi(n),\ell(n))_{\alpha(n)\alpha(n)} = 0
	\end{align*}
	for all $n\in\N$. Therefore $L$ is not globally hypoelliptic.
\end{proof}

With the aforementioned results, we have successfully demonstrated the necessity of the conditions stated in Theorems  \ref{teogsmain}, \ref{teoghmain} and \ref{teocs}. Thus, we have completed the proof of their necessity.

\subsection{Sufficient Conditions} \

In this subsection, we focus on establishing the sufficiency of the conditions stated in Theorems \ref{teogsmain}, \ref{teoghmain}, and \ref{teocs} for the operator $L$ to be globally solvable and globally hypoelliptic. To that end, we begin by proving the following lemma, which will play a crucial role in demonstrating the existence of solutions to the equation $Lu = g$.

\begin{lemma}\label{lemmadecay}
	Let $\widehat{g}(\cdot,\xi,\ell)_{\alpha\beta}\in C^\infty(\T^1)$ be the Fourier coefficients of a smooth function $g$ in $C^\infty(\G)$, and \ $\widehat{u}(\cdot,\xi,\ell)_{\alpha\beta}\in C^\infty(\T^1)$ be a sequence of functions satisfying equations \eqref{eqedo}.
	Suppose there exist $K', M > 0$ such that
	\begin{equation}\label{eqlemmadecayineq}
		|\widehat{u}(t,\xi,\ell)_{\alpha\beta}| \leq K'(|\xi| + |\ell|)^{M}\|\widehat{g}(\cdot,\xi,\ell)_{\alpha\beta}\|_{\infty}
	\end{equation}
	for all $t\in\T^1$ and $(\xi,\ell)\in\Z^r\times\frac{1}{2}\Z^s$ with $-\ell\leq\alpha,\beta\leq \ell$ and $\ell-\alpha,\ell-\beta\in\N_0^s$ (not all equal to zero). Then these coefficients define $u\in C^\infty(\G)$.
\end{lemma}

\begin{proof}
By Proposition \ref{lemmadecaysmoothpartial}, it suffices to prove that for all $m\in\mathbb{N}_0$ and $N>0$, there exists $K'_{mN}>0$ such that
$$ |\partial_t^m\widehat{u}(t,\xi,\ell)_{\alpha\beta}|\leq K'_{mN}(|\xi|+|\ell|)^{-N}.$$
	
	We will prove this by induction on $m$. For $m=0$, by Proposition \ref{lemmadecaysmoothpartial}, for each $N>0$, there exists $K_{0N}>0$ such that
	$$ |\widehat{g}(t,\xi,\ell)_{\alpha\beta}|\leq K_{0N}(|\xi|+|\ell|)^{-(N+M)}$$
	for all $t,\xi,\ell,\alpha,\beta$ as described above (with the condition not all null). Combining this inequality with \eqref{eqlemmadecayineq}, we obtain the desired result.

	Now, assume that the claim holds for every $m\leq k$. Let us prove it for $m=k+1$.
	\begin{align*}
		|\partial_t^{k+1}\widehat{u}(t,\xi,\ell)_{\alpha\beta}| = & \  |\partial_t^k(\partial_t\widehat{u}(t,\xi,\ell)_{\alpha\beta})| = \\[2mm]
		= & \ |\partial^k_t\left(-i(\langle c(t),\xi\rangle+\langle d(t),\alpha\rangle-iq)\widehat{u}(t,\xi,\ell)_{\alpha\beta}+\widehat{g}(t,\xi,\ell)_{\alpha\beta}\right)|\\[2mm]
		\leq &\left|\sum_{m=0}^k{k\choose m}\left[(\langle\partial_t^mc(t),\xi\rangle+\langle \partial_t^m d(t),\alpha\rangle)\partial_t^{k-m}\widehat{u}(t,\xi,\ell)_{\alpha\beta}\right]\right| \\[2mm]
		&+|q\partial_t^{k}\widehat{u}(t,\xi,\ell)_{\alpha\beta}|+|\partial^k_t\widehat{g}(t,\xi,\ell)_{\alpha\beta}|\\
		\leq & \ M\max_{m\leq k}\left\{\max_{j_1=1,\dots,r}\|\partial^m_tc_{j_1}\|_\infty,\max_{j_2=1,\dots,s}\|\partial^m_td_{j_2}\|_\infty\right\}(|\xi|+|\ell|)^1 \\[2mm]
		&\times \max_{m\leq k}|\partial_t^{m}\widehat{u}(t,\xi,\ell)_{\alpha\beta}|
		+|q||\partial_t^{k}\widehat{u}(t,\xi,\ell)_{\alpha\beta}|+|\partial^k_t\widehat{g}(t,\xi,\ell)_{\alpha\beta}|,
	\end{align*}
	for some $M>0$. Thus, by the inductive hypothesis and Proposition \ref{lemmadecaysmoothpartial}, the inequality holds as claimed.
\end{proof}

We now proceed with the proof of sufficiency, which is divided into the following three propositions. We first examine the case where $b_0$ and $f_0$ are zero vectors:

\begin{prop}\label{propb0f0connectedgs}
	Assume that:
	\begin{enumerate}[i.]
		\item some $b_j$ or $f_k$ is non-null;
		\item $(b_0,f_0)=\0$; 
		\item $(a_0,e_0,iq)\in\mathbb{Z}^r\times2\mathbb{Z}^s\times \mathbb{Z}$; and 
		\item All sublevel sets
		$$\Omega_m^{\xi,\alpha} = \left\{t\in\T^1|\int_0^t\langle b(\tau),\xi\rangle+\langle f(\tau),\alpha\rangle d\tau<m\right\} \mbox{ are connected.}$$ 
	\end{enumerate}
	 Then $L$ is globally solvable.
\end{prop}
\begin{proof}
	Based on assumptions ii. and iii., it follows that
	$$\langle c_0,\xi\rangle+\langle d_0,\alpha\rangle-iq = \langle a_0,\xi\rangle+\langle e_0,\alpha\rangle-iq\in\Z$$
	for every $\xi\in\Z^r$, and $\alpha\in\frac{1}{2}\Z^s$. 
		
	Given $g\in (\ker \Lt)^0$, the function $u$ is a solution to $Lu=g$ if and only if it satisfies $\widehat{Lu}(t,\xi,\ell)_{\alpha\beta} = \widehat{g} (t,\xi,\ell)_{\alpha\beta}$ for every $t\in\T^1$, $\xi\in\Z^r$, $\ell\in\frac{1}{2}\N_0^s$, $-\ell\leq\alpha,\beta\leq \ell$, and $\ell-\alpha,\ell-\beta\in\N_0^s$.	
			
	We claim that there exist smooth functions $\widehat{u}(\cdot,\xi,\ell)_{\alpha\beta}\in C^{\infty}(\T^1)$ solutions to the equations $\widehat{Lu}(t,\xi,\ell)_{\alpha\beta} = \widehat{g}(t,\xi,\ell)_{\alpha,\beta}$, for all $\xi,\ell,\alpha,\beta$, such that
	$$ |\widehat{u}(t,\xi,\ell)_{\alpha\beta}|\leq 2\pi ||\widehat{g}(\cdot,\xi,\ell)_{\alpha\beta}||_{\infty}$$
	for every $t\in\T^1$.
	
	Let $t_{\xi\alpha} \in \T^1$ be the point such that 
	\begin{align*}
		\int_0^{t_{\xi\alpha}}\langle b(\tau),\xi\rangle+\langle f(\tau),\alpha\rangle d\tau = \sup_{t\in\T^1}\left\{\int_0^{t}\langle b(\tau),\xi\rangle+\langle f(\tau),\alpha\rangle d\tau\right\}.
	\end{align*}
	
	By Remark \ref{obseqedo} and Lemma \ref{lemmaanulimpliessolution}, the equations $\widehat{Lu}(t,\xi,\ell)_{\alpha\beta} = \widehat{g}(t,\xi,\ell)_{\alpha\beta}$ admit solutions. According to Lemma \ref{lemmaodesol}, the functions given by
	$$\widehat{u}(t,\xi,\ell)_{\alpha\beta} = \int_{t_{\xi\alpha}}^t\exp\left\{-i\int_s^t\left(\langle c(\tau),\xi\rangle+\langle d(\tau),\alpha\rangle-iq\right)d\tau\right\}\widehat{g}(s,\xi,\ell)_{\alpha\beta} ds$$
	define a solution for each $\xi,\ell,\alpha,\beta$. 
		
	Let us now prove that these solutions satisfy the inequality stated in the claim. Fix $\xi,\ell,\alpha,\beta$ and consider $t\in\T^1$. We define
	$$ m_t\doteq\int_0^t\langle b(\tau),\xi\rangle+\langle f(\tau),\alpha\rangle d\tau.$$
	Note that $t$ and $t_{\xi\alpha}$ belong to the sublevel set $$\tilde{\Omega}_m^{\xi,\alpha} = \left\{s\in\mathbb{T}^1; \int_0^s\langle b(\tau), \xi\rangle+\langle f(\tau),\alpha\rangle d\tau\geq m\right\}.$$
	
	By Lemma \ref{lemmasublevel}, this set is connected. Therefore, there is a path (a circumference arc)
	$\gamma_t \subset \tilde{\Omega}_{m_t}^{\xi,\alpha}$ that connects $t_{\xi\alpha}$ to $t$. 
	
	Therefore, for all $s\in\gamma_t$, we have
$$\int_0^s\langle b(\tau),\xi\rangle+\langle f(\tau),\alpha\rangle d\tau \geq m_t.$$ 

It is important to note that for any $t\in\T^1$, there are exactly two paths on the circle that connect $t_{\xi\alpha}$ to $t$ (one of which is $\gamma_t$), and their union covers the entire $\T^1$. Since $g\in(\ker\Lt)^0$, as in the proof of Lemma \ref{lemmaanulimpliessolution}, we have:
	\begin{align*}
	\int_0^{2\pi} \exp\left\{-i\int_s^0\left(\langle c(\tau),\xi\rangle+\langle d(\tau),\alpha\rangle-iq\right)d\tau\right\}\widehat{g}(s,\xi,\ell)_{\alpha\beta}ds=0.
	\end{align*}
	Thus, we obtain 
	$$\int_0^{2\pi} \exp\left\{-i\int_s^t\left(\langle c(\tau),\xi\rangle+\langle d(\tau),\alpha\rangle-iq\right)d\tau\right\}\widehat{g}(s,\xi,\ell)_{\alpha\beta}ds=0$$
	and the integral from $t_{\xi\alpha}$ to $t$ in the definition of $\widehat{u}$ is independent on the choice of path connecting these points. 
	
	Therefore by choosing $\gamma_t$, we have
	\begin{align*}
		\widehat{u}(t,\xi,\ell)_{\alpha\beta}  = & \int_{\gamma_t}\exp\left\{-i\int_s^t\left(\langle c(\tau),\xi\rangle+\langle d(\tau),\alpha\rangle-iq\right)d\tau\right\}\widehat{g}(s,\xi,\ell)_{\alpha\beta} ds\\
		= &\int_{\gamma_t}\exp\left\{\left(\int_0^t\langle b(\tau),\xi\rangle+\langle f(\tau),\alpha\rangle d\tau\right)-\left(\int_0^s\langle b(\tau),\xi\rangle+\langle f(\tau),\alpha\rangle d\tau\right)\right\}\\[2mm]
		& \qquad \times\exp\left\{i(s-t)(\langle a_0,\xi\rangle+\langle e_0,\alpha\rangle+\Im(q))\right\}\widehat{g}(s,\xi,\ell)_{\alpha\beta} \mathop{ds}.
	\end{align*}
	Hence
	\begin{align*}
		|\widehat{u}(t,\xi,\ell)_{\alpha\beta} | \leq \int_{\gamma_t}|\widehat{g}(s,\xi,\ell)_{\alpha\beta}| ds \leq 2\pi||\widehat{g}(\cdot,\xi,\ell)_{\alpha\beta}||_{\infty}.
	\end{align*} 
Since $t$ was arbitrarily chosen, the inequality holds for every $t\in \T^1$. Furthermore, the choice of $\xi, \ell, \alpha, \beta$ was also arbitrary. Therefore, the claim holds for all choices. Finally, according to Lemma \ref{lemmadecay}, we can conclude that these coefficients define a function $u\in C^\infty(\G)$. It is also clear that $Lu=g$, which implies that $L$ is globally solvable.
\end{proof}

The case where $(b_0,f_0)\neq\0$ can be proven using the following two propositions.

\begin{prop}
		Assume that:
	\begin{enumerate}[i.]
		\item $(b_0,f_0)\neq\0$; 
		\item  $\dim\spn\left\{b_1,\dots,b_r,f_1,\dots,f_s\right\}=1;$
		\item all $b_j$ and $f_k$ does not change sign;
		\item  $L_0$ is globally solvable.
	\end{enumerate}
	Then $L$ is globally solvable.
	\end{prop}
\begin{proof}
	If $b$ and $f$ satisfy the conditions {\it i., ii.} and {\it iii.}, we can find non-null vectors $\lambda \in \mathbb{R}^r$ and $\gamma \in \mathbb{R}^s$, along with a smooth non-null function $\tilde{b}:\T^1\to\R$, such that $$b=(\lambda_1\tilde{b},\dots,\lambda_r\tilde{b}) \mbox{ \ and \ } f=(\gamma_1\tilde{b},\dots,\gamma_s\tilde{b}).$$ 
	Additionally, $\tilde{b}$ does not change sign, which implies $$\tilde{b}_0\doteq\frac{1}{2\pi}\int_{0}^{2\pi}\tilde{b}(\tau)\mathop{d\tau}\neq 0.$$
	
	Let $g\in(\ker \Lt)^0\cap C^\infty(\G)$. Our goal is to construct $u\in C^{\infty}(\G)$ such that $Lu=g$ by mean of its partial Fourier coefficients.
	
	First, observe that if $(\xi,\alpha)\in\Z^r\times\frac{1}{2}\Z^s$ satisfies $\langle c_0,\xi\rangle+\langle d_0,\alpha\rangle-iq\in\Z$, then its imaginary part is zero and we have
	$$0= (\langle\tilde{b}_0\lambda,\xi\rangle + \langle\tilde{b}_0\gamma, \alpha\rangle) - \Re(q) = \Re(q)\left(\frac{\tilde{b}(\tau)}{\tilde{b}_0}-1\right). $$

	Since $g\in (\ker \Lt)^0$, according to Lemma \ref{lemmaanulimpliessolution}, the equation
	\begin{equation}\label{eqode}
		\widehat{Lu}(t,\xi,\ell)_{\alpha\beta} = \widehat{g}(t,\xi,\ell)_{\alpha\beta},\ t\in\T^1   
	\end{equation}
	has infinitely many solutions. We define $\widehat{u}(\cdot,\xi,\ell)_{\alpha\beta}\in C^\infty(\T^1)$ to be the solution given by:
	$$\widehat{u}(t,\xi,\ell)_{\alpha\beta} = \int_{0}^t\exp\left\{-i\int_s^t\left(\langle c(\tau),\xi\rangle+\langle d(\tau),\alpha\rangle-iq\right)d\tau\right\}\widehat{g}(s,\xi,\ell)_{\alpha\beta} ds$$
	Then
	\begin{align*}
		|\widehat{u}(t,\xi,\ell)_{\alpha\beta} |&\leq\int_0^{t}\exp\left\{\int_s^t\Re(q)\left(\frac{\tilde{b}(\tau)}{\tilde{b}_0}-1\right)\mathop{d\tau} \right\}ds\|\widehat{g}(\cdot,\xi,\ell)_{\alpha\beta}\|_{\infty} \\[2mm]
		&\leq K_1\|\widehat{g}(\cdot,\xi,\ell)_{\alpha\beta}\|_{\infty} 
	\end{align*}
	
	Where $K_1$ is a constant that does not depend on $\xi,\ell,\alpha,\beta$.
	
	Now, if $(\xi,\alpha)\in\Z^r\times\frac{1}{2}\Z^s$ and $\langle c_0,\xi\rangle+\langle d_0,\alpha\rangle-iq\not\in\Z$, according to Lemma \ref{lemmaodesol}, the differential equation (\ref{eqode}) has a unique solution $\widehat{u}(\cdot,\xi,\ell)_{\alpha\beta}\in C^{\infty}(\T^1)$ given by two equivalent formulas, which we can choose conveniently.

	Set $\omega_{\xi\alpha} \doteq \langle c_0,\xi\rangle+\langle d_0,\alpha\rangle-iq$ and consider the two possible cases separately:
	
	\noindent Case 1: $(\langle \lambda,\xi\rangle+\langle\gamma,\alpha\rangle)\tilde{b}_0\leq0$ \\
	Since $\tilde{b}$ does not change sign, we have $(\langle\lambda,\xi\rangle+\langle\gamma,\alpha\rangle)\tilde{b}(t)\leq 0$ for all $t\in\T^1$, or  $\langle\lambda,\xi\rangle+\langle\gamma,\alpha\rangle =0$. By Lemma \ref{lemmaodesol}, the solution can be written as:
	\begin{align*}
		\widehat{u}(t,\xi,\ell)_{\alpha\beta}&=\left(1-e^{-2\pi i\omega_{\xi\alpha}}\right)^{-1}\int_0^{2\pi}\widehat{g}(t-s,\xi,\ell)_{\alpha\beta}e^{(\langle\lambda,\xi\rangle+\langle\gamma,\alpha\rangle)\int_{t-s}^t\tilde{b}(\tau)d\tau} e^{-is(\langle a_0,\xi\rangle+\langle e_0,\alpha\rangle)}e^{-qs}ds.
	\end{align*}
	
		\noindent Case 2: $(\langle \lambda,\xi\rangle+\langle\gamma,\alpha\rangle)\tilde{b}_0>0$\\
		Since $\tilde{b}$ does not change sign, we have $(\langle\lambda,\xi\rangle+\langle\gamma,\alpha\rangle)\tilde{b}(t)\geq 0$ for all $t\in\T^1$. By Lemma \ref{lemmaodesol}, the solution can be written as:
	\begin{align*}
		\widehat{u}(t,\xi,\ell)_{\alpha\beta}&=\left(e^{2\pi i\omega_{\xi\alpha}}-1\right)^{-1}\int_0^{2\pi}\widehat{g}(t+s,\xi,\ell)_{\alpha\beta}e^{-(\langle\lambda,\xi\rangle+\langle\gamma,\alpha\rangle)\int_{t}^{t+s}\tilde{b}(\tau)d\tau}e^{is(\langle a_0,\xi\rangle+\langle e_0,\alpha\rangle)}e^{qs}ds.
	\end{align*}
	
	Now, the global solvability of $L_0$ implies that the condition  \eqref{DC_condition_const_coef_section} is satisfied. By applying Proposition \ref{propctegs} and Lemma \ref{lemmadiofequi}, we obtain $K_2>0$ and $N>0$ such that:
	$$\left|1-e^{-2\pi i\omega_{\xi\alpha}}\right|\geq K_2\left(|\xi|+|\ell|\right)^{-N}$$
	and
	$$\left|e^{2\pi i\omega_{\xi\alpha}}-1\right|\geq K_2\left(|\xi|+|\ell|\right)^{-N}$$
	for all $\omega_{\xi\alpha}\not\in\Z$, with $\xi,\ell,\alpha$ not all null.

	Therefore, in both cases, $\widehat{u}(t,\xi,\ell)_{\alpha\beta}$ satisfies:
	\begin{align*}
		|\widehat{u}(t,\xi,\ell)_{\alpha\beta}|&\leq e^{2\pi|\Re(q)|}\frac{1}{K_2}\left(|\xi|+|\ell|\right)^{N}2\pi\|\widehat{g}(\cdot,\xi,\ell)_{\alpha\beta}\|_{\infty}.
	\end{align*}

	Setting $K = \max\left\{K_1,e^{2\pi|\Re(q)|}\frac{2\pi}{K_2}\right\}$ it follows that
	\begin{align*}
		|\widehat{u}(t,\xi,\ell)_{\alpha\beta}|&\leq K\left(|\xi|+|\ell|\right)^{N}\|\widehat{g}(\cdot,\xi,\ell)_{\alpha\beta}\|_{\infty}
	\end{align*}
	for all $t\in\T^1$, and $\xi,\ell,\alpha,\beta$ not all null.
	
	Therefore, based on Lemma \ref{lemmadecay}, we can conclude that these coefficients define a function $u\in C^\infty(\G)$. Moreover, it is clear that $Lu=g$, demonstrating that $L$ is globally solvable.

\end{proof}

\begin{prop}
		Assume that:
\begin{enumerate}[i.]
	\item some $b_j$ or $f_k$ is non-null;
	\item all $b_j$ and $f_k$ does not change sign;
	\item  $\dim\spn\left\{b_1,\dots,b_r,f_1,\dots,f_s\right\}=1;$
	\item  $L_0$ is globally hypoelliptic.
\end{enumerate}
Then $L$ is globally hypoelliptic.
\end{prop}
\begin{proof}
	If $b$ and $f$ satisfy conditions {\it i., ii.}, and {\it iii.}, there exist $\lambda \in \mathbb{R}^r$ and $\gamma \in \mathbb{R}^s$, not both zero, as well as a smooth non-null function $\tilde{b}:\T^1\to\R$ that does not change sign, such that $b(t) = \tilde{b}(t)\lambda$ and $f(t)=\tilde{b}(t)\gamma$, for all $t\in\T$.

	Let $u\in \mathscr{D}'(\G)$ be a distribution that satisfies the equation $Lu=g\in C^\infty(\G)$. Hence, the partial Fourier coefficients of $u$ must satisfy $\widehat{Lu}(t,\xi,\ell)_{\alpha\beta} = \widehat{g}(t,\xi,\ell)_{\alpha\beta}$. In other words, the following differential equations hold for every $\xi\in\Z^r$, $\ell\in\frac{1}{2}\N_0^s$, and $-\ell\leq\alpha,\beta\leq \ell$, in the sense of distributions:
	$$\partial_t\widehat{u}(t,\xi,\ell)_{\alpha\beta}+i(\langle c(t),\xi\rangle+\langle d(t)\alpha\rangle-iq)\widehat{u}(t,\xi,\ell)_{\alpha\beta} = \widehat{g}(t,\xi,\ell)_{\alpha\beta}.$$
	
	According to Proposition \ref{propghcte}, since $L_0$ is globally hypoelliptic, the set	$\mathcal{Z}_{L_0}$	is finite. Therefore, in order to study the decay of the Fourier coefficients of $u$, we can focus on $(\xi,\ell)\not \in\mathcal{Z}_{L_0}$. For such $(\xi,\ell)$, by Lemma \ref{lemmaodesol}, the above differential equation admits a unique smooth solution, which can be expressed as
	\begin{align*}
		\widehat{v}(t,\xi,\ell)_{\alpha\beta} = \left(1-e^{-2\pi i\omega_{\xi\alpha}}\right)^{-1}\int_0^{2\pi}\widehat{g}(t-s,\xi,\ell)_{\alpha\beta}e^{(\langle \lambda,\xi\rangle+\langle\gamma,\alpha\rangle)\int_{t-s}^t\tilde{b}(\tau)d\tau}e^{-is(\langle a_0,\xi\rangle+\langle e_0,\alpha\rangle-iq)}ds
	\end{align*}
	when $(\langle \lambda,\xi\rangle+\langle\gamma,\alpha\rangle)\tilde{b}_0<0$, and otherwise as
	\begin{align*}
		\widehat{v}(t,\xi,\ell)_{\alpha\beta} = \left(e^{2\pi i\omega_{\xi\alpha}}-1\right)^{-1}\int_0^{2\pi}\widehat{g}(t+s,\xi,\ell)_{\alpha\beta}e^{-(\langle \lambda,\xi\rangle+\langle\gamma,\alpha\rangle)\int_{t}^{t+s}\tilde{b}(\tau)d\tau}e^{is(\langle a_0,\xi\rangle+\langle e_0,\alpha\rangle-iq)}ds,
	\end{align*}
	where $\omega_{\xi\alpha} \doteq \langle c_0,\xi\rangle+\langle d_0,\alpha\rangle-iq$

	Since $L_0$ is globally hypoelliptic, Proposition \ref{propghcte} implies that it satisfies (\ref{DC_condition_const_coef_section}). Therefore, by Lemma \ref{lemmadiofequi}, there exist $C>0$ and $M>0$ such that
	$$|e^{2\pi i\omega_{\xi\alpha}}-1|\geq C(|\xi|+|\ell|)^{-M}$$
	and 
	$$|1-e^{-2\pi i\omega_{\xi\alpha}}|\geq C(|\xi|+|\ell|)^{-M}$$
	for every $\xi\in\Z^r$, $\ell\in\frac{1}{2}\N_0^s$, and $\alpha\in\frac{1}{2}\Z^r$ satisfying $-\ell\leq\alpha\leq \ell$, and $\ell-\alpha\in\N_0^s$. 

	Additionally,  note that, if $(\langle \lambda,\xi\rangle + \langle\gamma, \alpha\rangle) \tilde{b}_0<0$, then as $\tilde{b}$ does not change sign, we have $$(\langle \lambda,\xi\rangle+\langle\gamma,\alpha\rangle)\tilde{b}(t)\leq 0,$$ for all $t\in\T^1$. 
	
	Similarly, when $(\langle \lambda,\xi\rangle + \langle\gamma, \alpha\rangle) \tilde{b}_0>0$
	$$(\langle \lambda,\xi\rangle+\langle\gamma,\alpha\rangle)\tilde{b}(t)\geq 0$$  for all $t\in\T^1$. 
	
	Finally, if $(\langle \lambda,\xi\rangle+\langle\gamma,\alpha\rangle)\tilde{b}_0=0$, then since $\tilde{b}_0\neq 0$, we have $(\langle \lambda,\xi\rangle+\langle\gamma,\alpha\rangle)=0$, and thus the last inequality is also true.
	
	Therefore, in any case, 
	$$|\widehat{v}(t,\xi,\ell)_{\alpha\beta}|\leq 2\pi\frac{1}{C}(|\xi|+|\ell|)^{M}e^{2\pi|\Re(q)|}\|\widehat{g}(\cdot,\xi,\ell)_{\alpha\beta}\|_{\infty}.$$
	By Lemma \ref{lemmadecay}, these coefficients define a smooth function $v\in C^\infty(\G)$.
	
	Since $u$ and $v$ have the same Fourier coefficients, they are equal. Therefore $u\in C^\infty(\G)$ and we can conclude that $L$ is globally hypoelliptic.
\end{proof}

This completes the proof of Theorem \ref{teocs}. It is worth noting that Theorems \ref{teoghmain} and \ref{teogsmain} are equivalent to Theorem \ref{teocs}.

\begin{obs} 
{\em
In \cite{AFR2022_jam}, the authors define global solvability for a differential operator $L$ as follows: $L$ is $\mathscr{D}'$ globally solvable if, for every $f\in(\ker \Lt)^0$, there exists $u\in \mathscr{D}'$ such that $Lu=f$. In the proof of Proposition \ref{propctegs}, it is shown that if $L$ satisfies the conditions of that proposition, then it is $\mathscr{D}'$ globally solvable if and only if it is globally solvable as defined in this paper.

Furthermore, if $q=0$, then $\Lt=-L$. Therefore, we can apply Proposition 3.3 from \cite{AFR2022_jam}, which establishes the equivalence of these notions in the variable coefficient case as well.
}
\end{obs}

\subsection{Examples} \

To conclude our work, we would like to present some notable examples that have not been explored in the current literature.

\begin{exemp}
	Consider the differential operator $L$ defined as
	$$L = \partial_t+[\cos(t)+1+i\sin(t)]\partial_{x}+[\sin(t)+2+i\cos(t)]D_3+3i$$
	acting on the space $\T^{2}\times\mathbb{S}^3$.
	
	Observe that $(b_0,f_0)=0$, $(a_0,e_0)=(1,2)\in\Z^r \times 2\Z^s$, and $q=3i\in i\Z$.  According to Theorem \ref{teogsmain}, $L$ is globally solvable if and only if the sublevel sets
	$$\Omega_m^{\xi,\alpha}=\left\{t\in\T^1; \int_0^t\xi\sin(\tau)+\alpha\cos(\tau)d\tau <m\right\}$$
	are connected, for every  $m\in\R,\, \xi\in \Z^r$, and $\alpha\in \frac{1}{2}\Z$ .
	
	By standard integration and trigonometry, we can verify that the above set is indeed connected. Thus, we conclude that $L$ is globally solvable.
	
	Now, let us consider the differential operator $L$ given by
	$$L = \partial_t+[\cos(t)+2+i\sin(t)]\partial_{x}+[\sin(t)+1+i\cos(t)]D_3$$
	also acting on $\T^{2}\times\mathbb{S}^3$. Since $e_{0}=1\not\in2\Z$, by Theorem \ref{teogsmain}, we have that $L$ is not globally solvable.
	
	Similarly, if we examine the differential operator
	$$L = \partial_t+[\cos(t)+1+i\sin(t)]\partial_{x}+[\sin(t)+2+i\sin^3(2t)]D_3-i$$
	we observe that $a_0=1\in\Z$, $e_0=2\in 2\Z$, and $q=-i\in i\Z$. However, it is easy to prove  that the sublevel set $\Omega_{1/3}^{0,1}$ is disconnected, and hence $L$ is not globally solvable. Nonetheless, $L_0$ is globally solvable.
	\end{exemp}
\begin{exemp}
    Let $r,s\in\mathbb{N}$, $\lambda\in\R^2$, $q\in\mathbb{C}$, and consider the following differential operator on $\T^{r+1}\times\S3$:
    $$L = \partial_t+\sum_{j=1}^ri(\sin(t)+\lambda_1)\partial_{x_j}+\sum_{k=1}^{s}i(\cos(t)+\lambda_2)D_{3,k}+q,$$
    Then for $0<\max\{|\lambda_1|,|\lambda_2|\}<1$, $L$ is not globally solvable as $b_0$ and $f_0$ are not both null and some $b_j$ or $f_k$ changes sign. On the other hand, if $\min\{|\lambda_1|,|\lambda_2|\}>1$, $\lambda\in\mathbb{Q}^2$, then $L$ and $L_0$ are globally solvable by the same arguments used in Example \ref{exemptoruscte}.

    If now we consider
    $$L = \partial_t+\sum_{k=1}^{s-1}ke^{it}D_{3,k}+(\varepsilon e^{i t}+\mu)D_{3,s}$$
    where $\mu\in\mathbb{C}$, $\varepsilon>0$, we see that if $s>1$, it is globally solvable if and only if $\mu=0$. However, if $s=1$, then it is globally solvable for $\mu\in2\Z$ or $|\Im(\mu)|\geq\varepsilon$ and $\Re(\mu)\in\mathbb{Q}$ or is an irrational non-Liouville number.\\
\end{exemp}

\begin{exemp}
	Consider $L$ as a differential operator on $\T^2\times\mathbb{S}^3$, where
	$$L = \partial_t+[\cos(t)+1+i(\sin(t)+1)]\partial_{x}+[\sin(t)+2+i(\sin(t)+1)]D_3+\sqrt{2}+3i$$
	Then as $\dim\{b,f\}=1$, $b$ and $f$ do not change sign and $L_0$ is globally hypoelliptic, as seen in Example \ref{exemphypocte}, by Theorem \ref{teoghmain}, $L$ is globally hypoelliptic. On the other hand,
	$$L = \partial_t+[\cos(t)+1+i(\sin(t)+1)]\partial_{x}+[\sin(t)+2-i(\sin(t)+1)]D_3+\frac{1}{2}-2i$$
	is not globally hypoelliptic because $L_0$ is not globally hypoelliptic, as $\sigma_L(0,1,\frac{1}{2})=0$.
\end{exemp}

\appendix
\section{Technical results}

In this section we present some of the more technical lemmas used in this paper.
\begin{lemma}\label{lemmapseudo}
	Let $\xi = (\xi_1,\dots,\xi_r)\in\Z^r\approx\widehat{\T^r}$ and $\ell = (\ell_1,\dots,\ell_s)\in\frac{1}{2}\N_0^s\approx\widehat{\mathbb{S}^3}$. Then for each $\xi\otimes \ell\in\widehat{\G}$,
	$$\left\langle\xi\otimes \ell\right\rangle\asymp (|\xi|+|\ell|).$$
\end{lemma}
\begin{proof}
	According to the general theory for compact Lie groups: if $G$ is a compact Lie group and $\mathcal{L}_G$ is the positive Laplacian on $G$ then, for $[\mu]\in\widehat{G}$, the eigenvalues of the operator $(I+\mathcal{L}_G)^{\frac{1}{2}}$ are given by $\langle \mu\rangle$, where the eigenvalues correspond to the eigenfunctions $\mu_{ij}$, which are the coefficient functions of any matrix-valued representative of $[\mu]$. 
	Furthermore, for the direct product of compact Lie groups $G_1\times\dots\times G_n$, we have $\mathcal{L}_{G_1\times\dots\times G_n}=\mathcal{L}_{G_1}+\dots+\mathcal{L}_{G_n}$. 
		
	Therefore, for $\xi = \xi_1\otimes\dots\otimes\xi_r\in\widehat{\T^{r}}\approx\Z^{r}$, and  $\ell=\ell_1\otimes\dots\otimes\ell_s\in\widehat{\mathbb{S}^{3s}}\approx\frac{1}{2}\N_0^s$ we have
	\begin{align*}
		\frac{1}{r+s}\left(\langle\xi_1\rangle+\dots+\langle\xi_r\rangle+\langle \ell_1\rangle+\dots+\langle \ell_s\rangle\right)\leq \left\langle\xi\otimes \ell\right\rangle \leq\langle\xi_1\rangle+\dots+\langle\xi_r\rangle+\langle \ell_1\rangle+\dots+\langle \ell_s\rangle
	\end{align*}
	
	Now, using that $\langle\xi_j\rangle = \sqrt{1+\xi_j^2}$ and $\langle \ell_k\rangle = \sqrt{1+\ell_k(\ell_k+1)}$, we obtain $$(|\xi_1|+\dots+|\xi_r|+\ell_1+\dots+\ell_s)\leq\langle\xi_1\rangle+\dots+\langle\xi_r\rangle+\langle \ell_1\rangle+\dots+\langle \ell_s\rangle.$$
	
	And if $\xi\otimes l\not\equiv 0$, then  
	$$\langle\xi_1\rangle+\dots+\langle\xi_r\rangle+\langle \ell_1\rangle+\dots+\langle \ell_s\rangle\leq(2(r+s)+1)(|\xi_1|+\dots+|\xi_r|+\ell_1+\dots+\ell_s)$$
	since $1\leq2(|\xi_1|+\dots+|\xi_r|+\ell_1+\dots+\ell_s)$, so the claim follows.
\end{proof}

\begin{lemma}\label{lemmaanulker}
	Let $L$ be a differential operator on $\G$ with constant coefficients, with symbol $\sigma_L$, that is, $\sigma_L:\mathbb{Z}^{r+1}\times\frac{1}{2}\Z^s\to\mathbb{C}$ such that  for any $f\in \mathscr{D}'(\G)$:   
	$$\widehat{Lf}(\tau,\xi,\ell)_{\alpha\beta} = \sigma_L(\tau,\xi,\alpha)\widehat{f}(\tau,\xi,\ell)_{\alpha\beta}$$ 
	and 
	$$\widehat{\Lt f}(\tau,\xi,\ell)_{\alpha\beta}=\sigma_L(-\tau,-\xi,-\alpha)\widehat{f}(\tau,\xi,\ell)_{\alpha\beta}$$
	for each $(\tau,\xi)\in\Z^{r+1}, \ell\in\frac{1}{2}\N_0^s, -\ell\leq\alpha,\beta\leq \ell$. Then
	$$(\ker \Lt)^0= \left\{f\in \mathscr{D}'(\G)| \text{ such that } \widehat{f}(\tau,\xi,\ell)_{\alpha\beta} = 0 \text{ if } \sigma_L(\tau,\xi,\alpha)=0\right\}.$$ 
\end{lemma}
\begin{proof}
	Suppose first that $\widehat{f}(\tau,\xi,\ell)_{\alpha\beta} = 0 \text{ if }\sigma_L(\tau,\xi,\alpha)=0$. Let $v\in C^\infty(\G)$ be such that $v\in\ker \Lt$. Then, by Lemma \ref{lemmaformula0}:
	\begin{align*}
		\langle f,v\rangle &= (2\pi)^{r+1}\sum_{(\tau,\xi)\in\Z^{r+1}}\sum_{\ell\in\frac{1}{2}\N_0^s}d_l\sum_{-\ell\leq\alpha,\beta\leq \ell}\widehat{f}(\tau,\xi,\ell)_{\alpha\beta}\widehat{v}(-\tau,-{\xi},\ell)_{(-\alpha)(-\beta)}(-1)^{\sum \beta_j-\alpha_j}.
	\end{align*}
	If $\sigma_L(\tau,\xi,\alpha) = 0$, then $\widehat{f}(\tau,\xi,\ell)_{\alpha\beta} = 0$. If not, then since 
	\begin{align*}
		0 &= \widehat{\Lt v}(-\tau,-{\xi},\ell)_{(-\alpha)(-\beta)} = \sigma_L(\tau,\xi,\alpha)\widehat{v}(-\tau,-{\xi},\ell)_{(-\alpha)(-\beta)}
	\end{align*}
	this implies $\widehat{v}(-\tau,-{\xi},\ell)_{(-\alpha)(-\beta)}=0$, so every term in the sum above is zero and $\langle f,v\rangle=0$, so that $f\in(\ker \Lt)^0$. Now let $f\in(\ker \Lt)^0$ and suppose $\sigma_L(\tau,\xi,\alpha)=0$. For each $-\ell\leq\beta\leq \ell$, $l-\beta\in\N_0^s$, take $v_{\tau,\xi,\ell,\alpha,\beta}\in C^{\infty}(\G)$ given by:
	$$ \widehat{v_{\tau,\xi,\ell,\alpha,\beta}}(-\tau,-\xi,\ell)_{(-\alpha)(-\beta)} = 1,$$
	And $\widehat{v_{\tau,\xi,\ell,\alpha,\beta}}(\tau',\xi
	',\ell')_{\alpha'\beta'}=0$ otherwise.
	Then 
	\begin{align*}
		\widehat{\Lt v_{\tau,\xi,\ell,\alpha,\beta}}(-\tau,-{\xi},\ell)_{(-\alpha)(-\beta)} &= \sigma_L(\tau,\xi,\alpha)\widehat{v_{\tau,\xi,\ell,\alpha,\beta}}(-\tau,-{\xi},\ell)_{(-\alpha)(-\beta)}\\
		&=\sigma_L(\tau,\xi,\alpha)=0,
	\end{align*}
	so $v_{\tau,\xi,\ell,\alpha,\beta}\in\ker \Lt$. Therefore, by Lemma \ref{lemmaformula0}:
	$$0 = \langle f,v_{\tau,\xi,\ell,\alpha,\beta}\rangle = d_l\widehat{f}(\tau,\xi,\ell)_{\alpha\beta}(-1)^{\sum\beta_j-\alpha_j}$$
	so we conclude the other inclusion also holds.
\end{proof}

\begin{lemma}\label{lemmaodesol}
	Let $g,\theta \in C^{\infty}(\T^1)$ and $\theta_0 \doteq \frac{1}{2\pi}\int_0^{2\pi}\theta(t)\mathop{dt}$. If $\theta_0\not\in i\mathbb{Z}$, then the differential equation
	\begin{equation}\label{ode}
		\partial_t u(t)+\theta(t)u(t)=g(t),\,\qquad t\in\T^1
	\end{equation}
	admits unique solution in $C^\infty(\T^1)$ given by:
	\begin{equation}\label{sol-}
		u(t) = \frac{1}{1-e^{-2\pi\theta_0}}\int_0^{2\pi}g(t-s)e^{-\int_{t-s}^t\theta(\tau)d\tau}ds
	\end{equation} 
	or equivalently by:
	\begin{equation}\label{sol+}
		u(t) = \frac{1}{e^{2\pi\theta_0}-1}\int_0^{2\pi}g(t+s)e^{\int_{t}^{t+s}\theta(\tau)d\tau}ds.
	\end{equation}
	If $\theta_0\in i\mathbb{Z}$, then equation \ref{ode} admits infinitely many solutions given by:
	\begin{equation}\label{solz}
		u_\lambda(t) = \lambda e^{-\int_0^t\theta(\tau)d\tau}+\int_0^t g(s)e^{-\int_s^t\theta(\tau)d\tau}ds
	\end{equation}
	for every $\lambda\in\mathbb{R}$, if and only if
	$$ \int_0^{2\pi}g(t)e^{\int_0^t\theta(\tau)d\tau}\mathop{dt}=0.$$
\end{lemma}
\begin{proof}
	Since the functions on the torus may be seen as $2\pi$ periodic functions on $\R$, the proof follows from simple differentiation of the formulas above and applying the periodicity. 
\end{proof}

\begin{lemma}\label{lemmaf0v0}
	Let $\phi\in C^\infty(\T^1)$ be a non-null function, and let $\Phi$ be a function such that $\Phi'=\phi$. Suppose there exists $m\in\R$ such that the sublevel set
	$$\Omega_m = \{t\in\T^1; \Phi(t)<m\}$$ 
	is not connected.
	
	Then, there exists $m_0<m$ such that $\Omega_{m_0}$ has two connected components with disjoint closures. Consequently, we can define functions $g_0,v_0\in C^{\infty}(\T^1)$ such that:
	$$\int_0^{2\pi}g_0(t)\mathop{dt}=0,\ \supp(g_0)\cap\Omega_{m_0}=\varnothing, \  \supp(v_0')\subset\Omega_{m_0}\text{ and } \int_0^{2\pi} g_0(t)v_0(t)\mathop{dt}>0.$$
	
\end{lemma}
\begin{proof}
	Let $C_1\subset \T^1$ be a connected component of $\Omega_m$. Notice that $C_1$ is homeomorphic to an open interval and has two distinct boundary points: $\partial C_1=\{t_1,t_2\}$. Choose $t_3\in C_1$ such that $\Phi(t)<m$. Since $\Omega_m$ is not connected, there exists another connected component $C_2$ of $\Omega_m$ such that $C_1\cap C_2=\emptyset$. Similar to $C_1$, the component $C_2$ is also homeomorphic to an open interval and its boundary is given by two distinct points: $\partial C_2=\{t_4,t_5\}$. Choose $t_6\in C_2$ such that $\Phi(t_6)<m$.

	Now, choose $\epsilon>0$ such that $m_0 \doteq \max\{\Phi(t_3),\Phi(t_6)\}+\epsilon<r$. Since $\Phi(t_1)=m$, by the continuity of $\Phi$, there exists an open set $U_1\subset \T^1$ containing $t_1$ such that $\Phi(t)>m_0$ for each $t\in U_1$. Similarly, we can find an open set $U_2$ containing $t_2$ with the same property.

	Let $I$ and $J$ be the connected components of $\Omega_{m_0}$ that contain $t_3$ and $t_6$, respectively. It is important to note that $U_1$ and $U_2$ are contained in $\T^1\backslash(I\cup J)$. Moreover, $I\subset C_1$ and $J\subset C_2$ are ``separated" by $U_1$ and $U_2$, which implies that their closures do not intersect. In other words, if $x\in \overline{I}\cap\overline{J}$, then there exist sequences $(x_n)_n\subset I$ and $(y_n)_n\subset J$ such that $x_n\to x$ and $y_n\to x$. However, since $x_n\in I\subset C_1$, it follows that $\Phi(x_n)< m_0$ for all $n$, which implies $\Phi(x)\leq m_0<m$. Therefore, we have $x\in C_1$. The same logic applies to $y_n$, $J$, and $C_2$, which leads to $x\in C_1\cap C_2$, which is a contradiction.

    Let us consider the previously defined set as contained in the interval $K = [t_1,t_1+2\pi]\subset \R$. Without loss of generality, we can assume that 
	$$t_1<t_3<t_2\leq t_4<t_6<t_5\leq t_1+2\pi$$
	$$t_3\in I\subset C_1=(t_1,t_2), \quad t_6\in J\subset C_2=(t_4,t_5)$$
	$$U_1 = [t_1,t_1+\epsilon')\cup (t_1+2\pi-\epsilon',t_1+2\pi]$$
	where $0<\epsilon'$ and $t_1+\epsilon'<t_3$ and
	$$U_2 = (t_2-\epsilon'',t_2+\epsilon'')$$
	where $0<\epsilon''$ and $t_3<t_2-\epsilon''<t_2+\epsilon''<t_6$.

	Now, for $j=1,2$, let $g_j\in C_c^\infty(U_j)$ be a bump function such that $\int_0^{2\pi}g_j(t)\mathop{dt}=1$. Set $g_0 = g_2-g_1$, so that $\supp(g_0)\subset U_1\cup U_2$ and so $\supp(g_0)\cap \Omega_{m_0}=\emptyset$. Also, 
	$$\int_0^{2\pi}g_0=\int_{U_2}g_2-\int_{U_1}g_1=1-1=0$$

	Finally, let $\delta>0$ be such that $t_3+\delta\in I$ and $t_6-\delta\in J$. Choose $v_0\in C_c^{\infty}((t_3,t_6))$ such that $v_0\equiv1$ in $[t_3+\delta,t_6-\delta]$. In this case,
	$$\int_{0}^{2\pi}g_0v_0=\int_{U_2}g_2=1>0$$
	and $\supp(v_0')\subset I\cup J\subset\Omega_{m_0}$.
\end{proof}

\begin{lemma}\label{lemmaintegralineq}
	Let $\psi\in C^{\infty}(\T^1)$ be a smooth real function such that $\psi(s)\geq 0$ for all $s$, and let $s_0\in\T^1$ be	 a zero of order greater than one for $\psi$, i.e., $\psi(s_0)=0=\psi'(s_0)$. Then, there exists $M>0$ such that for all $\lambda>0$ sufficiently large and $\delta>0$:
	$$\int_{s_0-\delta}^{s_0+\delta}e^{-\lambda\psi(s)}ds\geq \left(\int_{-\delta}^{\delta}e^{-s^2}ds\right)\lambda^{-1/2}M^{-1/2}.$$
\end{lemma}

\begin{proof}
	Let us consider the Taylor expansion of $\psi$ around $s_0$. For each $s\in (s_0-\delta,s_0+\delta)$, there exists $s'\in (s_0-\delta,s_0+\delta)$ such that
	$$\psi(s) = \frac{\psi''(s')}{2}(s-s_0)^2$$
	Let $\tilde{M}=\sup_{s\in [s_0-\delta,s_0+\delta]}\left|\frac{\psi''(s)}{2}\right|\geq0$. If $\tilde{M}=0$, then $\psi\equiv0$ and the inequality is trivial with $M=1$. Otherwise, let $M=\tilde{M}$ and then for $\lambda M>1$ we have:
	\begin{align*}
		\int_{s_0-\delta}^{s_0+\delta}e^{-\lambda\psi(s)}ds&\geq \int_{s_0-\delta}^{s_0+\delta}e^{-(\sqrt{\lambda M}(s-s_0))^2}ds\geq \frac{1}{\sqrt{\lambda M}}\left(\int_{-\delta\sqrt{\lambda M}}^{\delta\sqrt{\lambda M}}e^{-s^2}ds\right) \\
		&\geq \frac{1}{\sqrt{\lambda M}}\left(\int_{-\delta}^{\delta}e^{-s^2}ds\right).
	\end{align*}
\end{proof}

\begin{lemma}\label{lemmasublevel}
    Let $\phi\in C^\infty(\T^1)$ be such that $\int_0^{2\pi}\phi(t)\mathop{dt} = 0$ and for every $r\in\R$, the set: $\Omega_r = \left\{t\in\T^1|\int_0^t\phi(\tau)d\tau<r\right\}$ is connected. Then so is the set:
	$$\tilde{\Omega}_r = \left\{t\in\T^1|\int_0^t\phi(\tau)d\tau\geq r\right\}=\left\{t\in\T^1|-\int_0^t\phi(\tau)d\tau\leq -r\right\}.$$
\end{lemma}
\begin{proof}
	This follows from $\tilde{\Omega}_r=\T^1\backslash\Omega_r$ and the general fact that any $A\subset\T^1$ is connected if and only if $\T^1\backslash A$ is connected. To see this, let $A$ be connected. Note that the claim $\T^1\backslash A$ is connected is trivially true if $A=\T^1$. Otherwise, then $\T^1\backslash A$ contains at least one point, which, without loss of generality we may assume it is $0=0+2\pi\Z$. If we consider 
 \begin{align*}
  f:(0,2\pi)&\to\T^1\backslash\{0+2\pi\Z\}\\
  x&\mapsto x+2\pi\Z
 \end{align*}
 then $f$ is an homeomorphism. Since $A$ is connected, $I=f^{-1}(A)$ also is. But then $I$ is an interval, so $I^c = (0,2\pi)\backslash I$ is either an interval containing $(0,\epsilon)$ or $(2\pi-\epsilon,2\pi)$ for some $\epsilon>0$ or the disjoint union of two intervals containing $(0,\epsilon)\cup(2\pi-\epsilon,0)$, for some $\epsilon>0$. In both cases, since $\T^1\backslash A = f(I^c)\cup\{0+2\pi\Z\}$ it is clearly connected. Switching the roles of $A$ and $A^c$, the converse also follows.
\end{proof}

\bibliographystyle{plain} 
\bibliography{references} 

\begin{thebibliography}{10}

\bibitem{Arau2019_agag}
Gabriel Ara{\'u}jo.
\newblock Global regularity and solvability of left-invariant differential
  systems on compact {Lie} groups.
\newblock {\em Ann. Global Anal. Geom.}, 56(4):631--665, 2019.

\bibitem{AFR2022_pams}
Gabriel Ara{\'u}jo, Igor~A. Ferra, and Luis~F. Ragognette.
\newblock Global analytic hypoellipticity and solvability of certain operators
  subject to group actions.
\newblock {\em Proc. Am. Math. Soc.}, 150(11):4771--4783, 2022.

\bibitem{AFR2022_jam}
Gabriel Ara{\'u}jo, Igor~A. Ferra, and Luis~F. Ragognette.
\newblock Global solvability and propagation of regularity of sums of squares
  on compact manifolds.
\newblock {\em J. Anal. Math.}, 148(1):85--118, 2022.

\bibitem{BMZ2012_cpde}
A.~P. Bergamasco, G.~A. Mendoza, and S.~Zani.
\newblock On global hypoellipticity.
\newblock {\em Commun. Partial Differ. Equations}, 37(9):1517--1527, 2012.

\bibitem{BCP2004_mc}
Adalberto~P. Bergamasco, Paulo~D. Cordaro, and Gerson Petronilho.
\newblock Global solvability for a class of complex vector fields on the
  two-torus.
\newblock {\em Commun. Partial Differ. Equations}, 29(5-6):785--819, 2004.

\bibitem{BDG2017_jfaa}
Adalberto~P. Bergamasco, Paulo~L. Dattori~da Silva, and Rafael~B. Gonzalez.
\newblock Existence and regularity of periodic solutions to certain first-order
  partial differential equations.
\newblock {\em J. Fourier Anal. Appl.}, 23(1):65--90, 2017.

\bibitem{BDGK2015_jpdo}
Adalberto~P. Bergamasco, Paulo~L. Dattori~da Silva, Rafael~B. Gonzalez, and
  Alexandre Kirilov.
\newblock Global solvability and global hypoellipticity for a class of complex
  vector fields on the 3-torus.
\newblock {\em J. Pseudo-Differ. Oper. Appl.}, 6(3):341--360, 2015.

\bibitem{CC2000_cpde}
Wenyi Chen and M.~Y. Chi.
\newblock Hypoelliptic vector fields and almost periodic motions on the torus
  {{\(T^n\)}}.
\newblock {\em Commun. Partial Differ. Equations}, 25(1-2):337--354, 2000.

\bibitem{Avil2020_jmaa}
Fernando de~{\'A}vila~Silva.
\newblock Global hypoellipticity for a class of periodic {Cauchy} operators.
\newblock {\em J. Math. Anal. Appl.}, 483(2):10, 2020.
\newblock Id/No 123650.

\bibitem{Avil2023_mn}
Fernando de~{\'A}vila~Silva.
\newblock Globally hypoelliptic triangularizable systems of periodic
  pseudo-differential operators.
\newblock {\em Math. Nachr.}, 296:2293--2320, 2023.

\bibitem{AC2022_jfa}
Fernando de~{\'A}vila~Silva and Marco Cappiello.
\newblock Time-periodic {Gelfand}-{Shilov} spaces and global hypoellipticity on
  {{\(\mathbb{T}\times\mathbb{R}^n\)}}.
\newblock {\em J. Funct. Anal.}, 282(9):29, 2022.
\newblock Id/No 109418.

\bibitem{AM2021_ampa}
Fernando de~{\'A}vila~Silva and Cleber de~Medeira.
\newblock Global hypoellipticity for a class of overdetermined systems of
  pseudo-differential operators on the torus.
\newblock {\em Ann. Mat. Pura Appl. (4)}, 200(6):2535--2560, 2021.

\bibitem{AGKM2019_jfaa}
Fernando de~{\'A}vila~Silva, Rafael~Borro Gonzalez, Alexandre Kirilov, and
  Cleber de~Medeira.
\newblock Global hypoellipticity for a class of pseudo-differential operators
  on the torus.
\newblock {\em J. Fourier Anal. Appl.}, 25(4):1717--1758, 2019.

\bibitem{AGK2018_jam}
Fernando de~{\'A}vila~Silva, Todor Gramchev, and Alexandre Kirilov.
\newblock Global hypoellipticity for first-order operators on closed smooth
  manifolds.
\newblock {\em J. Anal. Math.}, 135(2):527--573, 2018.

\bibitem{AK2019_jst}
Fernando De~{\'A}vila~Silva and Alexandre Kirilov.
\newblock Perturbations of globally hypoelliptic operators on closed manifolds.
\newblock {\em J. Spectr. Theory}, 9(3):825--855, 2019.

\bibitem{Mora2022_ampa}
Wagner Augusto~Almeida de~Moraes.
\newblock Regularity of solutions to a {Vekua}-type equation on compact {Lie}
  groups.
\newblock {\em Ann. Mat. Pura Appl. (4)}, 201(1):379--401, 2022.

\bibitem{DGY1996_auf}
D.~Dickinson, T.~Gramchev, and M.~Yoshino.
\newblock First order pseudodifferential operators on the torus: {Normal}
  forms, diophantine phenomena and global hypoellipticity.
\newblock {\em Ann. Univ. Ferrara, Nuova Ser., Sez. VII}, 41:51--64, 1996.

\bibitem{DGY2002_pems}
Detta Dickinson, Todor Gramchev, and Masafumi Yoshino.
\newblock Perturbations of vector fields on tori: {Resonant} normal forms and
  {Diophantine} phenomena.
\newblock {\em Proc. Edinb. Math. Soc., II. Ser.}, 45(3):731--759, 2002.

\bibitem{Forni08_cont-math}
Giovanni Forni.
\newblock On the {Greenfield}-{Wallach} and {Katok} conjectures in dimension
  three.
\newblock In {\em Geometric and probabilistic structures in dynamics. Workshop
  on dynamical systems and related topics in honor of Michael Brin on the
  occasion of his 60th birthday, College Park, MD, USA, March 15--18, 2008},
  pages 197--213. Providence, RI: American Mathematical Society (AMS), 2008.

\bibitem{GW1972_pams}
Stephen~J. Greenfield and Nolan~R. Wallach.
\newblock Global hypoellipticity and {Liouville} numbers.
\newblock {\em Proc. Am. Math. Soc.}, 31:112--114, 1972.

\bibitem{GW1973_tams}
Stephen~J. Greenfield and Nolan~R. Wallach.
\newblock Remarks on global hypoellipticity.
\newblock {\em Trans. Am. Math. Soc.}, 183:153--164, 1973.

\bibitem{HP1998_mc}
A.~Alexandrou Himonas and Gerson Petronilho.
\newblock Global hypoellipticity for sums of squares of vector fields of
  infinite type.
\newblock {\em Mat. Contemp.}, 15:145--155, 1998.

\bibitem{HPC2012_mn}
A.~Alexandrou Himonas, Gerson Petronilho, and L.~A. Carvalho Dos~Santos.
\newblock Global analytic, {Gevrey} and {{\(C^{\infty }\)}} hypoellipticity on
  the 3-torus.
\newblock {\em Math. Nachr.}, 285(2-3):265--282, 2012.

\bibitem{Houn1979_tams}
Jorge Hounie.
\newblock Globally hypoelliptic and globally solvable first order evolution
  equations.
\newblock {\em Trans. Am. Math. Soc.}, 252:233--248, 1979.

\bibitem{KMR2020_bsm}
Alexandre Kirilov, Wagner A.~A. de~Moraes, and Michael Ruzhansky.
\newblock Partial {Fourier} series on compact {Lie} groups.
\newblock {\em Bull. Sci. Math.}, 160:27, 2020.
\newblock Id/No 102853.

\bibitem{KMR2021_jfa}
Alexandre Kirilov, Wagner A.~A. de~Moraes, and Michael Ruzhansky.
\newblock Global hypoellipticity and global solvability for vector fields on
  compact {Lie} groups.
\newblock {\em J. Funct. Anal.}, 280(2):39, 2021.
\newblock Id/No 108806.

\bibitem{KMP2021_jde}
Alexandre Kirilov, Ricardo Paleari, and Wagner A.~A. de~Moraes.
\newblock Global analytic hypoellipticity for a class of evolution operators on
  {{\(\mathbb{T}^1 \times \mathbb{S}^3\)}}.
\newblock {\em J. Differ. Equations}, 296:699--723, 2021.

\bibitem{Petr2011_tams}
G.~Petronilho.
\newblock Global hypoellipticity, global solvability and normal form for a
  class of real vector fields on a torus and application.
\newblock {\em Trans. Am. Math. Soc.}, 363(12):6337--6349, 2011.

\bibitem{RT2007_fourier_su2}
Michael Ruzhansky and Ville Turunen.
\newblock On pseudo-differential operators on group {{\(SU(2)\)}}.
\newblock In {\em New developments in pseudo-differential operators. Selected
  papers of the 6th congress of the International Society for Analysis, its
  Applications and Computation (ISAAC), the ISAAC Group in Pseudo-Differential
  Operators (IGPDO), Middle East Technical University, Ankara, Turkey, August
  13--18, 2007}, pages 307--322. Basel: Birkh{\"a}user, 2007.

\bibitem{RT2007_fourier_torus}
Michael Ruzhansky and Ville Turunen.
\newblock On the {Fourier} analysis of operators on the torus.
\newblock In {\em Modern trends in pseudo-differential operators}, pages
  87--105. Basel: Birkh{\"a}user, 2007.

\bibitem{RT2010_book}
Michael Ruzhansky and Ville Turunen.
\newblock {\em Pseudo-differential operators and symmetries. {Background}
  analysis and advanced topics}, volume~2 of {\em Pseudo-Differ. Oper., Theory
  Appl.}
\newblock Basel: Birkh{\"a}user, 2010.

\bibitem{RT2013_imrn}
Michael Ruzhansky and Ville Turunen.
\newblock Global quantization of pseudo-differential operators on compact {Lie}
  groups, {{\(\mathrm{SU}(2)\)}}, 3-sphere, and homogeneous spaces.
\newblock {\em Int. Math. Res. Not.}, 2013(11):2439--2496, 2013.

\end{thebibliography}

\end{document}